\documentclass[12pt,reqno]{amsart}

\usepackage{geometry} 
\usepackage{graphicx}
\usepackage[latin1]{inputenc}
\usepackage{amsmath,amssymb,wasysym}
\usepackage{amsthm,enumerate, mathdots}
\usepackage{latexsym}
\usepackage{xypic}
\usepackage[all]{xy}
\usepackage{pxfonts}
\usepackage{multirow}
\usepackage{color}
\usepackage{hyperref}
\usepackage{bbold}
\usepackage{aurical}
\usepackage[T1]{fontenc}
\usepackage{leftidx}

\theoremstyle{plain}
\newtheorem{thm}{Theorem}[section]
\newtheorem{prop}[thm]{Proposition}
\newtheorem{lmm}[thm]{Lemma}
\newtheorem{cor}[thm]{Corollary}

\newtheorem{conjec}[thm]{Conjecture}

\theoremstyle{definition}
\newtheorem{rmk}[thm]{Remark}{\rm}
\newtheorem{defi}[thm]{Definition}
\newtheorem{expl}[thm]{Example}{\rm}

\newcommand{\gen}[1]{\langle #1 \rangle}

\newcommand{\defin}{\stackrel{def} =}

\newcommand{\frakx}{\mathfrak{X}}

\newcommand{\F}{\mathbb{F}}
\newcommand{\Z}{\mathbb{Z}}

\newcommand{\N}{\mathbb{N}}

\newcommand{\bb}{\mathcal{B}}

\newcommand{\FF}{\mathcal{F}}
\newcommand{\g}{\mathcal{G}}
\newcommand{\hh}{\mathcal{H}}

\newcommand{\LL}{\mathcal{L}}

\newcommand{\oo}{\mathcal{O}}
\newcommand{\pp}{\mathcal{P}}

\newcommand{\TT}{\mathcal{T}}

\newcommand{\wtt}{\widetilde{\TT}}

\newcommand{\Hom}{\mathrm{Hom}}
\newcommand{\Mor}{\mathrm{Mor}}
\newcommand{\Map}{\mathrm{Map}}

\newcommand{\Aut}{\mathrm{Aut}}
\newcommand{\Inn}{\mathrm{Inn}}
\newcommand{\Out}{\mathrm{Out}}
\newcommand{\Iso}{\mathrm{Iso}}

\newcommand{\Ob}{\mathrm{Ob}}
\newcommand{\Ker}{\mathrm{Ker}}
\newcommand{\Id}{\mathrm{Id}}
\newcommand{\Syl}{\operatorname{Syl}\nolimits}

\newcommand{\triv}{\mathrm{triv}}
\newcommand{\ind}{\mathrm{ind}}

\newcommand{\padic}{\Z^\wedge_p}
\newcommand{\dosadic}{\Z^\wedge_2}    
\newcommand{\prufferp}{\Z/p^{\infty}}      
\newcommand{\prufferdos}{\Z/2^{\infty}}        

\newcommand{\res}{\mathrm{res}}

\newcommand{\typ}{\mathrm{typ}}

\newcommand{\ploc}{(S, \FF, \LL)}

\newcommand{\slocl}{\mathbf{L}}

\newcommand{\Gps}{\mathrm{Gps}}

\usepackage{amssymb}

\newcommand{\Q}{{\mathbb Q}}

\DeclareMathAlphabet\EuR{U}{eur}{m}{n}
\SetMathAlphabet\EuR{bold}{U}{eur}{b}{n}

\newcommand{\Res}{\operatorname{Res}\nolimits}

\newcommand{\Inj}{\operatorname{Inj}\nolimits}

\newcommand{\ev}{\textup{ev}}

\newcommand{\limn}[1]{\operatornamewithlimits{\hbox{$\varprojlim^{#1}$}}}

\newcommand{\autcat}{\mathcal{A}ut}   

\newcommand{\fus}{_{\textup{fus}}}

\let\oldcirc=\circ
\renewcommand{\circ}{\mathchoice
    {\mathbin{\scriptstyle\oldcirc}}{\mathbin{\scriptstyle\oldcirc}}
    {\mathbin{\scriptscriptstyle\oldcirc}}
    {\mathbin{\scriptscriptstyle\oldcirc}}}

\newcommand{\hclim}[1]{\setbox1=\hbox{\rm hocolim}
    \setbox2=\hbox to \wd1{\rightarrowfill} \ht2=0pt \dp2=-1pt
    \mathop{\vtop{\baselineskip=5pt\box1\box2}}
    _{#1}}

\newcommand{\higherlim}[2]{\displaystyle\setbox1=\hbox{\rm lim}
    \setbox2=\hbox to \wd1{\leftarrowfill} \ht2=0pt \dp2=-1pt
    \setbox3=\hbox{$\scriptstyle{#1}$}
    \ifdim\wd1<\wd3
    \mathop{\hphantom{^{#2}}\vtop{\baselineskip=5pt\box1\box2}^{#2}}_{#1}
    \else
    \mathop{\vtop{\baselineskip=5pt\box1\box2}}\limits_{#1}\nolimits^{#2}
    \fi}

\newcommand{\aut}{\operatorname{aut}\nolimits}

\renewcommand{\Im}{\operatorname{Im}\nolimits}

\newcommand{\rk}{\operatorname{rk}\nolimits}
\newcommand{\longleft}[1]{\;{\leftarrow%
\count255=0 \loop \mathrel{\mkern-6mu}%
    \relbar\advance\count255 by1\ifnum\count255<#1\repeat}\;}
\newcommand{\longright}[1]{\;{\count255=0 \loop \relbar\mathrel{\mkern-6mu}%
    \advance\count255 by1\ifnum\count255<#1\repeat\rightarrow}\;}
\newcommand{\Right}[2]{\overset{#2}{\longright#1}}
\newcommand{\RIGHT}[3]{\mathrel{\mathop{\kern0pt\longright#1}
        \limits^{#2}_{#3}}}
\newcommand{\Left}[2]{{\buildrel #2 \over {\longleft#1}}}
\newcommand{\LEFT}[3]{\mathrel{\mathop{\kern0pt\longleft#1}\limits^{#2}_{#3}}
}
\newcommand{\dRIGHT}[3]{\mathrel{%
   \mathop{\vcenter{\baselineskip=0pt\hbox{$\kern0pt\longright#1$}%
   \hbox{$\kern0pt\longright#1$}}}\limits^{#2}_{#3}}}
\newcommand{\LRIGHT}[3]{\mathrel{%
   \mathop{\vcenter{\baselineskip=0pt\hbox{$\kern0pt\longleft#1$}%
   \hbox{$\kern0pt\longright#1$}}}\limits^{#2}_{#3}}}
\newcommand{\RLEFT}[3]{\mathrel{%
   \mathop{\vcenter{\baselineskip=0pt\hbox{$\kern0pt\longright#1$}%
   \hbox{$\kern0pt\longleft#1$}}}\limits^{#2}_{#3}}}
\newcommand{\onto}[1]{\;{\count255=0 \loop \relbar\joinrel
    \advance\count255 by1
    \ifnum\count255<#1 \repeat \twoheadrightarrow}\;}

\newtheoremstyle{slant}{}{}{\slshape}{}{\bfseries}{.}{.5em}{}%
\newtheoremstyle{special}{}{}{\slshape}{}{\bfseries}{.}{.5em}{\thmnote{#3}}

\newtheorem{Th}{Theorem}

\theoremstyle{remark}

\begin{document}

\title[Irreducible $p$-local compact groups I]{Irreducible $p$-local compact groups I. The structure of $p$-local compact groups of rank $1$}
\author{A. Gonz\'{a}lez}
\address{Department of Mathematics\\
Kansas State University\\
66506 Manhattan KS\\
United States of America}
\email{agondem@math.ksu.edu}
\urladdr{}


\begin{abstract}    

Let $p$ be a fixed prime number. The main purpose of this paper is to introduce the notion of \textit{irreducible} $p$-local compact group, which provides a first reduction towards a classification of all $p$-local compact groups. In order to test this idea, in this note we describe all $p$-local compact groups of rank $1$ in terms of this notion.

\end{abstract}

\maketitle

\tableofcontents

Roughly speaking, the concept of $p$-local compact group was introduced by C. Broto, R. Levi and B. Oliver in \cite{BLO3} as a generalization of the notions of compact Lie group and $p$-compact group from a $p$-local point of view, and shares with these notions many properties.

Given a fixed prime $p$, an important open problem regarding $p$-local compact groups is their classification. In this sense one cannot ignore the existing classifications of compact Lie groups and $p$-compact groups, since a good classification of $p$-local compact groups would be one that relates to the objects that we are generalizing.

Notice that the first reduction in the classification of compact Lie groups (respectively $p$-compact groups) is to reduce to connected compact Lie groups (respectively $p$-compact groups). Or, from the point of view of classifying spaces, we reduce to compact Lie groups (respectively $p$-compact groups) whose classifying space is simply connected. However from a topological point of view, this reduction does not make much sense in the context of $p$-local compact groups, which are of a more algebraic nature (there is also a cruder reason why one should not consider connectivity as the first reduction: there exist $p$-local \textit{finite} groups whose classifying space is simply connected).

Luckily, the connected component of a compact Lie group (respectively a $p$-compact group) can also be understood from an algebraic point of view. Indeed the connected component of a compact Lie group (respectively a $p$-compact group) is the minimal normal subgroup of finite index in the original Lie (or $p$-compact) group. And this is a point of view that we can implement in the context of $p$-local compact groups. However, since we are not adopting a topological point of view, we will not talk about \textit{connectivity} any more.

Let us be more specific. As a preview of Section \S \ref{Background}, recall that  a discrete $p$-torus is a group $T$ isomorphic to $(\prufferp)^r$ for some natural number $r \geq 0$, which we call \textit{the rank of $T$}. A \textit{discrete $p$-toral group} is a group $P$ containing a discrete $p$-torus as a normal subgroup of index a finite power of the prime $p$, which we call the maximal torus of $P$. The rank of a discrete $p$-toral group $P$ is the rank of its maximal torus. A $p$-local compact group is a triple $\g = \ploc$, where $S$ is a discrete $p$-toral group, $\FF$ is a saturated fusion system over $S$ (a category encoding fusion data), and $\LL$ is a centric linking system associated to $\FF$ (a category encoding the necessary data to associate a classifying space to $\FF$). The \textit{rank} of $\FF$ (respectively of $\g$) is the rank of the discrete $p$-toral group $S$. The classifying space of $\g$ is the space $B\g = |\LL|^{\wedge}_p$, where $(-)^{\wedge}_p$ denotes Bousfield-Kan's $p$-completion \cite{BK}.

We say that a $p$-local compact group $\g = \ploc$ with maximal torus $T$ is \textit{irreducible} if $\FF$ itself is the only normal fusion subsystem of $\FF$ of maximal rank (normality of fusion subsystems is taken in the sense of Aschbacher \cite{Aschbacher}). Notice that from this point of view, the only irreducible $p$-local finite group (i.e., of rank $0$) is the trivial $p$-local finite group.

Our first result is a complete list of all \textit{irreducible} $p$-local compact groups of rank $1$.

\begin{Th}\label{thmA}

Let $\g = \ploc$ be a irreducible $p$-local compact group of rank $1$. Then,
\begin{enumerate}[(i)]

\item if $p > 2$ then $\g$ is the $p$-local compact group induced by $SO(2)$; and

\item if $p = 2$ then $\g$ is the $p$-local compact group induced by either $SO(2)$, $SO(3)$ or $SU(2)$.

\end{enumerate}

\end{Th}

All the $p$-local compact groups listed above are explicitly described in Section \S \ref{Section2}. Notice that the above list differs from the list of connected $p$-compact groups for odd primes $p$, hence another reason not to talk about connectivity in the $p$-local setting.

Indeed recall that, for an odd prime $p$ and $n \geq 2$ dividing $p-1$, a Sullivan sphere of dimension $2n-1$ is a space $X$ which is equivalent to the $p$-completion of the $(2n-1)$-dimensional sphere, $S^{2n-1}$. For instance, if we set
$$
BX = (BSO(3))^{\wedge}_p \simeq (BSU(2))^{\wedge}_p
$$
then $X = \Omega BX$ is a Sullivan sphere of dimension $3$ for all primes $p > 2$. Sullivan spheres mark the difference between the above list of irreducible $p$-local compact groups of rank $1$, and the classification of $p$-compact groups of rank $1$.

This phenomenon is explained in greater level of detail in Section \S \ref{Section2}, but it can be summarized as follows. If $\g = \ploc$ is the $p$-local compact group induced by a Sullivan sphere, with maximal torus $T \leq S$, then $\FF_T(T) \subseteq \FF$ is obviously normal, and $\FF_T(T) \subsetneqq \FF$. In the terminology of \cite{BCGLO2}, $\FF_T(T)$ is a subsystem of (finite) \textit{index prime to $p$} in $\FF$, hence the difference with the classification of $p$-compact groups. In a separate note we will describe in full detail which connected compact Lie groups and connected $p$-compact groups give rise to irreducible $p$-local compact groups, without any restriction on the rank.

As an interesting consequence of the proof of Theorem \ref{thmA}, we will see how each $p$-local compact group $\g = \ploc$ of rank $1$ uniquely determines a $p$-local compact group $\g_0 = (S_0, \FF_0, \LL_0)$ satisfying the following two properties
\begin{enumerate}[(i)]

\item $\g_0$ is irreducible; and

\item $\FF_0$ is normal in $\FF$.

\end{enumerate}
In this situation it makes sense to call $\g_0$ the \textit{irreducible component} of $\g$.

Together with the list of irreducible $p$-local compact groups of rank $1$, we provide full descriptions of the spaces of self-equivalences of their classifying spaces. This makes the proof of the following result rather straightforward. Recall that a transporter system is a generalization of the notion of centric linking system introduced originally by Oliver and Ventura in \cite{OV}. We review this notion in Section \S \ref{Background}.

\begin{Th}\label{thmB}

Let $(\overline{S}, \overline{\FF}, \overline{\LL})$ be a finite transporter system, where $\overline{\FF}$ is a saturated fusion system over $\overline{S}$ and $\overline{\LL}$ contains all the $\overline{\FF}$-centric $\overline{\FF}$-radical subgroups of $\overline{S}$. If $\tau \colon X \Right2{} |\overline{\LL}|$ is a fibration whose fibre is equivalent to the classifying space of a irreducible $p$-local compact group of rank $1$, then $X^{\wedge}_p$ is the classifying space of a $p$-local compact group (of rank $1$).

\end{Th}

To complete the classification of $p$-local compact groups of rank $1$, we have to prove some converse of the above result. More precisely, let $\g = \ploc$ be a $p$-local compact group of rank $1$, and let $\g_0 = (S_0, \FF_0, \LL_0)$ be its irreducible component. Let also $N_{\g}(S_0) = (S, N_{\FF}(S_0), N_{\LL}(S_0))$ be the normalizer $p$-local compact group of $S_0$ in $\g$ (see Definition \ref{definorm}), and let $(\overline{S}, \overline{\FF}, \overline{\LL})$ be the quotient of $N_{\g}(S_0)$ by $S_0$ (see Definition \ref{defiquotient}), which is a transporter system satisfying the conditions in Theorem \ref{thmB}.

\begin{Th}\label{thmC}

Let $\g = \ploc$ be a $p$-local compact group of rank $1$, and let $\g_0$ and $\overline{\LL}$ be as above. Then, there is a fibration
$$
F \Right4{} X \Right4{} |\overline{\LL}|
$$
such that $X^{\wedge}_p \simeq B\g$ and $F \simeq B\g_0$.

\end{Th}

The paper is organized as follows. The first section contains a quick review on $p$-local compact groups and related notions. The second section contains a full description of the $p$-local compact groups induced by $SO(2)$, $SO(3)$ and $SU(2)$, together with the proof of Theorem \ref{thmA}. The third section studies fibrations over finite transporter systems with fibre the classifying space of a $p$-local compact groups (of rank $1$), and includes the proof of Theorem \ref{thmB}. The fourth section contains the proof of Theorem \ref{thmC}. The fifth section describes an example of simple $p$-local compact group of rank $2$ which is not induced by any compact Lie group or $p$-compact group. The sixth section is devoted to discuss the notion of irreducibility for $p$-local compact groups. We include two appendices at the end of the paper. The first one generalizes  part of the results of \cite{OV} to extensions of transporter systems by discrete $p$-toral groups, while the second one generalizes the Hyperfocal Subgroup Theorem for $p$-local compact groups, as well as some other results of \cite{BCGLO2} regarding the detection of subsystems of $p$-power index of a given fusion system.

The author would especially like to thank Carles Broto, Ran Levi and Bob Oliver for many discussions on the notion of irreducibility, many more useful hints, and their draft \cite{BLOprivate}, source of many interesting ideas. Also Assaf Libman has contributed with many useful conversations and discussions. The author would like to thank also Andy Chermak for his draft \cite{Chermak2}: even if localities are not applied in this paper, his results provided the ultimate evidence that our results could not be but true.

This work is not the result of an intensive period of work, but rather we have spent time on it intermittently. We would like to thank the following universities, where this work has developed during the years: Universitat Autonoma de Barcelona, Aberdeen University, Hebrew University of Jerusalem, Ben Gurion University of the Negev, and Kansas State University.

The author is partially supported by FEDER/MEC grant MTM2010-20692,  by SGR2009-1092,  and also by the grant UNAB10-4E-378 co-funded by FEDER, ``A way to build Europe''.


\section{Background on $p$-local compact groups and related notions}\label{Background}

In this section we review all these notions related to $p$-local compact groups that we need in this paper. The main references for this section are \cite{BLO2}, \cite{BLO3} and \cite{BLO6}.

\begin{defi}\label{defidiscptor}

A \textit{discrete $p$-torus} is a group which is isomorphic to $(\prufferp)^r$ for some finite number $r \geq 0$. A \textit{discrete $p$-toral group} $P$ is a group which is isomorphic to an extension of a finite $p$-group by a discrete $p$-torus.

\end{defi}

In other words, a discrete $p$-toral group $P$ fits in an exact sequence
$$
\{1\} \to P_0 \Right4{} P \Right4{} \pi(P) \to \{1\}
$$
where $P_0$ is a discrete $p$-torus and $\pi$ is a finite $p$-group. The \textit{rank} of $P$, denoted by $\rk(P)$, is $r$, where $P_0 \cong (\prufferp)^r$, and the \textit{order} of $P$ is then defined as the pair $|P| \defin (\rk(P), |\pi(P)|)$, considered as an element of $\N^2$ ordered lexicographically. This way we can compare the order of two discrete $p$-toral groups. For instance, we write $|Q| \leq |P|$ if either $\rk(Q) < \rk(P)$, or $\rk(Q) = \rk(P)$ and $|\pi(Q)| \leq |\pi(P)|$.

\begin{defi}

A \textit{fusion system} over a discrete $p$-toral group $S$ is a category $\FF$ whose object set is the collection of all subgroups of $S$ and whose morphism sets satisfy the following conditions:
\begin{enumerate}[(i)]

\item $\Hom_S(P,Q) \subseteq \Hom_{\FF}(P,Q) \subseteq \Inj(P,Q)$  for all $P, Q \in \Ob(\FF)$; and

\item every morphism in $\FF$ factors as an isomorphism in $\FF$ followed by an inclusion.

\end{enumerate}

\end{defi}

Given a fusion system $\FF$, we say that $P,Q \in \Ob(\FF)$ are \textit{$\FF$-conjugate} if they are isomorphic as objects in $\FF$. The $\FF$-conjugacy class of an object $P$ is denoted by $P^{\FF}$.

\begin{defi}\label{defisat}

Let $\FF$ be a fusion system over a discrete $p$-toral group $S$, and let $P \leq S$.
\begin{itemize}

\item $P$ is \textit{fully $\FF$-centralized} if $|C_S(P)| \geq |C_S(Q)|$ for all $Q \in P^{\FF}$.

\item $P$ is \textit{fully $\FF$-normalized} if $|N_S(P)| \geq |N_S(Q)|$ for all $Q \in P^{\FF}$.

\end{itemize}
The fusion system $\FF$ is \textit{saturated} if the following three conditions hold.
\begin{enumerate}[(I)]

\item For each $P \leq S$ which is fully $\FF$-normalized, $P$ is fully $\FF$-centralized, $\Out_{\FF}(P)$ is finite and $\Out_S(P) \in \Syl_p(\Out_{\FF}(P))$.

\item If $P \leq S$ and $f \in \Hom_{\FF}(P,S)$ are such that $f(P)$ is fully $\FF$-centralized, and if we set
$$
N_f = \{g \in N_S(P) \,\, | \,\, f \circ c_g \circ f^{-1} \in \Aut_S(f(P))\},
$$
there there is $\widetilde{f} \in \Hom_{\FF}(N_f, S)$ such that $\widetilde{f}|_P = f$.

\item If $P_1 \leq P_2 \leq P_3 \leq \ldots$ is an increasing sequence of subgroups of $S$, with $P = \bigcup_{n = 1}^{\infty} P_n$, and if $f \in \Hom(P,S)$ is a homomorphism such that $f|_{P_n} \in \Hom_{\FF}(P_n,S)$ for all $n$, then $f \in \Hom_{\FF}(P,S)$.

\end{enumerate}

\end{defi}

Different sets of axioms for saturation of fusion systems are available in the literature. For instance, in \cite[Corollary 1.8]{BLO6} the authors prove an equivalent set of saturation axioms for a fusion system over a discrete $p$-toral group, which in turn were inspired on a set of axioms due to Roberts and Shpectorov. We also prove here an equivalent set of axioms, this one inspired in \cite[Definition 2.4]{KS}.

\begin{lmm}\label{axiomsKS}

Let $\FF$ be a fusion system over a discrete $p$-toral group $S$. Then, $\FF$ is saturated if and only if it satisfies the following conditions, together with axiom (III) in Definition \ref{defisat}.
\begin{itemize}

\item[(I')] $\Out_{\FF}(S)$ is a finite group, and $\Out_S(S) \in \Syl_p(\Out_{\FF}(S))$.

\item[(II')] If $P \leq S$ and $f \in \Hom_{\FF}(P,S)$ are such that $f(P)$ is fully $\FF$-normalized, and if we set
$$
N_f = \{g \in N_S(P) \,\, | \,\, f \circ c_g \circ f^{-1} \in \Aut_S(f(P))\},
$$
there there is $\widetilde{f} \in \Hom_{\FF}(N_f, S)$ such that $\widetilde{f}|_P = f$.

\end{itemize}

\end{lmm}

\begin{proof}

Clearly, if $\FF$ is saturated then it satisfies (I'), (II') and axiom (III). Let us prove the converse.

\begin{itemize}

\item[\textbf{Step 1.}] $\FF$ satisfies axiom (I).

\end{itemize}

Fix some subgroup $P \leq S$, and let $f \in \Hom_{\FF}(P,S)$ be such that $f(P)$ is fully $\FF$-normalized. First we have to show that $f(P)$ is fully $\FF$-centralized. By (II'), $f$ extends to some $\widetilde{f} \in \Hom_{\FF}(N_f, S)$. Since $C_S(P) \leq N_f$, it follows that $\widetilde{f}(C_S(P)) \leq C_S(f(P))$. Hence $f(P)$ is fully $\FF$-centralized, since this holds for all $Q \in P^{\FF}$.

For simplicity, assume now that $P$ itself is fully $\FF$-normalized. Then we have to show that $\Out_{\FF}(P)$ is finite and $\Out_S(P) \in \Syl_p(\Out_{\FF}(P))$. If $P = S$ then the claim is obvious, so suppose that $P \lneqq S$. Suppose in addition that $P$ is maximal among all (fully $\FF$-normalized) subgroups of $S$ such that $\Aut_S(P)$ is not a Sylow $p$-subgroup of $\Aut_{\FF}(P)$ (this makes implicit use of the ``bullet'' construction in \cite[Section \S 3]{BLO3}, which we intentionally omit to simplify the exposition; details are left to the reader).

Let $H \in \Syl_p(\Aut_{\FF}(P))$ be such that $\Aut_S(P) \lneqq H$. Let also $f \in H \setminus \Aut_S(P)$ be a morphism normalizing $\Aut_S(P)$ (since both $\Aut_S(P)$ and $H$ are discrete $p$-toral groups, such a morphism exists).

It follows then that for each $x \in N_S(P)$, there exists some $y \in N_S(P)$ such that $f(xgx^{-1}) = y f(g) y^{-1}$, for all $g \in P$, and hence $N_f = N_S(P)$, and hence by axiom (II'), $f$ extends to some $\gamma \in \Aut_{\FF}(N_S(P))$. Furthermore, by taking an appropriate power of $\gamma$, we may assume that $\gamma$ has $p$-power order. 

Now, let $f' \in Hom_{\FF}(N_S(P), S)$ be such that $N' = f'(N_S(P))$ is fully $\FF$-normalized. Since $P \lneqq S$, it follows by \cite[Lemma 1.8]{BLO3} that $P \lneqq N_S(P)$, and hence $\Aut_S(N') \in \Syl_p(\Aut_{\FF}(N'))$. In particular, $\gamma ' = f' \gamma (f')^{-1}$ is conjugated in $\Aut_{\FF}(N')$ to an element in $\Aut_S(N')$, and hence we can assume that $f'$ has been chosen such that $\gamma' = c_h \in \Aut_S(N')$, for some $h \in N_S(N')$.

Since $\gamma|_P = f$, the automorphism $\gamma '$ restricts to an automorphism of $f'(P)$, and hence $y \in N_S(f'(P))$. It follows that $f'(N_S(P)) \leq N_S(f'(P))$, and since $P$ is fully $\FF$-normalized, the last inequality is in fact an equality, and
$$
\gamma(g) = (f'(h)) \cdot g \cdot (f'(h))^{-1}
$$
for all $g \in N_S(P)$, and thus $f \in \Aut_S(P)$, in contradiction with the hypothesis of $f \in H \setminus \Aut_S(P)$.

\begin{itemize}

\item[\textbf{Step 2.}] $\FF$ satisfies axiom (II).

\end{itemize}

Let $f \in \Hom_{\FF}(P,S)$ be such that $f(P)$ is fully $\FF$-centralized. We have to show that $f$ extends to some $\widetilde{f} \in \Hom_{\FF}(N_f, S)$.Choose then some $\gamma \in \Hom_{\FF}(Q,S)$ such that $R = \gamma(Q)$ is fully $\FF$-normalized and such that, in the notation of axiom (II'), $N_{\gamma} = N_S(Q)$, and let $f' = \gamma \circ f$.

Since $R$ is fully $\FF$-normalized, by (II') it follows that both $\gamma$ and $f'$ extend to morphisms $\widetilde{\gamma} \in \Hom_{\FF}(N_{\gamma},S)$ and $\widetilde{f}' \in \Hom_{\FF}(N_{f'},S)$ respectively. We want then to see that
\begin{enumerate}[(i)]

\item $N_f \leq N_{f'}$, and

\item $\widetilde{f}'(N_f) \leq \widetilde{\gamma}(N_{\gamma})$.

\end{enumerate}
Were it the case, the composition $(\widetilde{\gamma}^{-1} \circ \widetilde{f'})_{|N_f}$ would then be the extension of $f$ we are looking for.

Let first $g \in N_f$. It follows then that there is some $h \in N_S(Q)$ such that $f c_g f^{-1} = c_h$. Furthermore, since $N_{\gamma} = N_S(Q)$, it follows then that there is some $x \in N_S(R)$ such that
$$
c_x = \gamma \circ c_h \circ \gamma^{-1} = \gamma \circ (f c_g f^{-1}) \circ \gamma^{-1} = (\gamma f) \circ c_g \circ (\gamma f)^{-1} = (f') \circ c_g \circ (f')^{-1}.
$$
Since this holds for all $g \in N_f$, point (i) above follows.

Let now $g \in N_f$, and let $x = \widetilde{f}'(g)$. Let also $h \in N_S(Q)$ be such $f(gyg^{-1}) = hf(y)h^{-1}$ for all $y \in P$. Since $x = \widetilde{f}'(g)$, it follows that $f'(gyg^{-1}) = xf'(y) x^{-1}$, and hence
$$
\gamma(hf(y)h^{-1}) = f'(gyg^{-1}) = x f'(y) x^{-1},
$$
which in turn implies that $x = \widetilde{\gamma}(h)c$ for some $c \in C_S(R)$. Now, note that $C_S(P) \leq N_{f'}$, $\widetilde{\gamma}(C_S(Q)) \leq C_S(R)$ and, since $Q$ is fully $\FF$-centralized, we deduce that $\widetilde{\gamma}(C_S(Q)) = C_S(R)$. Thus, $c \in \widetilde{\gamma}(N_{\gamma})$, and hence $x \in \widetilde{\gamma}(N_{\gamma})$. It follows then that $\widetilde{f}'(N_f) \leq \widetilde{\gamma}(N_{\gamma})$.
\end{proof}

We also prove the following generalization of \cite[Proposition 1.1]{LO} for infinite fusion systems, which we will apply in Section \S \ref{EXO}.

\begin{prop}\label{Sat1}

Let $\FF$ be a fusion system over a discrete $p$-toral group $S$. Then, $\FF$ is saturated if and only if it satisfies axiom (III) of saturated fusion systems and there exists a set $\frakx$ of elements of order $p$ in $S$ such that the following conditions hold:
\begin{enumerate}[(i)]

\item each $x \in S$ of order $p$ is $\FF$-conjugate to some $y \in \frakx$;

\item if $x,y$ are $\FF$-conjugate elements of order $p$ and $y \in \frakx$, then there is some morphism $\rho \in \Hom_{\FF}(C_S(x), C_S(y))$ such that $\rho(x) = y$; and

\item for each $x \in \frakx$ the centralizer fusion system $C_{\FF}(x)$ is saturated.

\end{enumerate}

\end{prop}

\begin{proof}

Clearly if $\FF$ is saturated then the set $\frakx$ of all elements of order $p$ which are fully $\FF$-centralized satisfies the conditions in the statement. Suppose then that $\FF$ and $\frakx$ satisfy the conditions above, and let us prove that $\FF$ is saturated. Set
$$
\begin{array}{l}
U \defin \{(P,x) \,\, | \,\, P \leq S, \, x \in Z(P)^{\Gamma} \mbox{ of order } p \mbox{, with } \Gamma \in \Syl_p(\Aut_{\FF}(P)) \mbox{ such that } \Aut_S(P) \leq \Gamma\} \\[2pt]
U_0 \defin \{(P,x) \in U \,\, | \,\, x \in \frakx\}
\end{array}
$$
and note that for each nontrivial $P \leq S$ there is some $x \in P$ such that $(P,x) \in U$. We first check the following claim
\begin{itemize}

\item[(\textasteriskcentered)] If $(P,x) \in U_0$ and $P$ is fully $C_{\FF}(x)$-centralized, then $P$ is fully $\FF$-centralized.

\end{itemize}
Suppose otherwise, and let $P' \in P^{\FF}$ be fully $\FF$-centralized and $f \in \Iso_{\FF}(P,P')$. Let also $x' = f(x) \in Z(P')$. By property (ii) of $\frakx$, there exists some $\rho \in \Hom_{\FF}(C_S(x'), C_S(x))$ such that $\rho(x') = x$. Set then $P'' = \rho(P')$ and notice that in particular $\rho \circ f \in \Iso_{C_{\FF}(x)}(P,P'')$. Also, since $C_S(P') \leq C_S(x')$, the morphism $\rho$ sends $C_S(P')$ injectively into $C_S(P'')$. Furthermore,
$$
|C_S(P)| \leq |C_S(P')| \leq |C_S(P'')|,
$$
which contradicts the hypothesis that $P$ is fully $C_{\FF}(x)$-centralized. This proves (\textasteriskcentered).

Now, note that, by definition $N_S(P) \leq C_S(x)$ for all $(P,x) \in U$, and hence
$$
\Aut_{C_S(x)}(P) = \Aut_S(P).
$$
If $(P,x) \in U$ and $\Gamma \in \Syl_p(\Aut_{\FF}(P))$ are as in the definition of $U$, then $\Gamma \leq \Aut_{C_{\FF}(x)}(P)$, and in particular the following holds for all $(P,x) \in U$:
\begin{itemize}

\item[(\textasteriskcentered \textasteriskcentered)] $\Aut_S(P) \in \Syl_p(\Aut_{\FF}(P))$ if and only if $\Aut_{C_S(x)}(P) \in \Syl_p(\Aut_{C_{\FF}(x)}(P))$.

\end{itemize}
We can check now that $\FF$ satisfies axioms (I) and (II) of saturated fusion systems.

\begin{itemize}

\item[(I)] If $P$ is fully $\FF$-normalized then it is fully $\FF$-centralized and $\Aut_S(P) \in \Syl_p(\Aut_{\FF}(P))$.

\end{itemize}

By definition, $|N_S(P)| \geq |N_S(P')|$ for all $P' \in P^{\FF}$. Choose then $x \in Z(P)$ such that $(P,x) \in U$ and let $\Gamma \in \Syl_p(\Aut_{\FF}(P))$ be such that $\Aut_S(P) \leq \Gamma$ and $x \in Z(P)^{\Gamma}$. Then, by properties (i) and (ii) of the set $\frakx$, there is some $y \in \frakx$ and some $\rho \in \Hom_{\FF}(C_S(x), C_S(y))$ such that $\rho(x) = y$. Set $P' = \rho(P)$ and $\Gamma' = \rho \circ \Gamma \circ \rho^{-1} \in \Syl_p(\Aut_{\FF}(P'))$.

Note that $P'$ is fully $\FF$-normalized too, since $P$ is fully $\FF$-normalized and $N_S(P) \leq C_S(x)$. Also, $\Aut_S(P') \leq \Gamma'$ and $y \in Z(P')^{\Gamma'}$, so $(P', y) \in U_0$.

By property (iii) of $\frakx$, the fusion system $C_{\FF}(y)$ is saturated, and $P'$ is fully $C_{\FF}(y)$-normalized since it is fully $\FF$-normalized. Hence the following holds.
\begin{enumerate}[(i)]

\item By (\textasteriskcentered), $P'$ is fully $\FF$-centralized.

\item By (\textasteriskcentered \textasteriskcentered), $\Aut_S(P') \in \Syl_p(\Aut_{C_{\FF}(y)}(P')) = \Syl_p(\Aut_{\FF}(P'))$.

\end{enumerate}
Since $\rho(N_S(P)) = N_S(P')$, the same holds for $P$, and (I) is proved.

\begin{itemize}

\item[(II)] Let $f \in \Hom_{\FF}(P,S)$ be such that $f(P)$ is fully $\FF$-centralized. Then there exists $\widetilde{f} \in \Hom_{\FF}(N_f, S)$ such that $f = \widetilde{f}|_P$, where $N_f = \{g \in N_S(P) \, | \, f \circ c_g \circ f^{-1} \in \Aut_S(f(P))\}$.

\end{itemize}

Set for simplicity $P' = f(P)$, and choose $x' \in Z(P')$ of order $p$ and which is fixed by the action of $\Aut_S(P')$, and let $x = f^{-1}(x) \in Z(P)$. For all $g \in N_f$, the morphism $f c_g f^{-1}$ fixed $x'$, and thus $c_g(x) = x$. Hence,
\begin{itemize}

\item[(\textdagger)] $x \in Z(N_f)$, and $N_f \leq C_S(x)$, $N_S(P') \leq C_S(x')$.

\end{itemize}

Let also $y \in \frakx$ be $\FF$-conjugate to $x$ and $x'$, and let $\rho \in \Hom_{\FF}(C_S(x), C_S(y))$ and $\rho' \in \Hom_{\FF}(C_S(x'), C_S(y))$ be as in property (ii) of $\frakx$. Sel also $Q = \rho(P)$ and $Q' = \rho(P')$. Since $P'$ is fully $\FF$-centralized and $C_S(P') \leq C_S(x')$, it follows that
\begin{itemize}

\item[(\textdagger \textdagger)] $\rho'(C_{C_S(x')}(P')) = \rho(C_S(P')) = C_S(Q') = C_{C_S(y)}(Q')$.

\end{itemize}
since $Q'$ is clearly also fully $\FF$-centralized, and hence fully $C_{\FF}(y)$-centralized by (\textasteriskcentered).

Set $\omega = \rho' \circ f \circ \rho^{-1} \in \Iso_{\FF}(Q,Q')$. By construction, $\omega(y) = y$, and thus $\omega \in \Iso_{C_{\FF}(y)}(Q,Q')$, and we can apply axiom (II) of saturated fusion systems to $\omega$ (considered as a morphism in $C_{\FF}(y)$). It follows then that $\omega$ extends to some $\widetilde{\omega} \in \Hom_{C_{\FF}(y)}(N_{\omega}, C_S(y))$, where $N_{\omega} = \{g \in N_{C_S(y)}(Q) \, | \, \omega \circ c_g \circ \omega^{-1} \in \Aut_{C_S(y)}(Q')\}$.

By (\textdagger), for all $g \in N_f \leq C_S(x)$, we have
$$
c_{\widetilde{\omega}(\rho(g))} = \omega \, c_{\rho(g)} \, \omega^{-1} = (\omega \, \rho) \, c_g \, (\omega \, \rho)^{-1} = (\rho' \, f) \, c_g \, (\rho' \, f)^{-1} = c_{\rho'(h)} \in \Aut_{C_S(y)}(Q'),
$$
for some $h \in N_S(P')$ such that $f c_g f^{-1} = c_h$. In particular, $\rho(N_f) \leq N_{\omega}$ and, by (\textdagger \textdagger) we have $\widetilde{\omega}(\rho(N_f)) \leq \rho'(N_{C_S(x')}(P'))$. Thus axiom (II) is satisfied with $\widetilde{f} \defin (\rho')^{-1} \circ (\widetilde{\omega} \circ \rho)|_{N_f}$.
\end{proof}

\begin{defi}

Let $\FF$ be a saturated fusion system over a discrete $p$-toral group $S$.
\begin{itemize}

\item A subgroup $P \leq S$ is \textit{$\FF$-centric} if $C_S(Q) = Z(Q)$ for all $Q \in P^{\FF}$.

\item A subgroup $P \leq S$ is \textit{$\FF$-radical} if $\Out_{\FF}(P)$ contains no nontrivial normal $p$-subgroup.

\end{itemize}

\end{defi}

Given a saturated fusion system $\FF$ over a discrete $p$-toral group $S$, we denote by $\FF^c$ and $\FF^r$ the full subcategories of $\FF$ with object sets the collections of $\FF$-centric and $\FF$-radical subgroups, respectively. We also set $\FF^{cr} \subseteq \FF$ for the full subcategory of $\FF$-centric $\FF$-radical subgroups.

\begin{defi}\label{definormalA}

Let $\FF$ be a saturated fusion system over a discrete $p$-toral group.
\begin{itemize}


\item A subgroup $A \leq S$ is \textit{strongly $\FF$-closed} if, for all $P \leq S$ and all $f \in \Hom_{\FF}(P,S)$, $f(P \cap A) \leq A$.

\item A subgroup $A \leq S$ is \textit{$\FF$-normal} if, for all $P \leq S$ and all $f \in \Hom_{\FF}(P,S)$, there is $\gamma \in \Hom_{\FF}(P\cdot A, S)$ such that $\gamma|_P = f$ and $\gamma|_A \in \Aut_{\FF}(A)$.

\item A subgroup $A \leq S$ is \textit{$\FF$-central} if $A$ is $\FF$-normal and $\Aut_{\FF}(A) = \{\Id\}$.

\end{itemize}
The \textit{center} of $\FF$, denoted by $Z(\FF)$, is the maximal subgroup of $Z(S)$ that is $\FF$-central.

\end{defi}

It is an easy exercise to check that an $\FF$-normal subgroup is always strong $\FF$-closed. In particular, if $A \leq S$ satisfies any of the above three properties then $A$ is a normal subgroup of $S$. Regarding the notion of $\FF$-central subgroup, note that if $A$ is $\FF$-central then $A$ must be abelian, since $\Aut_{\FF}(A) = \{\Id\}$. There are several alternative definitions for the center of a saturated fusion system $\FF$, all of which are equivalent. See for instance \cite[Section \S 6]{BCGLO2}.

\begin{rmk}\label{rmknorm}

Let $\FF$ be a saturated fusion system over a discrete $p$-toral group $S$, and let $A \leq S$ be an strongly $\FF$-closed subgroup. Then $N_S(P) \leq N_S(P \cap A)$ for all $P \leq S$.

\end{rmk}

Aschbacher studied in \cite{Aschbacher} a notion of normality for saturated fusion systems and subsystems over finite $p$-groups. The same notion is valid for fusion systems over discrete $p$-toral groups, so we state it below without changes.

\begin{defi}\label{definormal}

Let $\FF$ be a saturated fusion system over a discrete $p$-toral group $S$, and let $\FF' \subseteq \FF$ be a fusion subsystem over a subgroup $S' \leq S$. The subsystem $\FF'$ is \textit{normal} in $\FF$ if the following conditions hold.
\begin{itemize}

\item[(N1)] $S'$ is strongly $\FF$-closed.

\item[(N2)] For each $P \leq Q \leq S'$ and each $\gamma \in \Hom_{\FF}(Q,S)$, the map sending $f \in \Hom_{\FF'}(P,Q)$ to $\gamma \circ f \circ \gamma^{-1}$ is a bijection of sets between $\Hom_{\FF'}(P,Q)$ and $\Hom_{\FF'}(\gamma(P), \gamma(Q))$.

\item[(N3)] $\FF'$ is a saturated fusion system over $S'$.

\item[(N4)] Each $f \in \Aut_{\FF'}(S')$ extends to some $\widetilde{f} \in \Aut_{\FF}(S' \cdot C_S(S'))$ such that
$$
[\widetilde{f}, C_S(S')] \defin \{\widetilde{f}(g) \cdot g^{-1} \,\, | \,\, g \in C_S(S')\} \leq Z(S').
$$

\end{itemize}

\end{defi}

In particular, if $\FF$ is a saturated fusion system over a discrete $p$-toral group $S$ and $A \leq S$ is an $\FF$-normal subgroup, then the fusion subsystem $\FF_A(A) \subseteq \FF$ is normal.

The notion of transporter system associated to a fusion system was first introduced in \cite{OV} for fusion systems over finite $p$-groups, and then extended to discrete $p$-toral groups in \cite{BLO6}, with centric linking systems being a particular case. We refer the reader to the aforementioned sources for further details.

Let $G$ be a group and let $\hh$ be a family of subgroups of $G$ which is invariant under $G$-conjugacy and over-groups. The transporter category of $G$ with respect to $\hh$ is the category $\TT_{\hh}(G)$ with object set $\hh$ and morphism sets
$$
\Mor_{\TT_{\hh}(G)}(P,Q) = \{x \in G \,\, | \,\, x \cdot P \cdot x^{-1} \leq Q\}
$$
for each pair of subgroups $P,Q \in \hh$.

\begin{defi}\label{defitransporter}

Let $\FF$ be a fusion system over a discrete $p$-toral group $S$. A \textit{transporter system} associated to $\FF$ is a nonempty category $\TT$ such that $\Ob(\TT) \subseteq \Ob(\FF)$ is closed under $\FF$-conjugacy and over-groups, together with a pair of functors
$$
\TT_{\Ob(\TT)}(S) \Right4{\varepsilon} \TT \qquad \mbox{and} \qquad \TT \Right4{\rho} \FF
$$
satisfying the following conditions.
\begin{itemize}

\item[(A1)] The functor $\varepsilon$ is the identity on objects and an inclusion on morphism sets, and the functor $\rho$ is the inclusion on objects and a surjection on morphism sets.

\item[(A2)] For each $P, Q \in \Ob(\TT)$, the kernel
$$
E(P) \defin \Ker \big[\rho_P \colon \Aut_{\TT}(P) \Right2{} \Aut_{\FF}(P) \big]
$$
acts freely on $\Mor_{\TT}(P,Q)$ by right composition, and $\rho_{P,Q}$ is the orbit map of this action. Also, $E(Q)$ acts freely on $\Mor_{\TT}(P,Q)$ by left composition.

\item[(B)] For each $P,Q \in \Ob(\TT)$, $\varepsilon_{P,Q} \colon N_S(P,Q) \to \Mor_{\TT}(P,Q)$ is injective, and the composite $\rho_{P,Q} \circ \varepsilon_{P,Q}$ sends $g \in N_S(P,Q)$ to $c_g \in \Hom_{\FF}(P,Q)$.

\item[(C)] For all $\varphi \in \Mor_{\TT}(P,Q)$ and all $g \in P$, the diagram
$$
\xymatrix{
P \ar[r]^{\varphi} \ar[d]_{\varepsilon_P(g)} & Q \ar[d]^{\varepsilon_Q(\rho(\varphi)(g))} \\
P \ar[r]_{\varphi} & Q
}
$$
commutes in $\TT$.

\item[(I)] Each $\FF$-conjugacy class of subgroups in $\Ob(\TT)$ contains a subgroup $P$ such that $\varepsilon_P(N_S(P)) \in \Syl_p(\Aut_{\TT}(P))$; that is, such that $[\Aut_{\TT}(P)\colon \varepsilon(N_S(P))]$ is finite and prime to $p$.

\item[(II)] Let $\varphi \in \Iso_{\TT}(P,Q)$, $P \lhd \widetilde{P} \leq S$ and $Q \lhd \widetilde{Q} \leq S$ be such that $\varphi \circ \varepsilon_P(\widetilde{P}) \circ \varphi^{-1} \leq \varepsilon_Q(\widetilde{Q})$. Then there is some $\widetilde{\varphi} \in \Mor_{\TT}(\widetilde{P}, \widetilde{Q})$ such that $\widetilde{\varphi} \circ \varepsilon_{P, \widetilde{P}}(1) = \varepsilon_{Q, \widetilde{Q}}(1) \circ \varphi$.

\item[(III)] Assume $P_1 \leq P_2 \leq P_3 \leq \ldots$ in $\Ob(\TT)$ and $\varphi_n \in \Mor_{\TT}(P_n,S)$ are such that, for all $n \geq 1$, $\varphi_n = \varphi_{n+1} \circ \varepsilon_{P_n, P_{n+1}}(1)$. Set $P = \bigcup_{n = 1}^{\infty}$. Then there is $\varphi \in \Mor_{\TT}(P,S)$ such that $\varphi_n = \varphi \circ \varepsilon_{P_n,P}(1)$ for all $n \geq 1$.

\end{itemize}

A \textit{centric linking system} associated to a saturated fusion system $\FF$ is a transporter system $\LL$ such that $\Ob(\LL)$ is the collection of all $\FF$-centric subgroups of $S$ and $E(P) = Z(P)$ for all $P \in \Ob(\LL)$.

\end{defi}

\begin{defi}

A \textit{$p$-local compact group} is a triple $\g = \ploc$, where $S$ is a discrete $p$-toral group, $\FF$ is a saturated fusion system over $S$, and $\LL$ is a centric linking system associated to $\FF$. The \textit{classifying space} of a $p$-local compact group $\g$ is the $p$-completed nerve of $\LL$, denoted by $B\g = |\LL|^{\wedge}_p$. The \textit{center} of $\g$ is the center of the fusion system $\FF$, and is denoted by $Z(\g)$.

\end{defi}

\begin{rmk}

Given a saturated fusion system $\FF$ over a discrete $p$-toral group $S$, Levi and Libman proved in \cite{Levi-Libman} that there is a unique centric linking system $\LL$ associated to $\FF$, thus extending the well-known result for finite saturated fusion systems and $p$-local finite groups.

\end{rmk}

\begin{prop}\label{product}

Let $\g_1 = (S_1, \FF_1, \LL_1)$ and $\g_2 = (S_2, \FF_2, \LL_2)$ be $p$-local compact groups. Then, $B\g_1 \times B\g_2$ is the classifying space of a $p$-local compact group.

\end{prop}

\begin{proof}

Consider the product fusion system $\FF_1 \times \FF_2$ over $S_1 \times S_2$ as defined in \cite[Lemma 1.5]{BLO2}. By this same result, $\FF_1 \times \FF_2$ is a saturated fusion system (the proof applies verbatim to the compact case, except for axiom (III) which is left to the reader). Let $\LL_1 \times \LL_2$ be the centric linking system associated to $\FF_1 \times \FF_2$. It follows that $B\g_1 \times B\g_2$ is equivalent to $B(\g_1 \times \g_2)$, where $\g_1 \times \g_2 = (S_1 \times S_2, \FF_1 \times \FF_2, \LL_1 \times \LL_2)$. 
\end{proof}

The following result will be useful in later sections.

\begin{lmm}\label{property1}

Let $\g = \ploc$ be a $p$-local compact group with maximal torus $T$, and let $P \leq S$ be $\FF$-conjugate to a subgroup of $T$. Then, $C_P(T) = T \cap P$.

\end{lmm}

In other words, if $x \in S$ is $\FF$-conjugate to an element of $T$, then either $x \in T$ or $x$ acts nontrivially on $T$.

\begin{proof}

Let $f \colon P \to T$ be a morphism in $\FF$. By taking restrictions if necessary, we may assume that $P = \gen{x}$, and furthermore we may assume that $f(P) \leq T$ is fully $\FF$-centralized. Hence, by axiom (II) of saturated fusion systems the morphism $f$ extends to some morphism $\widetilde{f} \colon C_S(P) \cdot P \to S$.

If $P$ acts trivially on $T$, then $T \leq C_S(P)$, and then in particular $\widetilde{f}$ restricts to a morphism $T \cdot P \to S$ (because $T$ is the maximal infinitely $p$-divisible subgroup of $S$). Furthermore, $\widetilde{f}(T \cdot P) = T$, since $f(P) \leq T$ and $\widetilde{f}(T) = T$. Hence $P \leq T$.
\end{proof}

In this paper we deal with fibrations involving (finite) transporter systems and $p$-local compact groups. Recall that a fibration over a space $B$ with fibre $F$ is classified by a map $B \to B \underline{\aut}(F)$, where $\underline{\aut}(F)$ is the topological monoid of self-equivalences of $F$. In the case where $F$ is (equivalent to) the classifying space of a $p$-local compact group $\g = \ploc$, the space $\underline{\aut}(B\g)$ can be described purely in terms of $\g$.

\begin{defi}

Let $\g = \ploc$ be a $p$-local compact group. An automorphism $\Psi \colon \LL \Right2{\cong} \LL$ is \textit{isotypical} if $\Psi(\varepsilon_P(P)) = \varepsilon_{\Psi(P)}(\Psi(P))$ for each $P \in \Ob(\LL)$.

\end{defi}

Let then $\Aut_{\typ}^I(\LL)$ be the collection of isotypical automorphisms of $\LL$ which send inclusions to inclusions. That is, $\Psi \in \Aut_{\typ}^I(\LL)$ if $\Psi(\varepsilon_{P,Q}(1)) = \varepsilon_{\Psi(P), \Psi(Q)}(1)$ whenever $P \leq Q$. This collection turns out to be a group by \cite[Lemma 1.14]{AOV}.

The elements of $\Aut_{\LL}(S)$ induce isotypical automorphisms of $\LL$ by conjugation, as follows. Fix $\varphi \in \Aut_{\LL}(S)$, and define $c_{\varphi}$ by
$$
c_{\varphi}(P) = \rho(\varphi)(P) \qquad \mbox{and} \qquad c_{\varphi}(\psi) = (\varphi|_{Q, c_{\varphi}(Q)}) \circ \psi \circ (\varphi^{-1}|_{c_{\varphi}(P), P})
$$
for each $P,Q \in \Ob(\LL)$ and all $\psi \in \Mor_{\LL}(P,Q)$. Notice that $c_{\varphi} \in \Aut_{\typ}^I(\LL)$ by construction. Actually, $\{c_{\varphi} \,| \, \varphi \in \Aut_{\LL}(S)\}$ is a normal subgroup of $\Aut_{\typ}^I(\LL)$, so we can define $\Out_{\typ}(\LL) \defin \Aut_{\typ}^I(\LL)/\{c_{\varphi} \,| \, \varphi \in \Aut_{\LL}(S)\}$. The following is a simplification of \cite[Theorem 7.1]{BLO3}.

\begin{prop}\label{auttyp}

Let $\g = \ploc$ be a $p$-local compact group. Then,
$$
\pi_i(\underline{\aut}(B\g)) \cong \left\{
\begin{array}{ll}
\Out_{\typ}(\LL) & i = 0 \\
\pi_1((BZ(\g))^{\wedge}_p) & i = 1 \\
\pi_2((BZ(\g))^{\wedge}_p) & i = 2 \\
\{0\} & i \geq 3 \\
\end{array}
\right.
$$

\end{prop}

The following generalization of \cite[Proposition 7.1]{BLO6} will also be of use in later sections.

\begin{lmm}\label{property2}

Let $\g = \ploc$ be a $p$-local compact group, and let $\TT$ be a finite transporter system associated to a saturated fusion system. Then there is a bijection from the set of equivalence classes of fibrations over $|\TT|$ with fibre $|\LL|$ and structure group  $N\autcat_{\typ}^I(\LL)$ to the set of equivalence classes of fibrations over $|\TT|$ with fibre homotopy equivalent to $B\g$: a bijection which sends the class of a fibration to the equivalence class of its fibrewise $p$-completion.

\end{lmm}

\begin{proof}

The same proof of \cite[Proposition 7.1]{BLO6} applies here upon observing that, for $i \geq 1$ we have $H^{i}(|\TT|; (\Q_p)^r) = 0$ because $\TT$ is a finite category.
\end{proof}

Next we recall the construction of the quotient fusion and transporter systems by a strongly closed subgroup and some other related notions. The following is a summary of \cite[Section \S 2]{BLO6} adapted to the needs of this paper. Let then $\FF$ be a saturated fusion system over a discrete $p$-toral group $S$. For each subgroup $A \leq S$ and each $K \leq \Aut(A)$, define
\begin{itemize}

\item $\Aut_{\FF}^K(A) = K \cap \Aut_{\FF}(A)$;

\item $\Aut_S^K(A) = K \cap \Aut_S(A)$; and

\item $N_S^K(A) = \{x \in N_S(A) \, | \, c_x \in K\}$.

\end{itemize}
The subgroup $A$ is \textit{fully $K$-normalized in $\FF$} if we have $|N_S^K(A)| \geq |N_S^{^{f}K}(f(A))|$ for each $f \in \Hom_{\FF}(A, S)$, where $^{f}K = \{f \gamma f^{-1} \, | \, \gamma \in K\} \leq \Aut(f(A))$.

\begin{defi}\label{definorm}

Let $\g = \ploc$ be a $p$-local compact group, and let $A \leq S$ be fully $K$-normalized in $\FF$ for some $K \leq \Aut(A)$. The \textit{$K$-normalizer $p$-local compact group of $A$ in $\g$} is the triple $N_{\g}^K(A) = (N_S^K(A), N_{\FF}^K(A), N_{\LL}^K(A))$, where
\begin{itemize}

\item $N_{\FF}^K(A)$ is the fusion system over $N_S^K(A)$ with morphism sets
$$
\begin{aligned}
\Hom_{N_{\FF}^K(A)}(P,Q) = \{f \in \Hom_{\FF}(P,Q) \,\, | \,\, & \exists \widetilde{f} \in \Hom_{\FF}(PA, QA) \\
 & \mbox{such that} \widetilde{f}|_P = f \mbox{ and } \widetilde{f}|_A \in K\};
\end{aligned}
$$

\item $N_{\LL}^K(A)$ is the category with object set the collection of $N_{\FF}^K(A)$-centric subgroups of $N_S^K(A)$ and with morphism sets
$$
\Mor_{N_{\LL}^K(A)}(P,Q) = \{\varphi \in \Mor_{\LL}(PA, QA) \,\, | \,\, \rho(\varphi)|_P \in \Hom_{N_{\FF}^K(A)}(P,Q)\}.
$$

\end{itemize}

\end{defi}

\begin{rmk}\label{rmknorm2}

By \cite[Theorem 2.3]{BLO6} we know that $N_{\FF}^K(A)$ is a saturated fusion system whenever $A$ is fully $K$-normalized in $\FF$, but it is not straightforward to see that $N_{\LL}^K(A)$ is well defined. This is a consequence of \cite[Lemma 6.2]{BLO2}: this result applies to the compact case, and for all $K \leq \Aut(A)$ such that $A$ is fully $K$-normalized in $\FF$, with minor modifications on the proof. Details are left to the reader.

There are some particular cases of interest arising from Definition \ref{definorm}:
\begin{itemize}

\item the \textit{normalizer $p$-local compact group of $A$ in $\g$}, denoted by $N_{\g}(A)$, corresponding to $K = \Aut(A)$; and

\item the \textit{centralizer $p$-local compact group of $A$ in $\g$}, denoted by $C_{\g}(A)$, corresponding to $K = \{\Id\}$.

\end{itemize}
A third case of interest will appear in Section \S \ref{THMC}. For $K = \Aut_S(A)$, we will denote the resulting $K$-normalized $p$-local compact group of $A$ by $N_{\g}^S(A) = (N_S(A), N_{\FF}^S(A), N_{\LL}^S(A))$. We are not aware of this particular case being mentioned anywhere else in the literature.

\end{rmk}

Let then $\g = \ploc$ be a $p$-local compact group, and let $A \leq S$ be an weakly $\FF$-closed subgroup. If $P, Q \leq S$ are such that $A \leq P,Q$, then each morphism $f \in \Hom_{\FF}(P,Q)$ restricts to an automorphism of $A$, and hence it also induces a homomorphism $\ind(f) \colon P/A \to Q/A$. For a subgroup $P/A \leq S/A$, we will denote by $P \leq S$ the unique subgroup of $S$ that contains $A$ with image $P/A$ through the projection $S \to S/A$.

The following constructions are introduced and studied in detail in \S \ref{Quotients}, so here we only recall the basics. Let $\g = \ploc$ be a $p$-local compact group, and let $A \leq S$ be a weakly $\FF$-closed subgroup, so $N_S(A) = S$. Let also $N_{\g}(A) = (N_S(A), N_{\FF}(A), N_{\LL}(A))$ be the normalizer $p$-local compact group of $A$ in $\g$. The \textit{quotient} of $\g$ by $A$ is the triple $(S/A, N_{\FF}(A)/A, N_{\LL}(A)/A)$, where
\begin{itemize}

\item $N_{\FF}(A)/A$ is the fusion system over $S/A$ with morphism sets
$$
\Hom_{N_{\FF}(A)/A}(P/A, Q/A) = \{\overline{f} \in \Hom(P/A,Q/A) \,\, | \,\, \exists f \in \Hom_{\FF}(P,Q) \mbox{ such that } \overline{f} = \ind(f)\}.
$$

\item $N_{\LL}(A)/A$ is the category with object set $\{P \leq S \, | \, A \leq P \in \Ob(N_{\LL}(A))\}$ and morphism sets
$$
\Mor_{N_{\LL}(A)/A}(P/A,Q/A) = \Mor_{N_{\LL}(A)}(P,Q)/\varepsilon_P(A).
$$

\end{itemize}
The notation for the quotient of $\g$ by $A$ is usually simplified as $\g/A = (S/A, \FF/A, \LL/A)$. By Proposition \ref{quotientF} the fusion system $\FF/A$ is saturated, and by Proposition \ref{quotientT} the category $\LL/A$ is a transporter system associated to $\FF/A$.

\begin{lmm}\label{property3}

Let $\g = \ploc$ be a $p$-local compact group with maximal torus $T$. If $T$ is strongly $\FF$-closed, then $T$ is $\FF$-normal.

\end{lmm}

The above result applies in particular when $C_S(T) = T$, as a consequence of Lemma \ref{property1}.

\begin{proof}

In order to show that $T$ is $\FF$-normal, we have to check the following: given $P,Q \leq S$ and $f \in \Hom_{\FF}(P,Q)$, there is a morphism $\widetilde{f} \in \Hom_{\FF}(P \cdot T, Q \cdot T)$ such that $\widetilde{f}|_P = f$ and $\widetilde{f}|_T \in \Aut_{\FF}(T)$. In fact, since morphisms in $\FF$ are compositions of restrictions of fully $\FF$-normalized $\FF$-centric $\FF$-radical subgroups, it is enough to check the above condition for all automorphisms of such particular kind subgroups.

Let then $P \leq S$ be and $\FF$-centric $\FF$-radical subgroup which is fully $\FF$-normalized, and let $f \in \Aut_{\FF}(P)$. By \cite[Lemma 3.5]{BCGLO2}, applied to $\FF_{T}(T) \subseteq \FF$, $P \cap T$ is $\FF_T(T)$-centric, and hence $P \cap T = T$ since $T$ is abelian. Normality of $T$ follows immediately.
\end{proof}














\section{Irreducible $p$-local compact groups of rank $1$}\label{Section2}

In this section we introduce the notion of irreducibility for $p$-local compact groups, and prove Theorem \ref{thmA} in Propositions \ref{ThmA1} (for odd primes) and \ref{ThmA2} (for $p = 2$). We leave all discussion about irreducibility and related notions to Section \S \ref{Irred}.

\begin{defi}\label{defiirred}

Let $\g = \ploc$ be a $p$-local compact group with maximal torus $T$. Then,
\begin{enumerate}[(a)]

\item $\g$ is \textit{irreducible} if $\FF$ itself is the only normal fusion subsystem of maximal rank of $\FF$.

\item $\g$ is \textit{simple} if it is irreducible and every proper normal subsystem of $\FF$ is finite.

\end{enumerate}

\end{defi}

\begin{rmk}

The notion of irreducibility implies in particular that the only irreducible $p$-local finite group is the trivial one. Note also that irreducibility does not imply simplicity. For instance, the direct product of two irreducible $p$-local compact groups is irreducible but not simple. However, when the rank is $1$ both notions agree. Recall that a simple Lie group is a Lie group that does not contain any proper normal connected subgroup. For instance, the group $SU(2)$ is simple, although it has a normal subgroup isomorphic to $\Z/2$. Thus our notion of simple $p$-local compact group is absolutely coherent with the philosophy of our approach.

\end{rmk}

\begin{prop}

Let $\g = \ploc$ be a irreducible $p$-local compact group. Then $B\g$ is simply irreducible.

\end{prop}

\begin{proof}

By the Hyperfocal Subgroup Theorem \ref{hyper3}, $\pi_1(B\g) \cong S/O^p_{\FF}(S)$. Suppose then that $\pi_1(B\g)$ is not trivial. Then, by Theorem \ref{fmt4} and Corollary \ref{fmt5}, there is a normal, proper subsystem $\FF_0 \subseteq \FF$ over $S_0 = O^p_{\FF}(S)$. In particular, $T \leq S_0$ by definition of $O^p_{\FF}(S)$, and this contradicts the hypothesis that $\g$ is irreducible.
\end{proof}

We describe all irreducible $p$-local compact groups of rank $1$ in the following three examples.  In particular we will tacitly use \cite[Section \S 9]{BLO3} as a recipe to construct $p$-local compact groups out of compact Lie groups, although we do not recall here the instructions given in \cite{BLO3}. After these examples have been analyzed we devote the rest of this section to show that these are in fact the only possible examples, hence proving Theorem \ref{thmA}.

\begin{expl}\label{expl1}

The $p$-local compact group $\g = \ploc$ induced by the circle $SO(2)$ is a rather trivial example, since its Sylow $p$-subgroup is abelian. Indeed, it is determined by the following data:
\begin{enumerate}[(a)]

\item $S = \prufferp$;

\item $\FF$ is generated by $\Aut_{\FF}(S) = \{\Id\}$; and

\item $\LL$ has object set $\Ob(\LL) = \{S\}$ and $\Aut_{\LL}(S) = S$.

\end{enumerate}
Note that there are equalities $\FF = \FF_S(S)$ and $\LL = \LL_S(S)$. Hence it is obvious that this $p$-local compact group is irreducible. Also, the following is a straightforward calculation.
\begin{equation}\label{AutT}
\Out_{\typ}(\LL) = \Aut_{\typ}(\LL) \cong \Aut(\prufferp) \cong \left\{
\begin{array}{ll}
\Z/2 \times \dosadic, & p = 2 \\
\Z/(p-1) \times \padic, & p > 2\\
\end{array}
\right.
\end{equation}

\end{expl}

\begin{cor}\label{autG1}

Let $\g = \ploc$ be the $p$-local compact group induced by $SO(2)$. Then,
$$
\pi_i(\underline{\aut}(B\g)) \cong \left\{
\begin{array}{ll}
\Aut(\prufferp) & i = 0 \\
\{0\} & i = 1 \\
\padic & i = 2 \\
\{0\} & i \geq 3 \\
\end{array}
\right.
$$

\end{cor}

\begin{proof}

This is a direct consequence of Proposition \ref{auttyp}.
\end{proof}

\begin{rmk}\label{expl11}

Let $p$ be an odd prime and let $n \geq 2$ be a divisor of $p-1$. By (\ref{AutT}), the group $\Aut(\prufferp)$ contains a subgroup isomorphic to $\Z/n$, namely $\Gamma_n$, and then the $p$-local compact group $\g = \ploc$ induced by the Sullivan sphere of dimension $2n-1$ is determined by the following data:
\begin{enumerate}[(a)]

\item $S = \prufferp$;

\item $\FF$ is generated by $\Aut_{\FF}(S) = \Gamma_n$; and

\item $\LL$ has object set $\Ob(\LL) = \{S\}$ and $\Aut_{\LL}(S) = S \rtimes \Gamma_n$.

\end{enumerate}
It is clear that $\FF_S(S) \lhd \FF$ is a proper normal subsystem of maximal rank, and thus $\g$ is not irreducible.

\end{rmk}

\begin{expl}\label{expl2}

The $p$-local compact group $\g = \ploc$ induced by $SO(3)$. The Weyl group of $SO(3)$ is $\Z/2$, acting by the reflection
\begin{equation}\label{reflection}
\xymatrix@R=1mm{
SO(2) \ar[rr]^{\tau} & & SO(2)\\
z \ar@{|->}[rr] & & z^{-1}
}
\end{equation}
and $N_{SO(3)}(SO(2)) \cong SO(2) \rtimes \Z/2$, where the $\Z/2$ factor is generated by the above reflection. Let $T \leq SO(2)$ be the subgroups of all $p^n$-roots of unity, for all $n \geq 0$, so $T \cong \prufferp$. Let us treat the case $p = 2$ apart from the case $p > 2$.

Suppose first that $p > 2$. In this case we have $T = S = \prufferp$. Since $S$ is abelian, the whole fusion system is determined by $\Aut_{\FF}(S) = \Z/2$ (generated by the restriction of the automorphism $\tau$ above). Also, $\Ob(\LL) = \{S\}$, with $\Aut_{\LL}(S) = S \rtimes \Z/2$. In particular, the classifying space of $\g$ is the classifying space of a Sullivan sphere, and it is not irreducible.

Suppose then that $p = 2$. In this case we have
$$
S = \gen{\{t_n\}_{n \geq 1}, x \,\, | \,\, \forall n: \,\, t_n^{2^n} = x^2 = 1, \,\, t_{n+1}^2 = t_n, \,\, x \cdot t_n \cdot x^{-1} = t_n^{-1} } \cong D_{2^{\infty}}.
$$
Set $T = \gen{\{t_n\}_{n \geq 1}} \leq S$, and note that $Z(S) = \gen{t_1} \cong \Z/2$. Set also $V = \gen{t_1,x} \cong \Z/2 \times \Z/2$, with $N_S(V) = \gen{t_2,x} \cong D_8$. By \cite[II.3.8]{Adem-Milgram} we have $H^1(\Z/2;T) = \{1\}$ (with coefficients twisted by the obvious action of $\Z/2$ on $T$), and it follows that the projection $S \to \gen{x} \cong \Z/2$ has a unique section up to $S$-conjugacy. From this, one can deduce that every elementary abelian subgroup of $S$ isomorphic to $(\Z/2)^2$ is $S$-conjugate to $V$.

In fact,  the only $\FF$-centric $\FF$-radical subgroups are $S$ and the elements of $V^{\FF}$, and thus we can see $\FF$ as the fusion system over $S$ generated by 
$$
\Aut_{\FF}(S) = \Inn(S) \qquad \mbox{and} \qquad \Aut_{\FF}(V) = \Aut(V) \cong \Sigma_3
$$
(that is, every morphism in $\FF$ is a composition of restrictions of morphisms in the above sets). The proof that $\FF$ defined as above is saturated is left to the reader as an easy exercise.

The linking system $\LL$ is then determined by
$$
\Aut_{\LL}(S) = S \qquad \mbox{and} \qquad \Aut_{\LL}(V) \cong \Sigma_4.
$$
By \cite[Corollary 4.2]{DMW} it follows that $\LL$ is the centric linking system induced by $SO(3)$ (up to isomorphism). Another interesting feature of this $2$-local compact group is that its center is trivial, i.e. $Z(\g) = \{1\}$.

It remains to check that in this case $\g$ is irreducible. Let $\FF_0 \subseteq \FF$ be a proper, saturated subsystem over some $S_0 \leq S$ such that $T \leq S_0$. Since $T \leq S$ has index $2$ it follows that either $S_0 = T$ or $S_0 = S$. Furthermore, a careful inspection shows that
$$
\FF_0 = \FF_T(T) \qquad \mbox{or} \qquad \FF_0 = \FF_S(S),
$$
depending on the isomorphism type of $S_0$. However, every element of $S$ is $\FF$-conjugate to an element of $T$, and hence $T$ is not strongly $\FF$-closed. Thus we must have $\FF_0 = \FF_S(S)$. In this case, normality of $\FF_0$ in $\FF$ fails since condition (N2) in Definition \ref{definormal} is not satisfied (take for instance $P = Q = V$ and $\gamma \in \Aut_{\FF}(V)$ an element of order $3$).

\end{expl}

\begin{expl}\label{expl3}

The $p$-local compact group $\g = \ploc$ induced by $SU(2)$. Recall that there is a principal fibration
$$
B\Z/2 \Right4{} BSU(2) \Right4{} BSO(3),
$$
which we can $p$-complete at all primes $p$ since $\pi_1(BSO(3))$ is trivial, by the Fibre Lemma \cite[II.5.1]{BK}. As we did in Example \ref{expl2}, we distinguish between the case $p > 2$ and the case $p = 2$.

Suppose first that $p > 2$. In this case the $p$-completion of the above fibration yields an equivalence
$$
(BSU(2))^{\wedge}_p \simeq (BSO(3))^{\wedge}_p.
$$
As an immediate consequence, the $p$-local compact group $\g$ induced by $SU(2)$ is the same than the $p$-local compact group induced by $SO(3)$. In particular, the resulting $p$-local compact group is not irreducible.

Suppose then that $p = 2$. In this case we have
$$
S = \gen{\{t_n\}_{n \geq 1}, y \,\, | \,\, \forall n: \,\, t_n^{2^n} = y^4 = 1, \,\, t_{n+1}^2 = t_n, \,\, y^2 = t_1, \,\, y \cdot t_n \cdot y^{-1} = t_n^{-1}} \cong Q_{2^{\infty}}.
$$
Set $T = \gen{\{t_n\}_{n \geq 1}} \leq S$, with $Z(S) = \gen{t_1} \cong \Z/2$. Set also $W = \gen{t_2, y} \cong Q_8$, with $N_S(W) = \gen{t_3,y} \cong Q_{16}$. By the same arguments used above we deduce that every subgroup $P \leq S$ which is isomorphic to $Q_8$ is $S$-conjugate to $W$.

The only $\FF$-centric $\FF$-radical subgroups are $S$ and the elements of $W^{\FF}$, and thus $\FF$ is the fusion system over $S$ generated by
$$
\Aut_{\FF}(S) = \Inn(S) \qquad \mbox{and} \qquad \Aut_{\FF}(W) = \Aut(W) \cong \Sigma_4.
$$
We leave as an exercise the proof that $\FF$ is saturated. Using the notation of \cite[Section \S 4]{DMW}, the linking system $\LL$ is determined by
$$
\Aut_{\LL}(S) = S \qquad \mbox{and} \qquad \Aut_{\LL}(W) \cong \widetilde{\oo}_{48}
$$
where $\widetilde{\oo}_{48}$ is the inverse image in $SU(2)$ of the subgroup $N_{SO(3)}(V) =\Sigma_4 \leq SO(3)$. By \cite[Theorem 4.1]{DMW} it follows that $\LL$ is the centric linking system induced by $SU(2)$ (up to isomorphism). The same arguments applied in Example \ref{expl2} imply that in this case $\g$ is irreducible.

We finish this example by relating it to \ref{expl2}. First notice that the center of $\g$ is
$$
Z(\g) = \gen{t_1} \cong \Z/2 = Z(S).
$$
Since central subgroups of $\g$ are, in particular, normal in $\FF$, we can consider the quotient of $\g$ by $Z(\g)$, namely $\g/Z(\g) = (S/Z(\g), \FF/Z(\g), \LL/Z(\g))$, which is easily seen to be isomorphic to the $2$-local compact group induced by $SO(3)$. Also, there is a fibration
$$
B\Z/2 \Right4{} |\LL| \Right4{} |\LL/Z(\g)|
$$
whose $2$-completion produces the principal fibration $B\Z/2 \to (BSU(2))^{\wedge}_2 \to (BSO(3))^{\wedge}_2$ (all this can be checked by direct inspection).

\end{expl}

The following computations, related to the above examples, will be needed later.

\begin{prop}\label{AutS}

Let $\g = \ploc$ be the $2$-local compact group induced by either $SO(3)$ or $SU(2)$. Then,
$$
\Out(S) = \Out_{\fus}(S) = \Out_{\typ}(\LL) \cong \dosadic.
$$

\end{prop}

\begin{proof}

The group $S$ fits in an extension $T \to S \to \Z/2$ where the quotient $\Z/2$ acts on $T$ via the restriction of (\ref{reflection}). Thus by \cite[2.8.7]{Suzuki} there is an exact sequence
\begin{equation}\label{exact1}
\{1\} \to H^1(\Z/2;T) \Right3{} \Aut(S)/\Aut_T(S) \Right3{} \Aut(T) \times \Aut(\Z/2),
\end{equation}
where the $H^1(\Z/2;T)$ stands for cohomology with twisted coefficients (by the aforementioned action of $\Z/2$ on $T$). We claim that the above sequence induces an exact sequence
$$
\{1\} \to \Aut_T(S) \Right4{} \Aut(S) \Right4{\Phi} \Aut(T) \to \{1\}.
$$

Indeed, by \cite[II.3.8]{Adem-Milgram} it follows that $H^1(\Z/2;T) = \{1\}$. Also the term $\Aut(\Z/2) = \{\Id\}$ can ignored, and then the sequence (\ref{exact1}) reduces to $\{1\} \to \Aut_T(S) \Right2{} \Aut(S) \Right2{\Phi} \Aut(T)$. To finish the proof of the claim we show that $\Phi$ is surjective.

Recall from (\ref{AutT}) that $\Aut(T) \cong \Z/2 \times \dosadic$, and in this case there is some $g \in S$ such that the subgroup $\Z/2 \leq \Aut(T)$ is generated by $c_g$. Finally, it is clear that every automorphism in $\dosadic \leq \Aut(T)$ extends to an automorphism of $S$. Hence the claim holds.

In fact, there is a commutative diagram of exact sequences
$$
\xymatrix{
\Aut_T(S) \ar[r] \ar@{=}[d] & \Inn(S) \ar[r] \ar[d] & \gen{c_g} \ar[d] \\
\Aut_T(S) \ar[r] & \Aut(S) \ar[r] \ar[d] & \Aut(T) \ar[d] \\
 & \dosadic \ar@{=}[r] & \dosadic\\
}
$$
Hence $\Out(S) \cong \dosadic$. Next we show that $\Out_{\fus}(S) = \Out(S)$. Recall that $\Aut_{\fus}(S)$ is the collection of automorphism of $S$ which induce an automorphism of $\FF$, and $\Out_{\fus}(S) = \Aut_{\fus}(S)/\Aut_{\FF}(S)$. An easy calculation shows in this case that $\Aut_{\fus}(S) = \Aut(S)$, and hence $\Out_{\fus}(S) = \Out(S)$ since $\Aut_{\FF}(S) = \Inn(S)$.

Finally, we show that $\Out_{\typ}(\LL) \cong \Out_{\fus}(S)$. By \cite[Proposition 7.2]{BLO3} together with \cite[Proposition 5.8]{BLO3} there is an exact sequence
$$
\{0\} \to \limn{1}_{\oo^c(\FF)} (\mathcal{Z}) \Right3{} \Out_{\typ}(\LL) \Right3{} \Out_{\fus}(S) \Right3{} \limn{2}_{\oo^c(\FF)} (\mathcal{Z}),
$$
where $\oo^c(\FF)$ is the orbit category of $\FF^c$, and $\mathcal{Z}(P) = Z(P)$ (recall that $P$ is a discrete $p$-toral group, and thus its center may have positive rank). By \cite[Theorem 4.1]{JMO1} both higher limits vanish, and the isomorphism holds.
\end{proof}

\begin{cor}\label{autG2}

Let $\g = \ploc$ be the $2$-local compact group induced by $SO(3)$. Then,
$$
\pi_i(\underline{\aut}(B\g)) \cong \left\{
\begin{array}{ll}
\dosadic & i = 0 \\
\{0\} & i \geq 1 \\
\end{array}
\right.
$$

\end{cor}

\begin{proof}

Since $Z(\g) = \{1\}$, this follows from Propositions \ref{auttyp} and \ref{AutS}.
\end{proof}

\begin{cor}\label{autG3}

Let $\g = \ploc$ be the $2$-local compact group induced by $SU(2)$. Then,
$$
\pi_i(B\underline{\aut}(B\g)) \cong \left\{
\begin{array}{ll}
\dosadic & i = 0 \\
\Z/2 & i = 1 \\
\{0\} & i \geq 2 \\
\end{array}
\right.
$$

\end{cor}

\begin{proof}

Since $Z(\g) \cong \Z/2$ this follows from Propositions \ref{auttyp} and \ref{AutS}.
\end{proof}

We prove now Theorem \ref{thmA}. We deal first with the case $p > 2$, since the case $p = 2$ requires a longer proof.

\begin{prop}\label{ThmA1}

Let $p$ be an odd prime, and let $\g = \ploc$ be a irreducible $p$-local compact group of rank $1$. Then $\g$ is induced by $SO(2)$.

\end{prop}

\begin{proof}

Since $p$ is odd, $\Aut(\prufferp)$ does not contain any nontrivial finite $p$-subgroup, and this implies that every $x \in S$ centralizes $T$. By Lemma \ref{property1} this means that $T$ is strongly $\FF$-closed, and then Lemma \ref{property3} implies that $T$ is $\FF$-normal. Since $\g$ is irreducible, this means that $\FF = \FF_T(T)$, and the statement follows.
\end{proof}

\begin{prop}\label{ThmA2}

Let $\g = \ploc$ be a irreducible $2$-local compact group. Then $\g$ is induced by either $SO(2)$, $SO(3)$ or $SU(2)$.

\end{prop}

\begin{proof}

Suppose first that $S = T$. Then, $\Aut_{\FF}(S) = \{1\}$ by (\ref{AutT}), and it follows that $\g$ is the $2$-local compact group induced by $SO(2)$. Suppose then that $T \lneqq S$. If $T$ has index $2$ in $S$, then $\g$ must be isomorphic to the $2$-local compact group induced by either $SO(3)$ (Example \ref{expl2}) or $SU(2)$ (Example \ref{expl3}). Indeed, if $|S/T| = 2$ then $S$ fits in a group extension
$$
T \Right4{} S \Right4{} \Z/2,
$$
where $\Z/2$ acts on $T$ by (\ref{reflection}). By \cite[II.3.8]{Adem-Milgram} we have $H^2(\Z/2; T) \cong \Z/2$ as sets (coefficients are twisted by the given action), and thus either $S \cong D_{2^{\infty}}$ or $S \cong Q_{2^{\infty}}$, as described in Examples \ref{expl2} and \ref{expl3} respectively. Using the explicit description that we have made of each of these examples it is easy to check then that there are no other irreducible $2$-local compact groups of rank $1$ over these discrete $2$-toral groups.

We have to show that if $\g$ is irreducible then $T$ must be of index $2$ in $S$. We proceed by contradiction, so suppose that $|S/T| \geq 4$. Fix also the following:
\begin{itemize}

\item some $x \in S\setminus T$ which is $\FF$-conjugate to an element of $T$; and

\item some morphism $f \colon \gen{x} \to T$ in $\FF$.

\end{itemize}
Set $S_0 = \gen{T,x} \leq S$ for short. By \cite[II.3.8]{Adem-Milgram}, $S_0$ is isomorphic to either $D_{2^{\infty}}$ or $Q_{2^{\infty}}$, and in particular $\overline{x} \in S/T$ is an element of order $2$, regardless of the isomorphism type of $S_0$.

Set also $K = \Ker(S/T \to \Aut(T))$ and let $\overline{x} \in S/T$ be the image of $x$ through the obvious projection. The above implies that $S/T \cong K \rtimes \gen{\overline{x}}$: if $\tau \in \Aut(T)$ denotes the reflection induced by (\ref{reflection}), then the map $S/T \to \Aut(T)$ maps $\overline{x}$ to $\tau$, and the exact sequence of groups
$$
\{1\} \to K \Right4{} S/T \Right4{} \gen{\tau} \to \{1\}
$$
has a section. The proof is divided into steps.
\begin{itemize}

\item[\textbf{Step 1.}] The action of $\overline{x}$ on $K$ is trivial. Hence $S/T \cong K \times \gen{\overline{x}}$.

\end{itemize}

Suppose otherwise, and let $\overline{g} \in (S/T) \setminus C_{S/T}(\overline{x})$ be such that $(\overline{g})^2 \in C_{S/T}(\overline{x})$. Let also $\overline{y} = \overline{g} \cdot \overline{x} \cdot (\overline{g})^{-1}$. The following is easy to check:
\begin{enumerate}[(i)]

\item $\overline{x} = \overline{g} \cdot \overline{y} \cdot (\overline{g})^{-1}$;

\item $(\overline{x})^2 = 1 = (\overline{y})^2$; and

\item $\overline{x} \cdot (\overline{y} \overline{x}) \cdot (\overline{x})^{-1} = \overline{x} \overline{y} = (\overline{y} \overline{x})^{-1}$.

\end{enumerate}
Hence the subgroup $\gen{\overline{x}, \overline{y}} \leq S/T$ is isomorphic to a dihedral group of order $2^n$, $D_{2^n}$, for some $n \geq 2$ (where $D_4 \cong \Z/2 \times \Z/2$). In any case, the element $\overline{x}$ is $S/T$-conjugate to some element $\overline{w} \in \gen{\overline{x}, \overline{y}}$ such that $\overline{w} \in C_{S/T}(\overline{x})$. Indeed, if $n = 2$ then $\overline{y} \in C_{S/T}(\overline{x})$ and $\overline{g}$ conjugates $\overline{x}$ to $\overline{y}$. If $n \geq 3$, then $\overline{x}$ is conjugate (within $D_{2^n}$) to the element $\overline{x} \overline{z}$, where $\overline{z}$ is the generator of $Z(D_{2^n})$.

Choose a representative $w \in S$ of the element $\overline{w}$. An easy computation shows that we can choose $w \in C_S(x)$. Indeed, if $w \cdot x \cdot w^{-1} = x \cdot t$ for some $t \in T$, then
$$
(\hat{t} w) \cdot x (\hat{t} w)^{-1} = x,
$$
where $\hat{t} \in T$ is such that $(\hat{t})^2 = t^{-1}$ (and such $\hat{t}$ always exists because $T$ is infinitely $2$-divisible). Choose also some $\alpha \in S$ realizing the conjugation $\overline{w} \mapsto \overline{x}$. Using the same trick as above, we can choose $\alpha$ such that $\alpha \cdot w \cdot \alpha^{-1} = x$. Note that in particular $w$ is $\FF$-conjugate to an element of $T$.

Apply now axiom (II) of saturated fusion systems to the morphism $f \colon \gen{x} \to T$. In particular, $f$ extends to some $\widetilde{f} \colon \gen{x,w} \to S$:
$$
\xymatrix@R=1mm{
\gen{x,w} \ar[rr]^{\widetilde{f}} & & S \\
x \ar@{|->}[rr] & & t \\
w \ar@{|->}[rr] & & \omega\\
}
$$
where $t \in T$ and $\omega \notin T$. We claim that there is a contradiction here.

First note that for each $g \in S$, either $g$ centralizes $T$ or the conjugation action of $g$ on $T$ is the same as the conjugation action of $x$ on $T$. Since $w$ and $\omega$ are both $\FF$-conjugate to $x$ it follows by Lemma \ref{property1} that both $w$ and $\omega$ act nontrivially on $T$, and that $x \cdot w$ centralizes $T$. Now, we have
$$
\widetilde{f}(x \cdot w) = t \cdot \omega,
$$
and $t \cdot \omega$ is $S$-conjugate to $\omega$. Hence $x \cdot w$ is $\FF$-conjugate to an element of $T$, contradicting Lemma \ref{property1}. This proves that $\overline{x}$ acts trivially on $K$.

\begin{itemize}

\item[\textbf{Step 2.}] The subgroup $S_0 \leq S$ is strongly $\FF$-closed.

\end{itemize}

Suppose otherwise and let $\gamma \colon \gen{y} \to S_0$ be a morphism in $\FF$ with $y \in S \setminus S_0$. Note that we may assume that $\gamma(y) \in T$: if $\gamma(y) \notin T$, then $\gamma(y) = x \cdot t$ for some $t$, and there is a sequence of morphisms in $\FF$ with the following effect
$$
\xymatrix{
y \ar@{|->}[rr]^{\gamma} & & \gamma(y) = xt \ar@{|->}[rr]^{c_{t'}} & & x \ar@{|->}[rr]^{f} & & f(x) \in T,
}
$$
for some $t' \in T$. It follows that $y \notin C_S(T)$ by Lemma \ref{property1}. In particular this implies that $y = x \cdot z \cdot t$ for some $z \in C_S(T) \setminus T$ (in particular $z \neq 1$) such that $\overline{z} \in C_{S/T}(\overline{x})$ and some $t \in T$. Furthermore, we can choose $z$ and $t$ such that $y \in C_S(x)$. The same arguments used in Step 1 apply now to show that the existence of $\gamma$ is impossible.

Indeed, by axiom (II) of saturated fusion systems the morphism $f \colon \gen{x} \to T$ extends to some $\widetilde{f} \colon \gen{x,y} \to S$, where $\widetilde{f}(y)$ acts nontrivially on $T$, and we have the following sequence of morphisms in $\FF$
$$
\xymatrix{
x \cdot y \ar@{|->}[rr]^{\widetilde{f}} & & f(x) \cdot \widetilde{f}(y) \ar@{|->}[rr]^{c_{t'}} & & \widetilde{f}(y) \ar@{|->}[rr]^{\widetilde{f}^{-1}} & & y \ar@{|->}[rr]^{\gamma} & & \gamma(y) \in T,
}
$$
for some $t' \in T$ such that $(t')^2 = f(x)$ (such element exists because $f(x) \in T$ and $\widetilde{f}(y)$ acts nontrivially on $T$). Now the above contradicts Lemma \ref{property1} because $xy \in C_S(T)\setminus T$ by hypothesis.

\begin{itemize}

\item[\textbf{Step 3.}] If $|S/T| \geq 4$ then $\g$ is not irreducible.

\end{itemize}

Let $S_0 \leq S$ be as defined above, and let $\FF_0$ be the fusion system over $S_0$ described in Examples \ref{expl2} or \ref{expl3}, depending on the isomorphism type of $S_0$. Let also $W = \gen{x, f(x)} \leq S_0$. Since $x$ is $\FF$-conjugate to an element of $T$, it follows now that
$$
\Aut_{\FF}(W) = \Aut_{\FF_0}(W) = \Aut(W) \qquad \qquad \Aut_{\FF}(S_0) = \Aut_{\FF_0}(S_0) = \Inn(S_0).
$$
Hence $\FF_0 \subseteq \FF$ is a proper fusion subsystem (since $S_0$ is a proper subgroup of $S$), and we claim that $\FF_0$ is normal subsystem of $\FF$.

By Step 2 the subgroup $S_0 \leq S$ is strongly $\FF$-closed. Condition (N2) follows by the above equalities. Saturation of $\FF_0$ was proved in Examples \ref{expl2} and \ref{expl3} (or rather left as an easy exercise to the reader). Finally, condition (N4) follows immediately because $\Aut_{\FF_0}(S_0) = \Inn(S_0)$.
\end{proof}

\begin{rmk}\label{Sylow0rk1}

Let $\g = \ploc$ be a $p$-local compact group of rank $1$. The above case-by-case arguments show in particular the following: there is an (up to isomorphism) unique irreducible $p$-local compact group $\g_0 = (S_0, \FF_0, \LL_0)$ such that $\FF_0$ is normal in $\FF$. Indeed, let $S_0$ be the intersection of all the strongly $\FF$-closed subgroups of $S$ that contain the maximal torus $T$, which is the minimal strongly $\FF$-closed subgroup of $S$ that contains $T$. If $S_0 = T$ then $\g_0$ is the $p$-local compact group induced by $SO(2)$. Otherwise (only if $p = 2$), $S_0$ has the isomorphism type of either $D_{2^{\infty}}$ or $Q_{2^{\infty}}$, and $\g_0$ is the $2$-local compact group described in either Example \ref{expl2} or \ref{expl3}.

\end{rmk}

\begin{defi}\label{defiirrcomp}

Let $\g = \ploc$ be a $p$-local compact group of rank $1$. The \textit{irreducible component} of $\g$ is the $p$-local compact group $\g_0$ in Remark \ref{Sylow0rk1}.

\end{defi}


\section{The proof of Theorem \ref{thmB}}

In this section we study fibrations where the base is the nerve of a finite transporter system associated to a saturated fusion system, and where the fibre is the classifying space of a irreducible rank $1$ $p$-local compact group. The goal is to prove Theorem \ref{thmB}: the total space of such a fibration is, up to $p$-completion, the classifying space of a $p$-local compact group. This is done in Propositions \ref{ThmB1}, \ref{ThmB2} and \ref{ThmB3}, depending on the isomorphism type of the fibre.

Fix then $(\overline{S}, \overline{\FF}, \overline{\LL})$, where $\overline{S}$ is a finite $p$-group, $\overline{\FF}$ is a saturated fusion system over $\overline{S}$, and $\overline{\LL}$ is a transporter system associated to $\overline{\FF}$. Assume in addition that $\Ob(\overline{\LL})$ contains all the $\overline{\FF}$-centric $\overline{\FF}$-radical subgroups of $\overline{S}$, so in particular $|\overline{\LL}|^{\wedge}_p$ is the classifying space of a $p$-local finite group.

\begin{prop}\label{ThmB1}

Let $F \to X \to |\overline{\LL}|$ be a fibration where $F \simeq (BSO(2))^{\wedge}_p$. Then $X^{\wedge}_p$ is the classifying space of a $p$-local compact group.

\end{prop}

\begin{proof}

The natural map $|\overline{\LL}| \to B\underline{\aut}((BSO(2))^{\wedge}_p)$ factors through the following commutative diagram
$$
\xymatrix{
|\overline{\LL}| \ar[r] \ar[d] & B\underline{\aut}((BSO(2))^{\wedge}_p) \ar[d] \\
\pi_1(|\overline{\LL}|) \ar[r] & B\Out((BSO(2))^{\wedge}_p) \\
}
$$
where $\Out((BSO(2))^{\wedge}_p) \simeq \Out(T)$. Let then $\Phi \colon \overline{S} \to \Out(T)$ be the group homomorphism induced by the composition $B\overline{S} \to |\overline{\LL}| \to B\Out(T)$, and let $\overline{S}_1 = \Ker(\Phi)$. As usual, we distinguish the case $p > 2$ from the case $p = 2$.

\begin{itemize}

\item[\textbf{Step 1.}] The case $p > 2$.

\end{itemize}

By Lemma \ref{property2} there is a fibration $BT \Right2{} X_{\ast} \Right2{} |\overline{\LL}|$ whose fibrewise $p$-completion is the original fibration, and this in turn corresponds to an extension of the transporter system $\overline{\LL}$ by the discrete $p$-toral group $T$ in the sense of Definition \ref{extensionT}. By Corollary \ref{extension3} there exists a transporter system $(S, \FF, \TT)$, where
\begin{itemize}

\item $S$ is a discrete $p$-toral group with maximal torus $T$ and $S/T \cong \overline{S}$;

\item $\FF$ is a fusion system over $S$, $T$ is $\FF$-normal and $\FF/T \cong \overline{\FF}$;

\item $\TT$ is a transporter system associated to $\FF$, and $\TT/T \cong \overline{\LL}$; and

\item the space $|\TT|$ is equivalent to $X_{\ast}$.

\end{itemize}
We claim that the extension $T \to \LL \to \overline{\LL}$ is admissible in the sense of Definition \ref{defiadmis}. We have to show that the following condition holds: if $\overline{P} \leq \overline{S}$ is fully $\overline{\FF}$-centralized and $C_{\overline{S}_1}(\overline{P}) \leq \overline{P}$, then $\overline{P} \in \Ob(\overline{\LL})$.

If $p > 2$ then $\overline{S}_1 = \overline{S}$ since $\Out(T)$ does not contain any nontrivial finite $p$-subgroup. Thus the extension is admissible if $\overline{\LL}$ contains all the $\overline{\FF}$-centric $\overline{\FF}$-radical subgroups of $\overline{S}$, which is the case by hypothesis.

\begin{itemize}

\item[\textbf{Step 2.}] The case $p = 2$.

\end{itemize}

In this case, either $\Phi$ is trivial or $\Phi$ surjects onto the $\Z/2$ factor of $\Out(T)$ generated by the reflection $\tau$ described in (\ref{reflection}). If $\Phi$ is trivial, then the above arguments for $p > 2$ apply without modification and the statement follows. Suppose then that $\Im(\Phi) = \gen{\tau} \leq \Out(T)$. In this case, consider the commutative diagram
$$
\xymatrix{
F \ar[r] \ar@{=}[d] & Y \ar[r] \ar[d] & \widetilde{F} \ar[d] \\
F \ar[r] & X \ar[r] \ar[d] & |\overline{\LL}| \ar[d] \\
 & B\Z/2 \ar@{=}[r] & B\Z/2
}
$$
where $\widetilde{F}$ is the homotopy fibre of the map $|\overline{\LL}| \to B\Out(T)$. By  construction, the fibration $F \Right1{} Y \Right1{} \widetilde{F}$ is a nilpotent fibration, and thus
$$
Y^{\wedge}_2 \Right4{} (\widetilde{F})^{\wedge}_2
$$
is a fibration whose fibre is equivalent to $F$ by the Nilpotent Fibration Lemma \cite[II.4.8]{BK} (recall that $F$ is already $2$-complete).

Notice also that $(\widetilde{F})^{\wedge}_2$ is the classifying space of a $2$-local compact group. Indeed, the transporter system $(\overline{S}, \overline{\FF}, \overline{\LL})$ determines a $2$-local finite group $\overline{\g}$, and the fibration
$$
\widetilde{F} \Right4{} |\overline{\LL}| \Right4{} B\Z/2
$$
implies that $(\widetilde{F})^{\wedge}_2$ is the classifying space of a certain $2$-local finite subgroup of $\overline{\g}$ of index $2$, by \cite[Theorem A]{BCGLO2}.

Let then $\widetilde{\g} = (\widetilde{S}, \widetilde{\FF}, \widetilde{\LL})$ be the $2$-local finite group induced by $(\widetilde{F})^{\wedge}_2$, and note that $\widetilde{S} = \overline{S}_1$ by construction. There is a commutative diagram of fibrations
$$
\xymatrix{
F \ar[r] \ar@{=}[d] & Y_{\ast} \ar[d] \ar[r] & |\widetilde{\LL}| \ar[d] \\
F \ar[r] & Y^{\wedge}_2 \ar[r] & (\widetilde{F})^{\wedge}_2
}
$$
and thus by Lemma \ref{property2} there is a fibration
$$
BT \Right4{} Y'_{\ast} \Right4{} |\widetilde{\LL}|.
$$
Since $\widetilde{S}$ acts trivially on $T$ by construction, it follows that the above fibration is the realization of an admissible extension of transporter systems, and thus
$$
(Y'_{\ast})^{\wedge}_2 \simeq (Y_{\ast})^{\wedge}_2 \simeq Y^{\wedge}_2
$$
is the classifying space of a $2$-local compact group. The proof is finished then by \cite[Theorem 7.3]{BLO6} applied to the fibrewise $2$-completion of the fibration $Y \to X \to B\Z/2$.
\end{proof}

\begin{prop}\label{ThmB2}

Let $F \to X \to |\overline{\LL}|$ be a fibration where $F \simeq (BSO(3))^{\wedge}_2$. Then $X^{\wedge}_2$ is the classifying space of a $2$-local compact group.

\end{prop}

\begin{proof}

In this case, the map $\pi \colon |\overline{\LL}| \to B\underline{\aut}(F^{\wedge}_2) \to B\Out(F^{\wedge}_2)$ is nulhomotopic. Indeed, by Proposition \ref{AutS} we know that $\Out(\F^{\wedge}_2) \cong \dosadic$ and the map $\pi$ must be nulhomotopic because $\overline{\LL}$ is a finite category. By Corollary \ref{autG2} this means that the map $|\overline{\LL}| \to B\underline{\aut}(F^{\wedge}_2)$ is also nulhomotopic, and hence
$$
X \simeq F \times |\overline{\LL}|.
$$
The statement follows by Proposition \ref{product}.
\end{proof}

\begin{prop}\label{ThmB3}

Let $F \to X \to |\overline{\LL}|$ be a fibration where $F \simeq (BSU(2))^{\wedge}_2$. Then $X^{\wedge}_2$ is the classifying space of a $2$-local compact group.

\end{prop}

\begin{proof}

Let $\g = \ploc$ be the $2$-local compact group described in Example \ref{expl3}, so $B\g \simeq F$, and recall that $S$ contains a single element of order $2$, namely $t_1 \in T$. Hence $\gen{t_1} \leq S$ is an $\FF$-central subgroup, and the quotient $(S/\gen{t_1}, \FF/\gen{t_1}, \LL/\gen{t_1})$ is the $2$-local compact group induced by $SO(3)$. Furthermore, the fibration $B\Z/2 \to |\LL| \to |\LL/\gen{t_1}|$ can be completed to a fibration
$$
B\Z/2 \Right4{} |\LL|^{\wedge}_2 \Right4{} |\LL/\gen{t_1}|^{\wedge}_2.
$$

Since every map $B\Z/2 \to B\g$ factors through $BS \to B\g$ up to homotopy, it follows now that $\Map(B\Z/2, B\g)$ has only two irreducible components, corresponding to the trivial map and to the inclusion map $\iota \colon B\gen{t_1} \to BT \to B\g$.

Actually, the map $\iota$ factors as
$$
\xymatrix{
B\gen{t_1} \ar[rr]^{\iota} \ar@{.>}[rrd] & & B\g \\
 & & BSU(2) \ar[u]_{(-)^{\wedge}_2}
}
$$
where the map $B\gen{t_1} \Right2{} BSU(2)$, which by abuse of notation we denote also by $\iota$, is induced by the inclusion of the center of $SU(2)$. Hence we have
$$
\Map(B\Z/2, B\g) \simeq \Map(B\Z/2, B\g)_{\triv} \coprod \Map(B\Z/2, B\g)_{\iota} \simeq B\g \coprod B\g,
$$
where $\Map(B\Z/2, B\g)_{\triv} \simeq B\g$ by \cite[Theorem 6.3 (c)]{BLO3} and $\Map(B\Z/2, B\g)_{\iota} \simeq B\g$ since $\Map(B\Z/2, BSU(2))_{\iota} \simeq BSU(2)$.

As a consequence, if $\gamma \colon B\Z/2 \to X$ is such that the composition $B\Z/2 \to X \to |\overline{\LL}|$ is nulhomotopic, but $\gamma$ itself is not nulhomotopic, then $\gamma$ lifts to $\iota$ (up to homotopy):
$$
\xymatrix{
 & B\Z/2 \ar[d]^{\gamma} \ar@{.>}[ld]_{\iota} \ar[rd]^{\simeq \ast} & \\
F \ar[r] & X \ar[r] & |\overline{\LL}|.
}
$$
There is then a fibration $\Map(B\Z/2,F)_{\iota} \to \Map(B\Z/2,X)_{\gamma} \to \Map(B\Z/2, |\overline{\LL}|)_{\triv}$, and it follows that $\Map(B\Z/2, X)_{\gamma} \simeq X$.

Consider then the composition $\gamma \colon B\Z/2 \Right2{\iota} F \Right2{} X$, which is easily seen not to be nulhomotopic. By the above $\Map(B\Z/2, X)_{\gamma} \simeq X$, and we can define the \textit{quotient} of $X$ by the map $\gamma$ as follows. We have
$$
B\Z/2 \times X \simeq B\Z/2 \times \Map(B\Z/2, X)_{\iota} \Right4{\ev} X
$$
and hence we have an action of $B\Z/2$ on $X$. Set then $Y$ to be the Borel construction of this action, $Y \defin EB\Z/2 \times_{B\Z/2} X$. There is a fibration sequence $B\Z/2 \Right1{\iota} X \Right1{} Y \Right1{} B^2\Z/2$, and in particular $Y$ is seen as the quotient of $X$ by $\iota$. Furthermore, there is a homotopy commutative diagram of fibrations
$$
\xymatrix{
B\Z/2 \ar@{=}[r] \ar[d] & B\Z/2 \ar[d] & \\
F \ar[r] \ar[d] & X \ar[r] \ar[d] & |\overline{\LL}| \ar@{=}[d] \\
\overline{F} \ar[r] & Y \ar[r] & |\overline{\LL}|.
}
$$
where $\overline{F} \simeq |\LL/\gen{t_1}|^{\wedge}_2 \simeq (BSO(3))^{\wedge}_2$.

The space $Y$ is up to $2$-completion the classifying space of a $2$-local compact group by Proposition \ref{ThmB2}, and then the statement follows for $X$ because it is a central extension of $2$-local compact groups.
\end{proof}


\section{The proof of Theorem \ref{thmC}}\label{THMC}

In this section we describe how each $p$-local compact group of rank $1$ gives rise to a fibration over a certain finite transporter systems with fibre the classifying space of the corresponding irreducible component.

Let then $\g = \ploc$ be a $p$-local compact group of rank $1$, and let $\g_0 = (S_0, \FF_0, \slocl_0)$ be its irreducible component. In particular, $S_0 \leq S$ is an strongly $\FF$-closed subgroup, and we can consider the normalizer of $S_0$ in $\g$, $N_{\g}(S_0) = (S, N_{\FF}(S_0), N_{\LL}(S_0))$, as defined in \ref{definorm}, and the quotient transporter system, which we denote by $\overline{\g} = (\overline{S}, \overline{\FF}, \overline{\LL})$ to simplify the notation (so $\overline{S} = S/S_0$, $\overline{\FF} = \FF/S_0$ and $\overline{\LL} = \LL/S_0$ in the notation of \ref{defiquotient}).

\begin{prop}

The transporter system $\overline{\LL}$ contains all the $\overline{\FF}$-centric subgroups.

\end{prop}

\begin{proof}

Suppose otherwise, and let $\overline{P} \leq \overline{S}$ be an $\overline{\FF}$-centric subgroup which is not contained in $\Ob(\overline{\LL})$. Let also $P \leq S$ be the preimage of $\overline{P}$. Then by construction $P$ is not $\FF$-centric. We may assume that $P$ is fully $\FF$-centralized, and thus $Z(P) \lneqq C_S(P)$. However, this would imply that $Z(\overline{P}) \lneqq C_{\overline{S}}(\overline{P})$, contradicting the centricity of $\overline{P}$, and thus a contradiction.
\end{proof}

The proof of Theorem \ref{thmC} is done, as usual, via a case-by-case argument, depending on the isomorphism type of $\g_0$. This way, Theorem \ref{thmC} is proved in Proposition \ref{ThmC} when $\g_0$ is induced by $SO(2)$, Proposition \ref{ThmC2}, when $\g_0$ is induced by $SO(3)$, and Proposition \ref{ThmC3} when $\g_0$ is induced by $SU(2)$. In this sense, the case of $SO(2)$ poses no difficulties, the case of $SO(3)$ takes most of the work in this section, while the case of $SU(2)$ follows easily from the case of $SO(3)$ once it has been solved.

\begin{prop}\label{ThmC}

Suppose that $\g_0$ is induced by $SO(2)$. Then there is a fibration $F \to X \to |\overline{\LL}|$, where $X^{\wedge}_p \simeq B\g$ and $F \simeq B\g_0$.

\end{prop}

\begin{proof}

Since $S_0 = T$ is strongly $\FF$-closed, it follows by Lemma \ref{property3} that $S_0$ is $\FF$-normal. Hence $N_{\g}(S_0) = \g$ and there is an extension of transporter systems
$$
S_0 \Right4{} \LL \Right4{} \overline{\LL}.
$$
The fibration $F \to X \to |\overline{\LL}|$ is the fibrewise $p$-completion of the realization of the above extension.
\end{proof}

In particular the above result proves Theorem \ref{thmC} for all odd primes, and thus we may assume that $p = 2$ for the rest of this section. Next we prove Theorem \ref{thmC} when $\g$ is a $2$-local compact group whose irreducible component is induced by $SO(3)$.

Fix then a $2$-local compact group $\g = \ploc$ whose irreducible component $\g_0 = (S_0, \FF_0, \LL_0)$ is induced by $SO(3)$. In order prove Theorem \ref{thmC} in this case, first we analyze in great detail the relationship between $\g$ and $\g_0$. Each statement from this point to the end of this section applies only to this situation unless otherwise specified. For the sake of a better reading let us recall the notation of Example \ref{expl2} for $\g_0$. First, we have
$$
S_0 = \gen{\{t_n\}_{n \geq 1}, x \,\, | \,\, \forall n: \,\, t_n^{2^n} = x^2 = 1, \,\, t_{n+1}^2 = t_n, \,\, x \cdot t_n \cdot x^{-1} = t_n^{-1} } \cong D_{2^{\infty}}.
$$
We set then $T = \gen{\{t_n\}_{n \geq 1}} \cong \prufferdos$ and $V = \gen{x,t_1} \cong \Z/2 \times \Z/2$. Note that $N_{S_0}(V) = \gen{x, t_2}$ is a dihedral group of order $8$. The subgroups $S_0$ and $V$ are representatives of the only two conjugacy classes of $\FF_0$-centric $\FF_0$-radical subgroups, with
$$
\Aut_{\FF_0}(S_0) = \Inn(S_0) \qquad \mbox{and} \qquad \Aut_{\FF_0}(V) = \Aut(V) \cong \Sigma_3.
$$
More specifically, we fix the following presentation of $\Aut_{\FF_0}(V)$:
$$
\Aut_{\FF_0}(V) = \gen{c_{t_2}, f_0 \, | \, (c_{t_2})^2 = \Id = f_0^3, \, c_{t_2} \circ f_0 \circ (c_{t_2})^{-1} = f_0^{-1}}.
$$
By inspection we deduce that if $P \leq S_0$ is $\FF_0$-centric but not $\FF_0$-radical then either
\begin{itemize}

\item $P = T$; or

\item $P$ is a finite dihedral group of order greater or equal than $8$.

\end{itemize}
Let $\LL_0^{cr} \subseteq \LL_0$ be the full subcategory of $\FF_0$-centric $\FF_0$-radical subgroups. We have
$$
\Aut_{\LL_0^{cr}}(S_0) \cong S_0 \qquad \qquad \Aut_{\LL_0^{cr}}(V) \cong \Sigma_4.
$$




\begin{lmm}\label{SO3-1}

If $P \leq S$ is $\FF$-centric $\FF$-radical then $P \cap S_0$ is $\FF_0$-centric $\FF_0$-radical.

\end{lmm}

\begin{proof}

For simplicity assume that $P$ is also fully $\FF$-normalized. By \cite[Lemma 3.5]{BCGLO2} we already know that $P \cap S_0$ is $\FF_0$-centric, and we prove the second part of the statement by contradiction, so suppose that $P \cap S_0$ is not $\FF_0$-radical.

In particular $P \cap S_0 \lneqq S_0$, so either $P \cap S_0 = T$ or $P \cap S_0$ is a finite dihedral group and $t_2 \in P \cap S_0$. Suppose first that $P \cap S_0 = T$, and let $\widetilde{P} \leq S$ be the preimage of $P/T$. Then $P \leq \widetilde{P}$ is a subgroup of index $2$, and in particular it is a normal subgroup. Furthermore, $S_0 \leq \widetilde{P} \leq N_S(P)$, and conjugation by $x \in S_0$ induces a nontrivial automorphism of $P$ which is not inner.

The exact sequence of \cite[2.8.7]{Suzuki} has the form
$$
\{1\} \to K_P \Right3{} \Aut_{\FF}(P) \Right3{} \Aut_{\FF}(T) \times \Aut(S/S_0),
$$
for some $2$-subgroup $K_P \leq \Aut_{\FF}(P)$, and where $\Aut_{\FF}(T) = \gen{c_x}$. Since $P$ is $\FF$-centric $\FF$-radical we know that $K_P \leq \Inn(P)$, and it follows that the subgroup $\gen{K_P, c_x} \leq \Aut_{\FF}(P)$ is normal, which contradicts the hypothesis that $P$ is $\FF$-radical.

Suppose then that $P \cap S_0 \cong D_{2^n}$ for some $n \geq 3$. In this case it is easy to check that $t_n \in N_S(P) \setminus P$. Since $P/(P \cap S_0)$ acts trivially on $T$ it also follows easily that $\gen{c_{t_n}} \leq \Aut_{\FF}(P)$ is a normal subgroup, and this again contradicts the radicality of $P$.
\end{proof}

\begin{lmm}\label{SO3-11}

Let $P \leq S$ be such that $T_n \leq P$ for some $n \geq 2$. Then, for each $f \in \Hom_{\FF}(P,S)$ there is some $\widetilde{f} \in \Hom_{\FF}(PT, S)$ such that $f = \widetilde{f}|_P$.

\end{lmm}

\begin{proof}

Every morphism $f \in \Hom_{\FF}(P,S)$ is a composition of restrictions of automorphisms of $\FF$-centric $\FF$-radical subgroups by Alperin's fusion theorem \cite[Theorem 3.6]{BLO3}. The statement follows then by Lemma \ref{SO3-1}.
\end{proof}

\begin{rmk}\label{rmkSO3-1}

Recall that every $H \in \FF_0^{cr}$ is fully normalized in both $\FF_0$ and $\FF$. Thus the following holds.

\begin{enumerate}[(i)]

\item If $P \in \FF^{cr}$ then $P \in N_{\FF}(P \cap S_0)^{cr}$ and
$$
\Aut_{N_{\LL}(P \cap S_0)}(P) = \Aut_{\LL}(P).
$$
This follows because $\Aut_{N_{\FF}(P \cap S_0)}(P) = \Aut_{\FF}(P)$ by definition of $N_{\FF}(P \cap S_0)$, and because $N_{N_S(P \cap S_0)}(P) = N_S(P)$ by Remark \ref{rmknorm}.

\item If $H \in \FF_0^{cr}$ and $P \in N_{\FF}(H)^c$ is such that $P \cap S_0 = H$ then $P \in \FF^c$, and
$$
\Aut_{N_{\LL}(H)}(P) = \Aut_{\LL}(P).
$$
This follows because $P$ is $\FF$-centric by Remark \ref{rmknorm2}, and $\Aut_{N_{\FF}(H)}(P) = \Aut_{\FF}(P)$ by definition of $N_{\FF}(H)$.

\end{enumerate}
In particular, in order to describe $\g$ it is enough to describe the normalizer $2$-local compact groups $N_{\g}(S_0)$ and $N_{\g}(V)$, since every centric radical subgroup in $\FF$ is, up to $S$-conjugacy, centric radical in one of these normalizers.

\end{rmk}

Let $N_{\g}^S(V) = (N_S(V), N_{\FF}^S(V), N_{\LL}^S(V))$ be the $K$-normalizer $p$-local finite group of $V$ in $\g$, for $K = \Aut_S(V)$, as defined in \ref{definorm}. The idea to prove Proposition \ref{ThmC2} is to extend the homology decomposition 
$$
B\Sigma_4 = B\Aut_{\LL_0}(V) \Left4{} BD_8 = BN_{S_0}(V) \Right4{} B\Aut_{\LL_0}(S_0) = BD_{2^{\infty}}
$$
of $B\g_0$ described in \cite[Corollary 4.2]{DMW} to a whole (homotopy) commutative diagram of fibrations
\begin{equation}\label{diagram1}
\xymatrix{
B\Aut_{\LL_0}(V) \ar[d] & BN_{S_0}(V) \ar[r] \ar[l] \ar[d] & B\Aut_{\LL_0}(S_0) \ar[d] \\
|N_{\LL}(V)| \ar[d] & |N_{\LL}^S(V)| \ar[r] \ar[l] \ar[d] & |N_{\LL}(S_0)| \ar[d] \\
|\overline{\LL}| & |\overline{\LL}| \ar@{=}[r] \ar@{=}[l] & |\overline{\LL}|
}
\end{equation}
Roughly speaking, the middle and rightmost fibrations arise naturally, since $S_0$ and $N_{S_0}(V)$ are normal in $N_{\FF}(S_0)$ and $N_{\FF}^S(V)$ respectively. The leftmost fibration requires a deeper understanding of $N_{\g}(V)$.

\begin{lmm}\label{SO3-12}

Let $P \in \Ob(N_{\LL}(S_0))$ be such that $S_0 \leq P$. Then, $C_P(V)$ is $\FF$-centric.

\end{lmm}

\begin{proof}

We have to show that $C_S(Q) \leq Z(Q)$ for each $Q \in C_P(V)^{\FF}$ which is fully $\FF$-normalized. Thus, fix some $Q \in C_P(V)^{\FF}$ as above, and let $f \in \Hom_{\FF}(N_S(C_P(V)), S)$ be such that $f(C_P(V)) = Q$. For simplicity, we may assume that $Q \cap S_0 = V$ and $f(V) = V$. The following holds:
\begin{enumerate}[(a)]

\item $P = C_P(V)T$. Indeed, $C_P(V)T \leq P$ by definition, and if $h \in P \setminus C_P(V)$ then there is some $t \in T$ such that $ht \in C_P(V)$.

\item There is some $\widetilde{f} \in \Hom_{\FF}(N_S(C_P(V))T, S)$ such that $\widetilde{f}|_{N_S(C_P(V))} = f$. This is because $N_S(C_P(V)) \cap S_0 = N_{S_0}(V)$, so $T_2 \leq N_S(C_P(V))$, and thus Lemma \ref{SO3-11} applies.

\end{enumerate}
Set then $R = \widetilde{f}(P)$. We have $Q = C_R(V)$ and $R = QT$.

Finally, let $g \in C_S(Q)$. Since $V \leq Z(Q)$ by definition, we may choose the element $g$ such that $g \in C_S(T)$ (otherwise replace $g$ by $gx$), and thus $g \in C_S(QT) = C_S(R) = Z(R)$, since $R$ is $\FF$-centric. In particular, $g \in C_R(V) = Q$, and the statement follows.
\end{proof}












\begin{lmm}\label{SO3-24}

There is an isomorphism $\Aut_{\LL}(C_S(V)) \cong \Aut_{\LL_0}(V) \times \Aut_{\overline{\LL}}(S/S_0)$. In particular, $C_S(V) \in N_{\FF}(V)^{cr}$.

\end{lmm}

\begin{proof}

Set for short $C = C_S(V)$. A first observation is that $C$ is $N_{\FF}(V)$-centric. Indeed, $V$ is $N_{\FF}(V)$-normal by definition, and hence $C$ is weakly $N_{\FF}(V)$-closed. Also, $N_S(V) = \gen{C, t_2}$, with $t_2 \notin C_{N_S(V)}(C)$, and the claim follows.

Another easy observation is that $C/V = S/S_0$. Clearly, $C \cap S_0 = V$, and thus we have $C/V = C/(C\cap S_0) \leq S/S_0$. Conversely, if $g \in S$ represents some class in $S/S_0$ then there is some $t \in T$ such that $gt \in C_S(V)$, and the claim follows.

By Remark \ref{rmkSO3-1} we have $\Aut_{N_{\FF}(V)}(C) = \Aut_{\FF}(C)$. Consider then the exact sequence induced by \cite[2.8.7]{Suzuki}
$$
\{1\} \to K_C \Right3{} \Aut_{\FF}(C) \Right3{\Phi} \Aut_{\FF}(V) \times \Aut(C/V),
$$
where $K_C$ is a $2$-group. Since $N_S(V) = \gen{C, t_2}$, it follows that $K_C \leq \Inn(C)$ (conjugation by $t_2$ cannot be an element of $K_C$). The rest of the proof is divided into steps.



\begin{itemize}

\item[\textbf{Step 1.}] The restriction $\res^C_V \colon \Aut_{\FF}(C) \to \Aut_{\FF}(V)$ has a section.

\end{itemize}

Let $f_0 \in \Aut_{\FF}(V)$ be an element of order $3$, and let $f \in \Aut_{\FF}(C)$ be such that $f|_V = f_0$. Consider also $\omega = c_{t_2} \in \Aut_{\FF}(C)$. We have
$$
\Phi(f) = (f_0, \overline{f}) \qquad \qquad \Phi(\omega) = (c_{t_2}, \Id)
$$
for some $\overline{f} \in \Aut(C/V)$. The second equality holds since $C/V$ acts trivially on $T$. Define then $\gamma = \omega \circ f \circ \omega^{-1} \circ f^{-1} \in \Aut_{\FF}(C)$. We have
$$
\Phi(\gamma) = (c_{t_2} \circ f_0 \circ c_{t_2}^{-1} \circ f_0^{-1}, \overline{f} \circ (\overline{f})^{-1}) = (f_0, \Id).
$$
In particular, $\gamma^3 \in \Ker(\phi) = K_C$, which is a $2$-group, and it follows that $\gen{\gamma} \leq \Aut_{\FF}(C)$ has a subgroup of order $3$. We may assume that $\gamma$ is of order $3$ for simplicity.

We have to check that $\gen{\gamma, \omega} \leq \Aut_{\FF}(C)$ has the right isomorphism type. Essentially we have to show that $\omega \circ \gamma \circ \omega^{-1} = \gamma^{-1}$. Since $\gamma = \omega \circ f \circ \omega^{-1} \circ f^{-1}$ and $\omega$ has order $2$, it follows that
$$
\omega \circ \gamma \circ \omega^{-1} = \omega \circ (\omega \circ f \circ \omega^{-1} \circ f^{-1}) \circ \omega^{-1} = \gamma^{-1}.
$$
Hence, the morphism $\sigma_C \colon \Aut_{\FF}(V) \to \Aut_{\FF}(C)$ defined by $\sigma_C(c_{t_2}) = \omega$ and $\sigma_C(f_0) = \gamma$ is a section of the restriction homomorphism.

\begin{itemize}

\item[\textbf{Step 2.}] The subgroup $K_C \leq \Aut_{\FF}(C)$ is trivial.

\end{itemize}

Suppose otherwise that $K_C \neq \{\Id\}$, and let $c_k \in K_C$ be such that $c_k \neq \Id$ for some $k \in C$. In particular there is some $g \in C_C(T)$ such that
$$
c_k(g) = g t_1.
$$
Indeed, $c_k \in K_C$ implies that $c_k(g) \in gV$ for all $g \in C$. Since we are assuming that $c_k \neq \Id$ there must be some $g \in C$ such that $c_k(g) \neq g$, and hence $c_k(g) = gt_1$ since $C_S(T) \lhd N_S(T)$. If $g \notin C_C(T)$, then $gx \in C_C(T)$ is such that $c_k(gx) = (gx)t_1$.

Let now $\sigma_C \colon \Aut_{\FF}(V) \to \Aut_{\FF}(C)$ be the section to $\res^C_V$ constructed in Step 1, and recall that every element of $\Im(\sigma_C)$ induces the identity on $C/V$. In particular, if $f \in \Im(\sigma_C)$ is such that $f(t_1) = x$ and $g \in C_C(T)$ is such that $c_k(g) = gt_1$, then
$$
c_{f(k)}(f(g)) = (f \circ c_k \circ f^{-1})(f(g)) = f(g)x,
$$
which is impossible: if $f(g) \in C_C(T)$ then $f(g)x \notin C_C(T)$, and if $f(g) \notin C_C(T)$ then $f(g)x \in C_C(T)$. Hence, $K_C = \{\Id\}$. As a consequence, $\Aut_{\FF}(C) \cong \Im(\Phi) = \Aut_{\FF}(V) \times A$ for some $A \leq \Aut(C/V)$. In particular, $\Im(\sigma_C) \lhd \Aut_{\FF}(V)$.

For simplicity let us identify $\Aut_{\FF}(C)$ with $\Aut_{\FF}(V) \times A$. Then, it is not difficult to see that every $f \in A \leq \Aut_{\FF}(C)$ extends to an automorphism of $C\cdot T = S$, and thus $A = \Aut_{\overline{\FF}}(S/S_0)$. In particular, $C$ must be $N_{\FF}(V)$-centric $N_{\FF}(V)$-radical, since both $\FF_0$ and $\overline{\FF}$ are saturated fusion systems.

\begin{itemize}

\item[\textbf{Step 3.}] $\Aut_{\LL}(C)$ contains $\Aut_{\LL_0}(V)$ as a normal subgroup.

\end{itemize}

Consider the exact sequence $\{1\} \to Z(C) \Right3{} \Aut_{\LL}(C) \Right3{} \Aut_{\FF}(C) \to \{1\}$, and let $B \leq \Aut_{\LL}(C)$ be the pull-back of the above sequence with $\sigma_C \colon \Aut_{\FF}(V) \to \Aut_{\FF}(C)$, the section of $\res^C_V$ constructed in Step 1. Let $f \in \Im(\sigma_C)$ be such that
$$
f(x) = t_1 \qquad f(t_1) = xt_1 \qquad f(xt_1) = x.
$$
Then there is some element $\varphi \in B$ of order $3$ such that $\rho(\varphi) = f$, and we claim that
$$
\Aut_{\LL_0}(V) \cong \gen{\varphi, N_{S_0}(V)} \lhd \Aut_{\LL}(C).
$$
The proof of $\Aut_{\LL_0}(V) \cong \gen{\varphi, N_{S_0}(V)}$ is left to the reader as an easy exercise: there is an exact sequence $V \to \gen{\varphi, N_{S_0}(V)} \to \Aut_{\FF}(V)$, and with this the claim follows easily.


Let us prove that $\gen{\varphi, N_{S_0}(V)} \lhd \Aut_{\LL}(C)$. Clearly, $\psi \circ \varepsilon(y) \circ \psi^{-1} \in \gen{\varphi, N_{S_0}(V)}$ for all $\psi \in \Aut_{\LL}(C)$ and all $y \in N_{S_0}(V)$. Also, by Step 2 it follows that
$$
\psi \circ \varphi \circ \psi^{-1} = \varepsilon(z) \circ \varphi
$$
for some $z \in Z(C)$. Equivalently, $\varphi \circ \psi^{-1} \circ \varphi^{-1} = \psi^{-1} \circ \varepsilon(z)$, and then
$$
\psi^{-1} = \varphi^3 \circ \psi^{-1} \circ \varphi^{-3} = \psi^{-1} \circ \varepsilon(z) \circ \varepsilon(f(z)) \circ \varepsilon(f^2(z)).
$$
This in turn implies that $z \cdot f(z) \cdot f^2(z) = 1$, which holds only when $z \in V \leq Z(C)$ since $f$ induces the identity on $C/V$. Thus $\gen{\varphi, N_{S_0}(V)}$ is normal in $\Aut_{\LL}(C)$.

\begin{itemize}

\item[\textbf{Step 4.}] The isomorphism type of $\Aut_{\LL}(C)$.

\end{itemize}

By Step 2, $Z(C/V) = Z(C)/V$, and there is a commutative diagram
$$
\xymatrix{
V \ar[r] \ar[d] & Z(C) \ar[r] \ar[d] & Z(C/V) \ar[d] \\
\Aut_{\LL_0}(V) \ar[r] \ar[d] & \Aut_{\LL}(C) \ar[r] \ar[d] & \Aut_{\LL}(C)/\Aut_{\LL_0}(V) \ar[d] \\
\Aut_{\FF_0}(V) \ar[r] & \Aut_{\FF}(C) \ar[r] & \Aut_{\overline{\FF}}(C/V)
}
$$
It is not difficult to see that $\Aut_{\LL}(C)/\Aut_{\LL_0}(V) \cong \Aut_{\overline{\LL}}(S/S_0)$. Now, $\Aut_{\LL_0}(V) \cong \Sigma_4$, which is a complete group (it is centerless, and all of its automorphisms are inner), and it follows that $\Aut_{\LL}(C) \cong \Aut_{\LL_0}(V) \times \Aut_{\overline{\LL}}(S/S_0)$.
\end{proof}

\begin{cor}

The are isomorphisms
$$
N_S(V) \cong N_{S_0}(V) \times S/S_0 \qquad \mbox{and} \qquad C_S(V) \cong V \times S/S_0.
$$

\end{cor}

\begin{proof}

This is a straightforward consequence of Lemma \ref{SO3-24} above.
\end{proof}

\begin{lmm}\label{SO3-23}

Let $N(V) \subseteq N_{\LL}(V)$ be the full subcategory of $N_{\LL}(V)$ with object set $\Ob(N(V)) = \{P \in \Ob(N_{\LL}(V)) \,\, | \,\, P \leq C_S(V)\}$. Then the following holds.
\begin{enumerate}[(i)]

\item If $P \in N(V)$ then $\Aut_{N_{\LL}(V)}(P) = \Aut_{\LL}(P) \cong \Aut_{\LL_0}(V) \times \Aut_{\overline{\LL}}(P/V)$.

\item If $R \in N_{\FF}(V)^{cr}$ then $C_R(V) \in N(V)$ and there is a monomorphism
$$
\Res^R_{C_R(V)} \colon \Aut_{N_{\LL}(V)}(R) \Right3{} \Aut_{N_{\LL}(V)}(C_R(V)).
$$

\end{enumerate}
In particular, there is an equivalence $|N(V)| \simeq |N_{\LL}(V)|$.

\end{lmm}

\begin{proof}

To prove (i) fix some $P \in N(V)$, which we may assume to be fully $N_{\FF}(V)$-normalized. Then $P \cong V \times P/V$. If we consider the exact sequence induced by \cite[2.8.7]{Suzuki},
$$
\{1\} \to K_P \Right3{} \Aut_{\FF}(P) \Right3{} \Aut_{\FF}(V) \times \Aut(P/V),
$$
then clearly $K_P \leq \Inn(P)$. By restriction from $\Aut_{\FF}(C_S(V))$, we get a commutative diagram
$$
\xymatrix{
\Aut_{\FF}(V) \ar[rr]^{\sigma_P} \ar[d]_{\sigma_C} & & \Aut_{\FF}(P) \ar[d] \\
\Aut_{\FF}(C_S(V)) \ar[rr]_{\res^{C_S(V)}_P} & & \Hom_{N_{\FF}(V)}(P, C_S(V))
}
$$
where $\sigma_C \colon \Aut_{\FF_0}(V) \to \Aut_{\FF}(C_S(V))$ is the section to $\res^{C_S(V)}_V$ constructed in the proof of Lemma \ref{SO3-24}, and $\sigma_P \colon \Aut_{\FF_0}(V) \to \Aut_{\FF}(P)$ is a section of $\res^P_V$.

Now Steps 2, 3 and 4 in the proof of Lemma \ref{SO3-24} apply to $P$ (just replace the subgroup $C_S(V)$ by the subgroup $P$ everywhere), and it follows that
$$
\Aut_{N_{\LL}(V)}(P) = \Aut_{\LL}(P) \cong \Aut_{\LL_0}(V) \times \Aut_{\overline{\LL}}(P/V) = \Aut_{\LL_0}(V) \times \Aut_{\overline{\LL}}(PT/S_0).
$$

To prove (ii) fix some $R \in N_{\FF}(V)^{cr}$, with $R \cong (R \cap S_0) \times R/(R \cap S_0)$. Since $V$ is $N_{\FF}(V)$-normal by definition it follows that $V \leq R$. If $R \cap S_0 = V$ then clearly $C_R(V) = R$ and there is nothing to check. Suppose then that $R \cap S_0 = N_{S_0}(V)$.

We may assume that $C_R(V)$ is fully $N_{\FF}(V)$-normalized. Indeed, suppose it is not, and let $Q \in C_R(V)^{N_{\FF}(V)}$ be fully $N_{\FF}(V)$-normalized. Choose also some homomorphism $f \in \Hom_{N_{\FF}(V)}(N_{N_S(V)}(C_R(V)), N_S(V))$ such that $f(C_R(V)) = Q$. Then we may replace $R$ and $C_R(V)$ by $\widetilde{Q} = f(R)$ and $Q$ respectively. By construction $Q = C_{\widetilde{Q}}(V)$.

Now, if $C_R(V)$ is not $N_{\FF}(V)$-centric, then there is some $g \in C_{N_S(V)}(C_R(V))\setminus Z(C_R(V))$. We may choose $g \in C_{N_S(V)}(T)$ (otherwise replace $g$ by $gx$), and then $g \in C_{N_S(V)}(R) = Z(R)$, since $R = \gen{C_R(V), t_2}$. But this implies that $g = yt_2$ for some $y \in C_R(V)$, and this is impossible since $yt_2 \notin C_R(V)$. Thus $C_R(V) \in N(V)$.

By \cite[Lemma 1.10]{BLO3} there is a restriction homomorphism
$$
\Res^R_{C_R(V)} \colon \Aut_{N_{\LL}(V)}(R) \Right4{} \Aut_{N_{\LL}(V)}(C_R(V)).
$$
We have to show that it is a monomorphism. If $R = C_R(V)$ then there is nothing to check, so suppose that $R = \gen{C_R(V), t_2}$.

First we claim that $N_{N_S(V)}(R) = N_{N_S(V)}(C_R(V))$. The inclusion $N_{N_S(V)}(R) \leq N_{N_S(V)}(C_R(V))$ follows since $V$ is $N_{\FF}(V)$-normal, and the inclusion $N_{N_S(V)}(R) \geq N_{N_S(V)}(C_R(V))$ follows since $\gen{t_2} \leq N_S(V)$ is weakly $N_{\FF}(V)$-closed. As a consequence there is a commutative diagram
$$
\xymatrix{
N_{N_S(V)}(R) \ar@{=}[rr] \ar[d]_{\varepsilon_R} & & N_{N_S(V)}(C_R(V)) \ar[d]^{\varepsilon_{C_R(V)}} \\
\Aut_{N_{\LL}(V)}(R) \ar[rr]_{\Res^R_{C_R(V)}} & & \Aut_{N_{\LL}(V)}(C_R(V))
}
$$

Let now $\varphi \in \Ker(\Res^R_{C_R(V)})$. This means in particular that $\rho(\varphi)|_{C_R(V)} = \Id \in \Aut_{\FF}(C_R(V))$. Since $R = \gen{C_R(V), t_2}$ and $\gen{t_2} \leq N_S(V)$ is weakly $N_{\FF}(V)$-closed, then
$$
\rho(\varphi)(t_2) = t_2 \qquad \mbox{or} \qquad \rho(\varphi)(t_2) = t_2^{-1} = c_x(t_2).
$$
In either case $\varphi \in \varepsilon_R(N_{N_S(V)}(R))$, and the above commutative diagram implies that $\varphi$ is the trivial element in $\Aut_{N_{\LL}(V)}(R)$.

Finally, the inclusion functor $N(V) \to N_{\LL}(V)$ induces an equivalence between the corresponding realization of nerves by \cite[Theorem A]{Quillen} and \cite[Theorem B]{BCGLO1}.
\end{proof}

\begin{prop}\label{ThmC2}

Suppose that $g_0$ is induced by $SO(3)$. Then there is a fibration $F \to X \to |\overline{\LL}|$, where $X^{\wedge}_2 \simeq B\g$ and $F \simeq B\g_0$.

\end{prop}

\begin{proof}

Consider the normalizer $2$-local compact groups $N_{\g}(S_0)$, $N_{\g}(V)$ and $N_{\g}^S(V)$ (see Remark (\ref{rmknorm2}), and notice that, by definition $N_{\LL}^S(V) \subseteq N_{\LL}(V)$). Let also $N(V) \subseteq N_{\LL}(V)$ be the full subcategory defined in Lemma \ref{SO3-23}, and let $N^S(V) \subseteq N_{\LL}^S(V)$ be the full subcategory with object set $\Ob(N^S(V)) = \Ob(N(V))$. In particular, there are equivalences
$$
|N(V)| \simeq |N_{\LL}(V)| \qquad  \mbox{and} \qquad |N^S(V)| \simeq |N_{\LL}^S(V)|.
$$
The proof is divided into several steps.

\begin{itemize}

\item[\textbf{Step 1.}] There are faithful functors
$$
\xymatrix@R=1mm{
N(V) & & N^S(V) \ar[rr]^{\iota_r} \ar[ll]_{\iota_l} & & N_{\LL}(S_0) \\
P & & P \ar@{|->}[ll] \ar@{|->}[rr] & & PT
}
$$

\end{itemize}

The functor $\iota_l$ is the inclusion of categories. The functor $\iota_r$ needs some explanation: we have to show that every morphism in $N^S(V)$ extends to a unique morphism in $N_{\LL}(S_0)$. Notice that $N_{S_0}(V)$ is $N_{\FF}^S(V)$-normal, and thus Lemma \ref{SO3-11} applies to all morphisms in $N_{\FF}^S(V)$. Since $N^S(V)$ is a subcategory of $\LL$ by definition, this means that every morphism in $N^S(V)$ extends to some morphism in $N_{\LL}(S_0)$. Furthermore, such extension must be unique by \cite[Lemma 4.3]{BLO3}.

\begin{itemize}

\item[\textbf{Step 2.}] There is a commutative diagram
$$
\xymatrix@R=15mm@C=2cm{
N(V) \ar[d]_{\tau_V} & N^S(V) \ar[l]_{\iota_l} \ar[r]^{\iota_r} \ar[d]^{\tau_I} & N_{\LL}(S_0) \ar[d]^{\tau_0} \\
\overline{\LL} & \overline{\LL} \ar@{=}[r] \ar@{=}[l] & \overline{\LL} \\
}
$$

\end{itemize}

The functor $\tau_0 \colon N_{\LL}(S_0) \to \overline{\LL} = N_{\LL}(S_0)/S_0$ is the obvious projection functor, so it needs no explanation. The functor $\tau_I \colon N^S(V) \to \overline{\LL}$ is defined simply by $\tau_I = \tau_0 \circ \iota_r$, so the rightmost square in the diagram above is clearly commutative. In particular, since $V \leq P$ for all $P \in \Ob(N^S(V))$, we have
\begin{itemize}

\item $\tau_I(P) = P/(P \cap S_0) = P \cdot T/S_0$; and

\item $\tau_I(\Aut_{N^S(V)}(P)) = \Aut_{N^S(V)}(P)/(N_{S_0}(V)) = \Aut_{N_{\LL}(S_0)}(P \cdot T)/S_0$.

\end{itemize}

Define now the functor $\tau_V \colon N(V) \to \overline{\LL}$ by the following rule. On objects, $\tau_V(P) = \tau_I(P)$ for all $P \in \Ob(N(V))$. To define $\tau_V$ on morphisms, recall that for each $P \in \Ob(N(V))$ there is an isomorphism $\Aut_{N(V)}(P) \cong \Aut_{\LL_0}(V) \times \Aut_{\overline{\LL}}(P/V)$ by Lemma \ref{SO3-23} (i). Thus, for each $P, Q \in \Ob(N(V))$ and each $\gamma \in \Mor_{N(V)}(P,Q)$ there is some $\psi \in \Aut_{\LL_0}(V) \leq \Aut_{N(V)}(Q)$ such that $\psi \circ \gamma \in \iota_l(\Mor_{N^S(V)}(P,Q))$, and we define $\tau_V(\gamma) = \tau_I(\psi \circ \gamma)$. Clearly, if $\psi' \in \Aut_{\LL_0}(V)$ is another element such that $\psi' \circ \gamma \in \iota_l(\Mor_{N^S(V)}(P,Q))$, then $\tau_I(\psi \circ \gamma) = \tau_I(\psi' \circ \gamma)$, so $\tau_V$ is well defined. Commutativity of the leftmost square follows by construction.


\begin{itemize}

\item[\textbf{Step 3.}] The diagram (diagram1).

\end{itemize}

We have to prove that the functors $\tau_0$, $\tau_I$ and $\tau_V$ induce the vertical fibrations in the commutative diagram
$$
\xymatrix{
B\Aut_{\LL_0}(V) \ar[d] & BN_{S_0}(V) \ar[r] \ar[l] \ar[d] & B\Aut_{\LL_0}(S_0) \ar[d] \\
|N_{\LL}(V)| \ar[d]_{|\tau_V|} & |N_{\LL}^S(V)| \ar[r] \ar[l] \ar[d]^{|\tau_I|} & |N_{\LL}(S_0)| \ar[d]^{|\tau_0|} \\
|\overline{\LL}| & |\overline{\LL}| \ar@{=}[r] \ar@{=}[l] & |\overline{\LL}|
}
$$
First of all, notice that for all $\overline{P} \in \Ob(\overline{\LL})$ there exists some $P \in N(V)$ such that $P/V = \overline{P}$, by Lemma \ref{SO3-12}. Now, the fibrations $|\tau_0|$ and $|\tau_I|$ exist by Corollary \ref{extension3} (i), and the map $|\tau_V|$ is a fibration with fibre $B\Aut_{\LL_0}(V)$ by \cite[Theorem B]{Quillen} and its Corollary.

Now the statement follows by Puppe's Theorem \cite{Puppe}: the (horizontal) homotopy colimit of diagram (\ref{diagram1}) produces a fibration $F_{\ast} \to Y \to |\overline{\LL}|$ such that $(F_{\ast})^{\wedge}_2 \simeq (BSO(3))^{\wedge}_2$ by \cite[Corollary 4.2]{DMW}. The fibrewise $2$-completion of the above fibration is the fibration claimed to exist in the statement.
\end{proof}

\begin{rmk}

The above result implies in particular that, if $\g = \ploc$ is a $2$-local compact group of rank $1$ whose irreducible component is induced by $SO(3)$, then $B\g \simeq (BSO(3))^{\wedge}_2 \times |\overline{\LL}|$.

\end{rmk}

\begin{prop}\label{ThmC3}

Suppose that $g_0$ is induced by $SU(2)$. Then there is a fibration $F \to X \to |\overline{\LL}|$, where $X^{\wedge}_2 \simeq B\g$ and $F \simeq B\g_0$.

\end{prop}

We could use the homology decomposition of $B\g_0$ described in \cite[Theorem 4.1]{DMW} and similar arguments as above to create a diagram similar to (\ref{diagram1}) in this case. Instead we propose an easier, shorter proof in this case.

\begin{proof}

Since $S_0 \leq S$ is strongly $\FF$-closed and $\gen{t_1} \leq S_0$ is $\FF_0$-central, it follows that $\gen{t_1}$ is $\FF$-central, and there is (homotopy) fibration
$$
B\Z/2 \Right4{} B\g \Right4{} B\g',
$$
where $\g' = (S', \FF', \LL')$ is a $2$-local compact group with irreducible component $\g'_0 = (S'_0, \FF'_0, \LL'_0)$ induced by $SO(3)$. Consider the fibration
$$
F' \Right4{} X' \Right4{} |\overline{\LL'}|
$$
in Proposition \ref{ThmC2}. In particular, $(X')^{\wedge}_2 \simeq B\g'$, $F' \simeq B\g'_0$, and an easy computation shows that $\overline{\LL'} \cong \overline{\LL}$ (recall that $\overline{\LL} = N_{\LL}(S_0)/S_0$ while $\overline{\LL'} = N_{\LL'}(S'_0)/S'_0$).

By pull-backing the fibration $B\g \to B\g'$ through the $2$-completion map $X' \to B\g'$ we obtain a commutative diagram of fibrations
$$
\xymatrix{
B\Z/2 \ar[r] & B\g \ar[r] & B\g' \\
B\Z/2 \ar[r] \ar@{=}[u] & X_{\ast} \ar[u] \ar[r] & X' \ar[u].
}
$$
Furthermore, the diagram induces an isomorphism of between the corresponding Serre spectral sequences with mod $2$ (both fibrations are principal), and thus $(X_{\ast})^{\wedge}_2 \simeq B\g$.

The pull-back of $F' \Right2{} X' \Left2{} X_{\ast}$ induces a homotopy commutative diagram
$$
\xymatrix{
B\Z/2 \ar@{=}[r] \ar[d] & B\Z/2 \ar[d] & \\
F_{\ast} \ar[r] \ar[d] & X_{\ast} \ar[r] \ar[d] & |\overline{\LL'}| \ar@{=}[d] \\
F' \ar[r] & X' \ar[r] & |\overline{\LL'}|
}
$$
Since $|\overline{\LL'}| \cong |\overline{\LL}|$, the sequence $F_{\ast} \Right2{} X_{\ast} \Right2{} |\overline{\LL}|$ is a homotopy fibration, and its fibrewise $2$-completion gives the fibration in the statement.
\end{proof}

In view of our results in this section, it makes sense to establish the following definition.

\begin{defi}\label{defitranspirrcomp}

Let $\g = \ploc$ be a $p$-local compact group of rank $1$, and let $\g_0 = (S_0, \FF_0, \LL_0)$ be its irreducible component. Then, the quotient $\overline{\g} = \g/S_0$ is the \textit{transporter system of components} of $\g$.

\end{defi}


\section{An simple $3$-local compact group of rank $2$}\label{EXO}

In this section we present an example of an \textit{exotic} simple $3$-local compact group of rank $2$. Here by exotic we mean a $3$-local compact group that is not induced by any compact Lie group or $3$-compact group.

This example has been discussed, although never formally published, by many people before. Among them, A. D\'iaz, A. Ruiz and A. Viruel who first pointed out this example after their work in \cite{DRV}, and C. Broto, R. Levi and B. Oliver, who already started working on this example in \cite{BLOprivate}. The author is specially grateful to A. Ruiz for letting him include this example in this paper, since it is part of a joint work which is still under development.

Let $S$ be the discrete $3$-toral group with presentation
$$
\begin{aligned}
S = \gen{x, \{u_k, v_k\}_{k \in \N} \,\, | \,\, \forall \, k \colon & x^3 = u_k^{3^k} = v_k^{3^k} = 1, \, u_{k+1}^3 = u_k, \, v_{k+1}^3 = v_k, \\
 & [u_k,x] = v_k, \, [v_k, x] = (u_k \cdot v_k)^{-3}, \, [u_k,v_k] = 1},
\end{aligned}
$$
where $[a,b] = a^{-1} \cdot b^{-1} \cdot a \cdot b$. An easy calculation shows that its center is $Z(S) = \gen{v_1} \cong \Z/3$. Set also $T = \gen{\{u_k, v_k\}_{k \in \N}} \cong (\Z/3^{\infty})^{\times 2}$ and $V = \gen{v_1, x} \cong \Z/3 \times \Z/3$. In particular, notice that there is an split extension $T \Right2{} S \Right2{} \Z/3$.

\begin{lmm}

The automorphism group of $T$ contains a subgroup $\Gamma$ which is isomorphic to $GL_2(3)$ and which is unique up to conjugation.

\end{lmm}

\begin{proof}

Let $\Psi_a$, $\Psi_b$ and $\Psi_c \in \Aut(T)$ be the automorphisms defined for all $k$ by
$$
\Psi_a \colon \left\{
\begin{array}{l}
u_k \Right2{} u_k \\
v_k \Right2{} v_k^{-1} \\
\end{array}
\right.
\qquad
\Psi_b \colon \left\{
\begin{array}{l}
u_k \Right2{} v_k \\
v_k \Right2{} u_k \\
\end{array}
\right.
\qquad\Psi_c \colon \left\{
\begin{array}{l}
u_k \Right2{} u_k \cdot v_k^{-1} \\
v_k \Right2{} u_k \cdot v_k \\
\end{array}
\right.
$$
Together with $\Psi_x = c_x$ they generate a subgroup $\Gamma \leq \Aut(T)$ which is isomorphic to $GL_2(3)$.

Suppose now that $\Gamma' \leq \Aut(T)$ is another subgroup which is isomorphic to $GL_2(3)$, and let $T_k = \gen{u_k, v_k} \leq T$ for all $k$. Since $T_k \leq T$ is a characteristic subgroup for all $k$, we have
$$
\Aut(T) \cong \varprojlim \Aut(T_k).
$$
Also, $\Gamma$ and $\Gamma'$ restrict to subgroups $\Gamma_k$ and $\Gamma'_k \leq \Aut(T_k)$ for all $k$. By \cite[Lemma A.17]{DRV} $\Gamma'_k$ is conjugate in $\Aut(T_k)$ to $\Gamma_k$, for all $k$, and thus $\Gamma'$ is conjugate to $\Gamma$ in $\Aut(T)$.
\end{proof}

Let $\Gamma \leq \Aut(T)$ be the subgroup described above, and let $H = T \rtimes \Gamma$, with $S \in \Syl_3(H)$. Define also $\FF$ as the fusion system over $S$ generated by $\FF_S(G)$ and $\Aut(V) = GL_2(3)$. Clearly, by definition we have
$$
\Aut_{\FF}(V) = \Aut(V) \cong GL_2(3) \qquad \mbox{and} \qquad \Aut_{\FF}(T) = \Gamma.
$$

\begin{lmm}\label{Exo1}

The fusion system $\FF$ has a single conjugacy class of elements of order $3$.

\end{lmm}

\begin{proof}

Since $\Aut_{\FF}(T) \cong GL_2(3)$, all the elements of order $3$ in $T$ (including $v_1^{-1}$) are $\FF$-conjugate to the element $v_1 \in Z(S)$. Also, the elements $x$ and $x^{-1}$ are $\FF$-conjugate to $v_1$ since $\Aut_{\FF}(V) \cong GL_2(3)$.

To finish the proof, we show that every element in $S$ of the form $\tau \cdot x$, with $\tau \in T$, is $S$-conjugate to $x$ (the same computations apply to show that elements of the form $\tau \cdot x^{-1}$ are conjugate to $x^{-1}$). Let then $\tau \in T$ be any element. Since $T$ is infinitely $3$-divisible, there exist natural numbers $i$, $j$ and $k$ such that $\tau = u_k^{3i} \cdot v_k^{3j}$, and then it is easy to check that the element $t = u_k^{-3j} \cdot v_k^{i-3j} \in T$ conjugates $\tau \cdot x$ to $x$.
\end{proof}

\begin{lmm}\label{Exo2}

The centralizer of $Z(S)$ in $\FF$ is $C_{\FF}(Z(S)) \cong \FF_S(T \rtimes \Sigma_3)$.

\end{lmm}

\begin{proof}

By definition, the morphisms in  $C_{\FF}(Z(S))$ are those that leave $Z(S)$ invariant element-wise. In particular, it is easy to see that
$$
\Aut_{C_{\FF}(Z(S))}(V) \cong \Sigma_3 \qquad \Aut_{C_{\FF}(Z(S))}(T) \cong \Sigma_3 \qquad \Out_{C_{\FF}(Z(S))}(S) \cong \Z/2,
$$
and that all morphisms in $C_{\FF}(Z(S))$ are restrictions of automorphisms of $S$. The statement follows easily.
\end{proof}

\begin{prop}\label{Exo3}

The fusion system $\FF$ is saturated over $S$ and has a unique associated linking system.

\end{prop}

\begin{proof}

We prove saturation of $\FF$ by means of Proposition \ref{Sat1}. Let then $\frakx = \{v_1\}$. By Lemma \ref{Exo1}, every element of order $3$ in $S$ is $\FF$-conjugate to $v_1$, and $\frakx$ satisfies condition (i) in \ref{Sat1}.

Next we have to show that if $g \in S$ is of order $3$ then there is some morphism $\rho \in \Hom_{\FF}(C_S(g), C_S(v_1))$ such that $\rho(g) = v_1$. We distinguish the following situations.
\begin{enumerate}[(a)]

\item If $g = v_1^{-1}$, then there is an automorphism of $S$ that sends $v_1^{-1}$ to $v_1$.

\item If $g \in T$, $g \neq v_1^{-1}$, then $C_S(g) = T$, and there is an automorphism of $T$ than sends $g$ to $v_1$.

\item If $g = \tau \cdot x$ for some $\tau \in T$, then $g$ is $S$-conjugate to $x$ by Lemma \ref{Exo1}. Thus it is enough to check property (ii) for $g = x$. In this case, $C_S(x) = V$, and there is an automorphism of $V$ that sends $x$ to $v_1$.

\end{enumerate}

Finally, we have to check that $C_{\FF}(v_1)$ is a saturated fusion system. This follows immediately from Lemma \ref{Exo2}. Clearly, $\FF$ satisfies axiom (III) of saturated fusion systems, and hence it is saturated. The existence and uniqueness of a centric linking system follows from \cite{Levi-Libman}. Alternatively, it follows from  \cite[Corollary 3.5]{BLO2} (which applies verbatim to $p$-local compact groups), since the $\F_3$-rank of $S$ (the rank of a maximal elementary abelian subgroup of $S$) is $2$.
\end{proof}

Let then $\g = \ploc$ be the resulting $3$-local compact group. Below we show that $\g$ is both simple and \textit{exotic}, in the sense that it is not induced by any compact Lie group or $3$-compact group.

\begin{thm}\label{Exo4}

The $3$-local compact group $\g$ is exotic and simple.

\end{thm}

\begin{proof}

The proof is divided into two main parts.

\begin{itemize}

\item[\textbf{Part 1.}] $\g$ is exotic.

\end{itemize}

Suppose first that $\g$ is the $3$-local compact group induced by some compact Lie group $G$, $\FF = \FF_S(G)$. Since $\Aut_{\FF}(T)$ is not an integral reflection group, it follows that $G$ cannot be irreducible. Let then $G_0 \lhd G$ be the irreducible component of $G$ that contains the identity element, with Sylow $S_0 \lhd S$ and Weyl group $W_0 = \Aut_{G_0}(T) \lhd W$.

If $W_0$ has order prime to $3$ then $S_0 = T$, which means that $G_0$ is the maximal torus of $G$, and this is not possible. Thus $W_0$ has nontrivial Sylow $3$-subgroups, and since
$$
|GL_2(3)| = 2^4 \cdot 3,
$$
this means that $W_0$ contains a Sylow $3$-subgroup of $W$, and hence contains them all by normality. It follows that $SL_2(3) \leq W_0 \leq GL_2(3)$, and $W_0$ is not an integral reflection group, a contradiction.

Suppose now that $\g$ is the $3$-local compact group induced by a $3$-compact group. Again we proceed by contradiction, so assume that $\FF = \FF_S(X)$ for some $3$-compact group $(X, BX, e)$. By the classification of $p$-compact groups, \cite{AGMV}, it follows that $X$ cannot be irreducible. Let then $(X_0, BX_0, e_0)$ be its irreducible component.

Since $S \cong T \rtimes \Z/3$ and $X$ is not irreducible, we have $\pi_1(BX) = \Z/3$, and we claim that there is a fibration
$$
BX_0 \Right4{} BX \Right4{} B\Z/3.
$$
Indeed, the Sylow $3$-subgroup $S_0$ of $X_0$ must contain $T$, so $T \leq S_0 \leq S$. If $S_0 = S$, then $X_0 = X$, which is impossible. Thus $S_0 = T$ and the fibration above follows. However, this implies an exact sequence of Weyl groups
$$
\{1\} \to W_0 \Right4{} W \cong GL_2(3) \Right4{} \Z/3 \to \{1\},
$$
and no such extension exists.

\begin{itemize}

\item[\textbf{Part 2.}] $\g$ is simple.

\end{itemize}

We show something stronger: the fusion system $\FF$ does not contain any proper normal subsystem. Suppose otherwise, and let $\FF_0 \lhd \FF$ be a proper normal subsystem over some $S_0 \leq S$.

Clearly, $S$ does not contain any proper strongly $\FF$-closed subgroup, and thus $S_0 = S$. Then, we must have
$$
\Aut_{\FF_0}(V) = \Aut_{\FF}(V) \cong GL_2(3) \qquad \mbox{and} \qquad \Aut_{\FF_0}(T) = \Aut_{\FF}(T),
$$
since otherwise condition (N2) of normal subsystems does not hold. Finally, $\Aut_{\FF}(S)$ is generated by inner automorphisms an extensions of automorphisms of $T$. Thus, $\Aut_{\FF_0}(S) = \Aut_{\FF}(S)$, and $\FF_0 = \FF$ (see \cite[Lemmas 5.6 and 5.7]{DRV}, which apply verbatim in this case).
\end{proof}


\section{Irreducibility for $p$-local compact groups}\label{Irred}

In this section we want to address a short discussion on several features of the notion of irreducibility introduced in Definition \ref{defiirred}. In this sense, the first aspect that we want to consider is the relationship between connectivity for compact Lie groups and $p$-local compact groups, and irreducibility for $p$-local compact groups.

As Theorem \ref{thmA} has already evidenced, not every connected compact Lie group (respectively connected $p$-compact group) induces an irreducible $p$-local compact group. Needless to say that it is very important to check which connected compact Lie groups and connected $p$-compact groups give rise to connected $p$-local compact groups. This study has not been included in this paper due to its length, but we can sketch here our approach.

Let then $G$ be a connected compact Lie group, and let $\g = \ploc$ be the $p$-local compact group induced by $G$. If $\g$ is not irreducible it means that there exists some proper normal subsystem $\FF_0 \lhd \FF$, over some $S_0 \lhd S$, and in particular $S_0$ is strongly $\FF$-closed. We can apply here a result of Notbohm, \cite[Proposition 4.3]{Not1}, to get a first restriction on the possible subgroups $S_0$ (and there will be not that many of them). The analysis of the classical Lie groups, namely the four families $U(n)$, $SU(n)$, $SO(n)$ and $Sp(n)$, will follow then from the results in \cite{Oliver1}. This will leave us with a series of exceptional cases to deal with, in a case-by-case argument.

Another point of interest is the existence and uniqueness of \textit{irreducible components} and \textit{transporter systems of components} for $p$-local compact groups (see Definitions \ref{defiirrcomp} and \ref{defitranspirrcomp}). Before we state our conjecture in this direction, let us formalize our ideas.

Let $\g = \ploc$ be a $p$-local compact group. In view of Definition \ref{defiirred}, the \textit{irreducible component} of $\g$ should be defined by something in the lines of \textit{the minimal irreducible $p$-local compact group $\g_0 = (S_0, \FF_0, \LL_0)$ such that $\FF_0$ is subnormal in $\FF$}. Now, why should such a $p$-local compact group exist since Aschbacher already showed in \cite[Example 6.4]{Aschbacher} that the intersection of normal subsystems need not be a normal subsystem? That is, how could we even expect to find a minimal normal subsystem of $\FF$ whose Sylow $p$-subgroup contains the maximal torus?

In recent, yet unpublished, work \cite[Theorem F]{Chermak2}, Chermak shows that the collection of normal subsystems of a given, finite, saturated fusion system is in one-to-one correspondence with the collection of normal partial subgroups of the associated locality. The advantage with localities and partial groups is that the intersection of normal partial subgroups is again a normal partial subgroup. In particular, if $\FF$ is a saturated fusion system over a finite $p$-group $S$, and $A \leq S$ is an (strongly $\FF$-closed) subgroup, then there is a \textit{minimal normal fusion subsystem} $\FF_0 \lhd \FF$ whose Sylow $S_0$ contains $A$ (although $A$ will probably be a proper subgroup of $S_0$).

In view of this it makes sense to work on a generalization of the results of Chermak to localities associated to $p$-local compact groups. But there is yet another point of interest in generalizing these results. Indeed, the results in \cite{Chermak2} also include the construction of a quotient of a locality by a normal partial subgroup, and this construction would then provide a solid notion of \textit{transporter system of components}. We plan to develop this idea in a subsequent paper. Notice that this means quite a lot of work, since many of the results in \cite{Chermak} have to be generalized as well.

\begin{conjec}

Every $p$-local compact group determines a unique irreducible component and a unique transporter system of components.

\end{conjec}

There is still another feature of high interest to discuss regarding irreducibility. Indeed, it is not clear from our definition that, for a fixed $r \geq 0$, the number of irreducible $p$-local compact groups of rank $r$ is finite, although it is our belief that this is the case.

\begin{conjec}

Let $r \geq 0$ be a fixed natural number. Then there are finitely many irreducible $p$-local compact groups of rank $r$.

\end{conjec}

Let $\g = \ploc$ be an irreducible $p$-local compact group with maxima torus $T$. A key point to prove the above conjecture is to study the action of $S/T$ on $T$. More specifically, is the action of $S/T$ on $T$ faithful under these hypothesis? A positive answer would most likely imply the above conjecture, as well as other interesting results.


\appendix

\section{Extensions of transporter systems with discrete $p$-toral group kernel}\label{Extensions}

In \cite{OV}, the authors study a particular construction to create new transporter systems (and linking systems). From the moment this classification program started it was clear that their construction had to be extended to a more general setting, including (at least) extensions of (finite) transporter systems by discrete $p$-toral groups.

Thus we proceed to provide the consistent theory of extensions that we need in previous sections of this paper. Some of the results in this section have been already proved in \cite[Appendix \S A]{BLO6}, or the proof of the corresponding statement in \cite{OV} applies without restriction here. For the sake of simplicity we will just state the result and refer the reader to appropriate source when this is the case.


\subsection{Quotients of transporter systems by discrete $p$-toral groups}\label{Quotients}

In this section we describe the quotient of a transporter system by a normal discrete $p$-toral subgroup. For a subgroup $A \leq S$ we denote by $\hh_A$ the collection of subgroups of $S$ that contain $A$, $\hh_A = \{P \leq S \,| \, A \leq P\}$.

\begin{defi}\label{defiquotient}

Let $(S, \FF, \TT)$ be a transporter system over a discrete $p$-toral group. The \textit{quotient} of $(S, \FF, \TT)$ by an $\FF$-normal subgroup $A \leq S$ is the triple $(S/A, \FF/A, \TT/A)$, where
\begin{itemize}

\item $\FF/A$ is the fusion system over $S/A$ with morphism sets
$$
\begin{aligned}
\Hom_{\FF/A}(P/A, Q/A) = \{\overline{f} \in \Hom(P/A,Q/A) \,\, | & \,\, \exists P,Q \in \hh_A \mbox{ and } f \in \Hom_{\FF}(P,Q) \\
 & \mbox{ such that } \overline{f} = \ind(f)\}.
\end{aligned}
$$

\item $\LL/A$ is the category with object set
$\{P/A \, | \, P \in \hh_A \cap \Ob(\TT)\}$ and morphism sets
$$
\Mor_{\TT/A}(P/A,Q/A) = \Mor_{\TT}(P,Q)/\varepsilon_P(A).
$$

\end{itemize}
Finally, set $\TT_{\geq A, \, c} \subseteq \TT$ and $(\TT/A)^c \subseteq \TT/A$ be the full subcategories whose objects are the subgroups $P \leq S$ and $P/A \leq S/A$, respectively, such that $P/A$ if $(\FF/A)$-centric.

\end{defi}
Notice that we do not assume saturation of $\FF$ in the above definition.

\begin{prop}\label{quotientT}

Let $(S, \FF, \TT)$ be a transporter system over the discrete $p$-toral group $S$, and let $A \leq S$ be an $\FF$-normal subgroup. Then, the quotient $(S/A, \FF/A, \TT/A)$ is a transporter system.

\end{prop}

\begin{proof}

The functor $\overline{\varepsilon}: \TT_{Ob(\TT/A)}(S/A) \to \TT/A$ is defined by
$$
(\overline{\varepsilon})_{P/A, Q/A}(gA) = [\varepsilon_{P,Q}(g)] \in Mor_{\TT}(P,Q)/A,
$$
while the functor $\overline{\rho}: \TT/A \to \FF/A$ is induced by $\rho$.

Since $\Syl_p(A) = \{A\}$, axioms (A1) and (B) hold for $\TT/A$ because they already hold for $\TT$, and axioms (A2), (C), (I) and (II) hold using the same arguments in the proof for Proposition 3.10 in \cite{OV}.

Thus, we just have to show that axiom (III) also holds for $\TT/A$. Suppose that we are given an ascending chain of subgroups in $S/A$, $P_1/A \leq P_2/A \leq \ldots$. Set then $P/A = \cup P_n/A$, and for each $n$ let $\overline{\varphi}_n \in \Mor_{\TT/A}(P_n/A, S/A)$ be such that
$$
\overline{\varphi}_n = \overline{\varphi}_{n+1} \circ \overline{\varepsilon}_{P_n/A, P_{n+1}/A}(1).
$$
We want to see that there exists $\overline{\varphi} \in \Mor_{\TT/A}(P/A,S/A)$ such that for each $n$ $\overline{\varphi}_n$ is the corresponding restriction of $\overline{\varphi}$.

The idea is to choose liftings in $\TT$ of the morphisms $\overline{\varphi}_n$, so that we can apply axiom (III) in $\TT$. Start by lifting $\overline{\varphi}_1$ to some $\varphi_1 \in \Mor_{\TT}(P_1,S)$, and now suppose we have already chosen liftings $\varphi_1, \ldots, \varphi_n$ such that, for each $i = 1, \ldots, n-1$,
$$
\varphi_i = \varphi_{i+1} \circ \varepsilon_{P_i,P_{i+1}}(1),
$$
and choose a lifting $\varphi_{n+1}' \in \Mor_{\TT}(P_{n+1}, S)$ of $\overline{\varphi}_{n+1}$. This lifting may not satisfy that $\varphi_n = \varphi_{n+1}' \circ \varepsilon_{P_n,P_{n+1}}(1)$, but by definition of $\TT/A$ there exists some $a \in A$ such that
$$
\varphi_n = \varphi_{n+1}' \circ \varepsilon_{P_n,P_{n+1}}(1) \circ \varepsilon_{P_n}(a) = (\varphi_{n+1}' \circ \varepsilon_{P_{n+1}}(a)) \circ \varphi_{P_n, P_{n+1}}(1),
$$
where the second equality holds by axiom (C) for transporter systems applied on $\TT$. Thus, $\varphi_{n+1} = \varphi_{n+1}' \circ \varepsilon_{P_{n+1}}(a)$ satisfies de desired condition. Inductively, we obtain liftings for all $\overline{\varphi}_n$ such that each lifting is the restriction of the next one.

Now, we can apply axiom (III) for transporter systems on $\TT$ for the family $\{\varphi_n\}$: there exists some $\varphi \in \Mor_{\TT}(P,S)$ such that $\varphi_n = \varphi \circ \varepsilon_{P_n,P}(1)$ for all $n$, and the induced morphism $\overline{\varphi} \in \Mor_{\TT/A}(P/A,S/A)$ is the morphism we were looking for.
\end{proof}

Consider now the following, particular situation. Let $\g = \ploc$ be a $p$-local compact group, and let $A \leq S$ be an $\FF$-normal subgroup. The above result states that the quotient $\g/A = (S/A, \FF/A, \LL/A)$ is a transporter system, since $\g$ is a transporter system in particular. However, in this case we can say more.

\begin{prop}\label{quotientF}

 Let $\g = \ploc$ be a $p$-local compact group, and let $A \leq S$ be an $\FF$-normal subgroup. Then, $\FF/A$ is a saturated fusion system over $S/A$.

\end{prop}

\begin{proof}

We use the alternative set of axioms introduced in Lemma \ref{axiomsKS}. Note that axiom (III) for $\FF/A$ follows immediately from axiom (III) for $\FF$. Hence, we only have to show that $\FF/A$ satisfies (I') and (II').

To show that (I') holds it is enough to check that $\{1\} \in \Syl_p(\Out_{\FF/A}(S/A))$, and this follows immediately because there by definition an epimorphism $\xymatrix{\Out_{\FF}(S) \ar@{->>}[r] & \Out_{\FF/A}(S/A)}$, and $\{1\} \in \Syl_p(\Out_{\FF}(S))$.

Suppose now that $\overline{f} \in \Hom_{\FF/A}(P/A, S/A)$ is such that $\overline{f}(P/A)$ is fully $\FF/A$-normalized. We have to show that $\overline{f}$ extends to some $\overline{\gamma} \in \Hom_{\FF/A}(N_{\overline{f}}, S/A)$, where
$$
N_{\overline{f}} = \{gA \in N_{S/A}(P/A) \,\, | \,\, \overline{f} \circ c_{gA} \circ (\overline{f})^{-1} \in \Aut_{S/A}(\overline{f}(P/A))\}.
$$

Notice that $Q/A \leq S/A$ is fully $\FF/A$-normalized if and only if $Q \leq S$ is fully $\FF$-normalized. Choose then some representative $f \in \Hom_{\FF}(P,S)$ of $\overline{f}$. Then, $f(P)$ is fully $\FF$-normalized, and axiom (II') applies: $f$ extends to some $\gamma \in \Hom_{\FF}(N_f, S)$. And easy calculation shows that $N_{\overline{f}} \leq N_f/A$, and the claim follows.
\end{proof}


\subsection{Extensions of transporter systems by discrete $p$-toral groups}

The quotients described in the previous section have their counterpart in the following extension theory for transporter system.

\begin{defi}\label{extensionT}

Let $(S, \FF, \TT)$ be a transporter system over the discrete $p$-toral group $S$. An \textit{extension} of $\TT$ by a discrete $p$-toral group is a category $\wtt$, together with a functor $\tau \colon \wtt \to \TT$ which is the identity on objects, and such that the following holds for all $\widehat{P}, \widehat{Q} \in \Ob(\wtt)$:
\begin{enumerate}[(i)]

\item $K_{\widehat{P}} \defin \Ker[\Aut_{\wtt}(\widehat{P}) \Right2{} \Aut_{\TT}(P)]$ is a discrete $p$-toral group;

\item $K_{\widehat{P}}$ acts freely on $\Mor_{\wtt}(\widehat{P}, \widehat{Q})$ by right composition and $\tau$ is the orbit map of this action; and

\item $K_{\widehat{Q}}$ acts freely on $\Mor_{\wtt}(\widehat{P}, \widehat{Q})$ by left composition and $\tau$ is the orbit map of this action.

\end{enumerate}

\end{defi}
We adopt the notation $\widehat{P}, \widehat{Q}, \ldots$ for the objects in $\wtt$ to distinguish them from the objects in $\TT$, even if the functor $\tau$ is the identity on objects. This way, it is clear that $\tau(\widehat{P}) = P$, $\tau(\widehat{Q}) = Q, \ldots$

By definition, the functor $\tau$ is source and target regular in the sense of \cite[Definition A.5]{OV}, and in particular there is some discrete $p$-toral group $A$ such that
$$
K_{\widehat{P}} = \cong A
$$
for all $\widehat{P} \in \Ob(\wtt)$ by \cite[Lemma A.7]{OV}. Thus, we can describe $\wtt$ as an extension of $\TT$ by the discrete $p$-toral group $A \cong K_{\widehat{S}}$. We will use the following notation
$$
A \Right4{} \wtt \Right4{\tau} \TT
$$
for such an extension.

Fix then some extension $A \Right2{} \wtt \Right2{\tau} \TT$, and let $\widetilde{S}$ be the pull-back
$$
\xymatrix@C=2cm{
\widetilde{S} \ar[r]^{\widetilde{\varepsilon}_{\widetilde{S}}} \ar[d]_q & \Aut_{\wtt}(\widehat{S}) \ar[d]^{\tau_{\widehat{S}}} \\
S \ar[r]_{\varepsilon_S} & \Aut_{\TT}(S)
}
$$
Set also $A = \Ker(q)$, and $\widetilde{P} = q^{-1}(P)$ for each $P \in \TT$. Hence, for each $P \in \Ob(\TT)$ there is a group extension
$$
A \Right4{} \widetilde{P} \Right4{q|_P} P,
$$
and we can identify the set $\Ob(\wtt)$ with the set $\{\widetilde{P} \, | \, P \in \Ob(\TT)\}$. Note that $\widetilde{S}$ is a discrete $p$-toral group since both $S$ and $A$ are so.

Clearly, the goal is to show that $\wtt$ has the structure of a transporter system. We start by defining functors
$$
\TT_{\Ob(\wtt)}(\widetilde{S}) \Right4{\widetilde{\varepsilon}} \wtt \Right4{\widetilde{\rho}} \Gps.
$$
First of all, choose for each $\widetilde{P}$ a lifting $\widetilde{\iota}_{\widetilde{P}} \in \Mor_{\wtt}(\widetilde{P}, \widetilde{S})$ of the morphism $\varepsilon_{P,S}(1) = \iota_P \in \Mor_{\TT}(P,S)$. The choices are made randomly, except that we require $\widetilde{\iota}_{\widetilde{S}} = \Id_{\widetilde{S}}$. The functors $\widetilde{\varepsilon}$ and $\widetilde{\rho}$ are defined by the following properties.
\begin{enumerate}[(a)]

\item For each $\widetilde{P}, \widetilde{Q} \in \Ob(\wtt)$ and each $\widetilde{x} \in N_{\widetilde{S}}(\widetilde{P}, \widetilde{Q})$ there is a unique morphism $\widetilde{\varepsilon}_{\widetilde{P},\widetilde{Q}}(\widetilde{x})$ in $\wtt$ that makes the following diagram commute
$$
\xymatrix@C=1.5cm{
\widetilde{P} \ar[r]^{\widetilde{\iota}_{\widetilde{P}}} \ar[d]_{\widetilde{\varepsilon}_{\widetilde{P},\widetilde{Q}}(\widetilde{x})} & \widetilde{S} \ar[d]^{\widetilde{\varepsilon}_{\widetilde{S}}(\widetilde{x})} \\
\widetilde{Q} \ar[r]_{\widetilde{\iota}_{\widetilde{Q}}} & \widetilde{S}
}
$$

\item For each $\widetilde{P}, \widetilde{Q} \in \Ob(\wtt)$ and each $\widetilde{\varphi} \in \Mor_{\wtt}(\widetilde{P}, \widetilde{Q})$ there is a unique group homomorphism $\widetilde{f} \in \Hom(\widetilde{P}, \widetilde{Q})$ such that the following diagram commutes for all $x \in \widetilde{P}$
$$
\xymatrix@C=1.5cm{
\widetilde{P} \ar[r]^{\widetilde{\varphi}} \ar[d]_{\widetilde{\varepsilon}_{\widetilde{P},\widetilde{P}}(\widetilde{x})} & \widetilde{Q} \ar[d]^{\widetilde{\varepsilon}_{\widetilde{Q},\widetilde{Q}}(\widetilde{f}(x))} \\
\widetilde{Q} \ar[r]_{\widetilde{\varphi}} & \widetilde{Q}
}
$$

\end{enumerate}
Property (a) corresponds to \cite[Lemma 5.3]{OV} while property (b) corresponds to \cite[Lemma 5.5]{OV}. We omit the corresponding proofs since those in \cite{OV} apply without restriction in this case. The functors $\widetilde{\varepsilon}$ and $\widetilde{\rho}$ are the defined accordingly. Define also $\widetilde{\FF}$ as the fusion system over $\widetilde{S}$ generated by the image of $\widetilde{\rho}$.

\begin{prop}\label{extension1}

Let $\wtt$ be an extension of the transporter system $(S, \FF, \TT)$ by the discrete $p$-toral group $A$. Then, $(\widetilde{S}, \widetilde{\FF}, \wtt)$ is a transporter system. Furthermore, $A$ is $\widetilde{\FF}$-normal, and
$$
(S, \FF, \TT) \cong (\widetilde{S}/A, \widetilde{\FF}/A, \wtt/A).
$$

\end{prop}

\begin{proof}

The proof of \cite[Proposition 5.6]{OV} applies verbatim in this case, and we are left to prove that $(\widetilde{S}, \widetilde{\FF}, \wtt)$ satisfies axiom (III) of transporter systems.

Let then $\widetilde{P}_1 \leq \widetilde{P}_2 \leq \ldots $ be an increasing sequence of subgroups in $Ob(\widetilde{\TT})$, with $\widetilde{P} = \bigcup_n \widetilde{P}$. Suppose also that for all $n$ there exists $\widetilde{\varphi}_n \in \Mor_{\widetilde{\TT}}(\widetilde{P}_n,\widetilde{S})$ such that
$$
\widetilde{\varphi}_n = \widetilde{\varphi}_{n+1} \circ \widetilde{\varepsilon}_{\widetilde{P}_n, \widetilde{P}_{n+1}}(1).
$$
We have to prove then that there exists $\widetilde{\varphi} \in \Mor_{\widetilde{\TT}}(\widetilde{P}, \widetilde{S})$ such that
$$
\widetilde{\varphi}_n = \widetilde{\varphi} \circ \widetilde{\varepsilon}_{\widetilde{P}_n, \widetilde{P}}(1).
$$

By projecting all the $\widetilde{P}_n$ and the $\widetilde{\varphi}_n$ to $\TT$ through the functor $\widetilde{\rho}$, we get a family of subgroups $\{P_n\}$ and morphisms $\{\varphi_n\}$ like the above, and we can apply axiom (III) on $\TT$ to see that there exists $\varphi \in \Mor_{\TT}(P,S)$ such that $\varphi_n = \varphi \circ \varepsilon_{P_n,P}(1)$ for all $n$.

Let $\widetilde{\varphi}' \in \Mor_{\widetilde{\TT}}(\widetilde{P}, \widetilde{S})$ be a lifting in $\widetilde{\TT}$ of $\varphi$. By construction, the projections of $\widetilde{\varphi}_1$ and of $\widetilde{\varphi}' \circ \widetilde{\varepsilon}_{\widetilde{P}_1, \widetilde{P}}(1)$ on $\TT$ are equal, and hence as morphisms in $\widetilde{\TT}$ they differ by a morphism in $A = Ker(q)$. This means that there exists some $a \in A$ such that
$$
\widetilde{\varphi} \defin \widetilde{\varphi}' \circ \widetilde{\varepsilon}(a)
$$
restricts to $\widetilde{\varphi}_1$ and is still a lifting of $\varphi$.

Applying the equalities $\widetilde{\varphi}_1 = \widetilde{\varphi}_2 \circ \widetilde{\varepsilon}_{\widetilde{P}_1, \widetilde{P}_2}(1)$ and $\widetilde{\varepsilon}_{\widetilde{P}_1, \widetilde{P}}(1) = \widetilde{\varepsilon}_{\widetilde{P}_2, \widetilde{P}}(1) \circ \widetilde{\varepsilon}_{\widetilde{P}_1, \widetilde{P}_2}(1)$, we obtain new equalities
$$
\widetilde{\varphi} \circ \widetilde{\varepsilon}_{\widetilde{P}_2, \widetilde{P}}(1) \circ \widetilde{\varepsilon}_{\widetilde{P}_1, \widetilde{P}_2}(1) = \widetilde{\varphi}_1 = \widetilde{\varphi}_2 \circ \widetilde{\varepsilon}_{\widetilde{P}_1, \widetilde{P}_2}(1).
$$
Since the natural projection functor $\widetilde{\rho} \colon \widetilde{\TT} \to \TT$ is, in particular, target regular by definition, it follows by \cite[Lemma A.8]{OV} that morphisms in $\widetilde{\TT}$ are epimorphisms in the categorical sense, and hence, from the above equalities we deduce that the restriction of $\widetilde{\varphi}$ to $\widetilde{P}_2$ is $\widetilde{\varphi}_2$ as desired. Repeating this process we see that axiom (III) holds in $\widetilde{\TT}$.
\end{proof}

Let then $(\widetilde{S}, \widetilde{\FF}, \wtt)$ be an extension of a transporter system $(S, \FF, \TT)$ by a discrete $p$-toral group $A$. We now address the question of whether $\wtt$ contains all the $\widetilde{\FF}$-centric $\widetilde{\FF}$-radical subgroups (which in particular would imply that $\widetilde{\FF}$ is saturated).

The following is proved in \cite[Lemma 5.7]{OV}. Given the extension $(\widetilde{S}, \widetilde{\FF}, \wtt)$, the map defined by
$$
\xymatrix@R=1mm{
\Mor(\TT) \ar[rr]^{\Phi} & & \Out(A) \\
\varphi \ar@{|->}[rr] & &  [\widetilde{\rho}(\widetilde{\varphi})|_A]
}
$$
for some $\widetilde{\varphi}$ such that $\tau(\widetilde{\varphi}) = \varphi$, is well defined, and satisfies
\begin{enumerate}[(a)]

\item $\Phi(\varphi) \cdot \Phi(\psi) = \Phi(\varphi \circ \psi)$, for each pair of composable morphisms $\varphi, \psi \in \Mor(\TT)$; and

\item for eany lifting $\widetilde{\varphi} \in \Mor(\wtt)$ of $\varphi$ and all $a \in A$, the following diagram commutes
$$
\xymatrix@C=1.5cm{
\widetilde{P} \ar[r]^{\widetilde{\varphi}} \ar[d]_{\widetilde{\varepsilon}_{\widetilde{P},\widetilde{P}}(\widetilde{a})} & \widetilde{Q} \ar[d]^{\widetilde{\varepsilon}_{\widetilde{Q},\widetilde{Q}}(\widetilde{\rho}(\widetilde{\varphi})(a))} \\
\widetilde{Q} \ar[r]_{\widetilde{\varphi}} & \widetilde{Q}
}
$$

\end{enumerate}
The map $\Phi$ induces then a morphism $\Phi \colon \pi_1(|\TT|) \to \Out(A)$.

\begin{defi}\label{defiadmis}

Let $(S, \FF, \TT)$ be a transporter system over a discrete $p$-toral group $S$.
\begin{enumerate}[(i)]

\item A morphism $\Phi \colon \pi_1(|\TT|) \to \Gamma$ is \textit{admissible} if, upon setting $S_1 = \Ker(\Phi \circ \varepsilon_{S,S})$, $P \leq S$ fully $\FF$-centralized and $C_{S_1}(P) \leq P$ imply $P \in \Ob(\TT)$.

\item An action $\Phi \colon \pi_1(|\TT|) \to \Out(A)$ of $\TT$ on a discrete $p$-toral group $A$ is \textit{admissible} if the homomorphism $\Phi$ is admissible.

\item An extension $A \Right2{} \wtt \Right2{} \TT$ is \textit{admissible} if the action of $\TT$ on $A$ defined above is admissible.

\end{enumerate}

\end{defi}

\begin{thm}\label{extension2}

Let $A \Right2{} \wtt \Right2{} \TT$ be an admissible extension of the transporter system $(S, \FF, \TT)$ by the discrete $p$-toral group $A$. Then $\widetilde{\FF}$ is a saturated fusion system.

\end{thm}

\begin{proof}

The proof of \cite[Theorem 5.11 (b)]{OV} applies here without modification.
\end{proof}

\begin{cor}\label{extension3}

Let $\TT$ be a transporter system, and  let $A$ be a discrete $p$-toral group. Then, the following holds.
\begin{enumerate}[(i)]

\item If $A \Right2{} \wtt \Right2{\tau} \TT$ is an extension, then the realization of $\tau$ is a fibration
$$
BA \Right4{} |\wtt| \Right4{|\tau|} |\TT|.
$$

\item Every fibration $BA \Right2{} X \Right2{} |\TT|$ is, up to equivalence, the realization of an extension
$$
A \Right4{} \wtt \Right4{} \TT.
$$

\end{enumerate}

\end{cor}

\begin{proof}

Property (i) follows from \cite[Proposition A.10]{OV}, and property (ii) from \cite[Proposition 5.8]{OV}.
\end{proof}


\section{The Hyperfocal Subgroup Theorem}

In this section we prove the well-known Hyperfocal Subgroup Theorem for $p$-local compact groups. We then use this result to show that a nontrivial fundamental group implies the existence of a proper normal subsystem. This section is simply a generalization of some parts of \cite{BCGLO2}, and no attempt is done to cover all their results.

Let us fix some notation for the whole section. Given a group $G$, we will denote by $O^p(G) \leq G$ the subgroup generated by all the infinitely $p$-divisible elements. Also, $\bb G$ will denote the category with a single object $\circ_G$ (or simply $\circ$ if $G$ is understood) with automorphism group $G$.

\begin{defi}\label{hyperfocal}

Let $\FF$ be a saturated fusion system over a discrete $p$-toral group $S$, with maximal torus $T$. The \textit{hyperfocal subgroup} of $\FF$ is the following subgroup of $S$
$$
O^p_{\FF}(S) \defin \gen{T, \{g^{-1} \cdot \alpha(g) \,\, | \,\, g \in P \leq S, \alpha \in O^p(\Aut_{\FF}(P))\}}.
$$

\end{defi}

Let $G$ be an artinian locally finite group such that has Sylow $p$-subgroups, fix $S \in Syl_p(G)$, and define
$$
\begin{array}{rl}
O^p_G(S) \stackrel{def} = O^p_{\FF_S(G)}(S) & =  \gen{\{g^{-1} \alpha(g) | g \in P \leq S \mbox{, } \alpha \in O^p(Aut_G(P))\}, T} \\
 & = \gen{\{[g,x] | g \in P \leq S \mbox{, } x \in N_G(P) \mbox{ of } p' \mbox{ order}\}, T}.\\
\end{array}
$$

\begin{lmm}\label{hyper0}
 
Let $G$ be an artinian locally finite group such that has Sylow $p$-subgroups, and let $S \in Syl_p(G)$. Then,
$$
O^p_G(S) = S \cap O^p(G).
$$

\end{lmm}

\begin{proof}

Let $\underline{O}^p(G)$ be the subgroup of $G$ generated by all elements of order prime to $p$ in $G$, which is a subgroup of $O^p(G)$. We prove the following result:
$$
\underline{O}^p_G(S) \stackrel{def} = \gen{\{g^{-1} \alpha(g) | g \in P \leq S \mbox{, } \alpha \in O^p(Aut_G(P))\}} = S \cap \underline{O}^p(G),
$$
which is equivalent to the statement in the lemma.

The inclusion $\underline{O}^p_G(S) \leq S \cap \underline{O}^p(G)$ holds by the same arguments as in \cite[Lemma 2.2]{BCGLO2}, and we want to see the converse inclusion. Let $\{G_i\}$ be a family of finite subgroups of $G$ such that $G = \cup G_i$, and set $S_i = S \cap G_i$. It follows then that $S = \cup S_i$. We can choose the subgroups $G_i$ such that $S_i \in Syl_p(G_i)$ for all $i$.

We first check that $\underline{O}^p_G(S) = \cup \underline{O}^p_{G_i}(S_i)$. This is in fact clear since, for each $g^{-1} \alpha(g) \in O^p_G(S)$, we can find $M_1$ such that, for all $i \geq M_1$, $g \in S_i$ and $\alpha$ is conjugation by an element in $G_i$, and hence $g^{-1} \alpha(g) \in \underline{O}^p_{G_i}(S_i)$.

Next we show that $S \cap \underline{O}^p(G) = \cup (S_i \cap \underline{O}^p(G_i))$. Note that, since $S = \cup S_i$, we have
$$
S \cap \underline{O}^p(G) = \cup (S_i \cap \underline{O}^p(G)).
$$
Since, for each $i$, $\underline{O}^p(G_i)$ is generated by all elements in $G_i$ of order prime to $p$, it follows that for all $i$  we have inclusions
$$
\begin{array}{c}
S_i \cap \underline{O}^p(G_i) \leq S_i \cap \underline{O}^p(G),\\
S_i \cap \underline{O}^p(G_i) \leq S_{i+1} \cap \underline{O}^p(G_{i+1}).\\
\end{array}
$$
Let $y \in S_i \cap \underline{O}^p(G)$, in particular, $y \in \underline{O}^p(G)$, and hence $y \in \underline{O}^p(G_j)$ for all $j \geq i$ big enough. Since $S_i \leq S_j$ for $j \geq i$, it follows that, for each $i$, there exists some $j$ such that
$$
S_i \cap \underline{O}^p(G_i) \leq S_i \cap \underline{O}^p(G) \leq S_j \cap \underline{O}^p(G_j).
$$
Hence, $S \cap \underline{O}^p(G) = \cup (S_i \cap \underline{O}^p(G_i))$, and by the hyperfocal subgroup theorem in the finite case, \cite[Lemma 2.2]{BCGLO2}, $\underline{O}^p_G(S) = \cup \underline{O}^p(G_i) = \cup (S_i \cap \underline{O}^p(G_i)) = S \cap \underline{O}^p(G)$.
\end{proof}

Let now $\g = \ploc$ be a $p$-local compact group, with maximal torus $T$, and let $\{\iota_{P,Q}\} \subseteq \Mor(\LL)$ be the set of inclusions $\varepsilon_{P,Q}(1) \in \Mor_{\LL}(P,Q)$, for all $P \leq Q$ in $\Ob(\LL)$. For simplicity, we set $\iota_P = \iota_{P,P}$ for all $P \in \Ob(\LL)$. Let also
$$
J \colon LL \Right4{} \bb \pi_1(|\LL|)
$$
be the functor that sends a morphism $\varphi \colon P \to Q$ to the homotopy class of the loop $\iota_Q \circ \varphi \circ \iota_P^{-1}$.

\begin{lmm}\label{hyper1}

Let $\Gamma$ be a discrete group and let $\lambda \colon \LL \Right2{} \bb\Gamma$ be a functor that sends the set $\{\iota_{P,Q}\}$ to the identity automorphism. Then there exists a unique homomorphism $\rho \colon \pi_1(|\LL|) \Right2{} \Gamma$ such that $\lambda = \bb \rho \circ J$.

\end{lmm}

\begin{proof}

By geometric realization, every functor $\lambda \colon \LL \to \bb \Gamma$ induces a homomorphism $\rho \colon \pi_1(|\LL|) \to \Gamma$. If in addition $\gamma$ sends inclusions to identity, then the equality $\lambda = \bb \rho \circ J$ follows.
\end{proof}

\begin{lmm}\label{hyper2}

There is a unique functor $\lambda \colon \LL \Right2{} \bb (S/O^p_{\FF}(S))$ satisfying the following properties:
\begin{enumerate}[(i)]

\item $\lambda$ sends the set $\{\iota_{P,Q}\}$ to the identity automorphism; and

\item $\lambda(\varepsilon_S(g)) = gO^p_{\FF}(S)$ for all $g \in S$.

\end{enumerate}

\end{lmm}

\begin{proof}

This is essentially \cite[Proposition 2.4]{BCGLO2}, with minor modifications. Essentially, let $\LL^{\bullet} \subseteq \LL$ be the retract of $\LL$ defined in \cite[Sections \S 3 and 4]{BLO3}, and let $r \colon \LL \to \LL^{\bullet}$ be the retraction defined there. The key point is that $\LL^{\bullet}$ contains finitely many conjugacy classes of objects. Now, \cite[Lemma 2.3]{BCGLO2} can be adapted to work on the $\LL^{\bullet}$ (it is irrelevant here that $\FF^{\bullet}$ is not closed by overgroups), and this way one can construct a functor $\lambda_0 \colon \LL^{\bullet} \to \bb (S/O^p_{\FF}(S))$, just following \cite[Proposition 2.4]{BCGLO2}. The functor $\lambda$ is just defined as the composition $\lambda_0 \circ r$.
\end{proof}

\begin{thm}\label{hyper3}

Let $\g = \ploc$ be a $p$-local compact group. Then, $\pi_1(B\g) \cong S/O^p_{\FF}(S)$. More precisely, the natural map $\tau \colon S \to \pi_1(B\g)$ is surjective with kernel $O^p_{\FF}(S)$.

\end{thm}

\begin{proof}

Let $\lambda \colon \LL \to \bb(S/O^p_{\FF}(S))$ be the functor from Lemma \ref{hyper2}, and let $|\lambda|$ be the induced map between geometric realizations. Since $O^p_{\FF}(S)$ is a finite $p$-group, the space $|\bb(S/O^p_{\FF}(S))|$ is $p$-complete, and $|\lambda|$ factors through $B\g$. There is then a commutative diagram
$$
\xymatrix@C=2cm{
S \ar[r]^{j} \ar[rd]_{\tau} & \pi_1(|\LL|) \ar[rd]^{\pi_1(|\lambda|)} \ar[d] & \\
 & \pi_1(B\g) \ar[r]_{\pi_1(|\lambda|^{\wedge}_p)} & S/O^p_{\FF}(S)
}
$$
The morphism $\tau$ is surjective by \cite[Proposition 4.4]{BLO3}, and the composition $\pi_1(|\lambda|^{\wedge}_p) \circ \tau = \pi_1(|\lambda|) \circ j$ is the natural projection. Hence $\Ker(\tau) \leq O^p_{\FF}(S)$. The rest of the proof is the same as the proof of \cite[Theorem 2.5]{BCGLO2}, so we omit the details.
\end{proof}

Next we generalize the techniques of \cite{BCGLO2} to detect saturated fusion subsystems of a given fusion system. In order to reduce the exposition, we only consider subsystems of \textit{$p$-power index}, leaving the \textit{$p'$ index} case and its necessary generalizations to the interested reader. In general, there are no big differences between the finite case, studied in \cite{BCGLO2}, and the compact case studied here. Thus, we will omit most of the details, which are left to the reader, and only mention those steps that require some discussion.

\begin{defi}

Let $\FF$ be a saturated fusion system over a discrete $p$-toral group $S$, and let $\FF_0 \subseteq \FF$ be a saturated subsystem over a discrete $p$-toral subgroup $S_0 \leq S$. We say that $\FF_0$ is of \textit{$p$-power index} in $\FF$ if $S_0 \geq O^p_{\FF}(S)$ and $\Aut_{\FF_0}(P) \geq O^p(\Aut_{\FF}(P))$ for all $P \leq S_0$.

\end{defi}

We will also use the following notation from \cite{BCGLO2}. Given a discrete $p$-toral group $S$, a \textit{restrictive category} over $S$ is a category $\FF^{\ast}$ such that $\Ob(\FF^{\ast})$ is the set of all subgroups of $S$, morphisms in $\FF^{\ast}$ are group monomorphisms, and satisfying the following two properties:
\begin{enumerate}[(i)]

\item for each $P' \leq P \leq S$ and $Q' \leq Q \leq S$, and for each $f \in \Hom_{\FF^{\ast}}(P,Q)$ such that $f(P') \leq f(Q')$, the restriction $f|_{P'}$ is a morphism in $\Hom_{\FF^{\ast}}(P', Q')$; and

\item for each $P \leq S$, $\Aut_{\FF}(P)$ is an Artinian locally finite group.

\end{enumerate}

For a (discrete) group $\Gamma$, let $\mathfrak{Sub}(\Gamma)$ denote the set of nonempty subsets of $\Gamma$. Let $\FF$ be a saturate fusion system over a discrete $p$-toral group $S$. A subgroup $P \leq S$ is \textit{$\FF$-quasicentric} if, for all $Q \in P^{\FF}$, the centralizer fusion system $C_{\FF}(Q)$ is the fusion system of $C_S(Q)$. Given a subset $\hh \subseteq \Ob(\FF)$, we will denote by $\FF_{\hh} \subseteq \FF$ the full subcategory with object set $\hh$. For simplicity, $\FF^q \subseteq \FF$ will denote the full subcategory with object set the family of all $\FF$-quasicentric subgroups of $S$.

\begin{defi}\label{fmt}

Let $\FF$ be a saturated fusion system over a discrete $p$-toral group $S$, and let $\hh \subseteq \Ob(\FF)$ be a subset of $\FF$-quasicentric subgroups which is closed by conjugation in $\FF$. A \textit{fusion mapping triple} for $\FF_{\hh}$ consists of a triple $(\Gamma, \theta, \Theta)$, where $\Gamma$ is a discrete group, $\theta \colon S \to \Gamma$ is a homomorphism, and
$$
\Theta \colon \Mor(\FF_{\hh}) \Right4{} \mathfrak{Sub}(\Gamma)
$$
is a map satisfying the following properties for all subgroups $P, Q, R \in \hh$:
\begin{enumerate}[(i)]

\item for all $P \Right2{f} Q \Right2{f'} R$ in $\FF$ and all $x \in \Theta(f')$, $\Theta(f' \circ f) = x \cdot \Theta(f)$;

\item if $P$ is fully $\FF$-centralized, then $\Theta(\Id_P) = \theta(C_S(P))$;

\item if $f = c_g \in \Hom_S(P,Q)$, then $\theta(g) \in \Theta(f)$;

\item for all $f \in \Hom_{\FF}(P,Q)$, all $x \in \Theta(f)$ and all $g \in P$, $x \cdot \theta(g) \cdot x^{-1} = \theta(f(g))$;

\item $\Theta(\Id_P)$ is a subgroup of $\Gamma$, and $\Theta$ restricts to a homomorphism
$$
\Theta_P \colon \Aut_{\FF}(P) \Right4{} N_{\Gamma}(\Theta(\Id_P))/\Theta(\Id_P).
$$
In particular $\Theta_P(f) = \Theta(f)$ (as a coset of $\Theta(\Id_P)$) for all $f \in \Aut_{\FF}(P)$;

\item for all $P \Right2{f} Q \Right2{f'} R$ in $\FF_{\hh}$ and all $x \in \Theta(f)$, $\Theta(f' \circ f) \subseteq \Theta(f') \cdot x$, with equality if $f(P) = Q$. In particular $\Theta(f'|_P) \supseteq \Theta(f')$ if $P \leq Q$;

\item if $S \in \hh$, then for all $P \Right2{f} Q$ in $\FF$, all $\gamma \in \Aut_{\FF}(S)$ and all $x \in \Theta(\gamma)$, $\Theta(\gamma \circ f \circ \gamma^{-1}) = x \cdot \Theta(f) \cdot x^{-1}$.

\end{enumerate}

\end{defi}

The original definition of fusion mapping triple, \cite[Definition 3.6]{BCGLO2}, only included properties (i) to (iv) above, while properties (v) to (vii) where proved in \cite[Lemma 3.7]{BCGLO2}. Since this result holds as well in our situation, we have opted for a definition that includes all properties at once.

Let $\FF$ be a saturated fusion system over $S$, and let $(\Gamma, \theta, \Theta)$ be a fusion mapping triple for $\FF_{\hh}$, for some $\hh$. For a subgroup $H \leq \Gamma$, let $\FF_H^{\ast} \subseteq \FF$ be the smallest restrictive subcategory which contains all $f \in \Mor(\FF^q)$ such that $\Theta(f) \cap H \neq \emptyset$. Let also $\FF_H \subseteq \FF_H^{\ast}$ be the full subcategory whose objects are the subgroups of $\theta^{-1}(H)$.

\begin{prop}\label{fmt1}

Let $\FF$ be a saturated fusion system over a discrete $p$-toral group $S$, and let $(\Gamma, \theta, \Theta)$ be a fusion mapping triple on $\FF^q$, where $\Gamma$ is a finite $p$-group. Then the following holds for all $H \leq \Gamma$.
\begin{enumerate}[(i)]

\item $\FF_H$ is a saturated fusion system over $S_H = \theta^{-1}(H)$.

\item A subgroup $P \leq S_H$ is $\FF_H$-quasicentric if and only if it is $\FF$-quasicentric.

\end{enumerate}

\end{prop}

\begin{proof}

The proof is identical to the proof of \cite[Proposition 3.8]{BCGLO2}, except that we have to check that $\FF_H$ satisfies axiom (III) of saturated fusion systems. Let then $P_1 \leq P_2 \leq \ldots$ be an ascending family of subgroups of $S_H$, and let $P = \bigcup_n P_n$. Let also $f \in \Hom(P,S)$ be such that $f|_{P_n} \in \Hom_{\FF_H}(P_n,S)$ for all $n$. We have to show that $f \in \Hom_{\FF_H}(P,S)$.

Since $\FF$ is saturated, it follows that $f \in \Hom_{\FF}(P,S)$. Using property (v) of fusion mapping triples, we see that $\Theta(f_{n+1}) \subseteq \Theta(f_n)$ for all $n$. Since $\Gamma$ is finite by assumption, it follows that there exists some number $M$ such that, for all $n \geq M$,
$$
\Theta(f_n) \cap H = \Theta(f_{n+1}) \cap H \neq \emptyset.
$$
It follows then that $\Theta(f) \cap H \neq \emptyset$, and $f \in \Hom_{\FF_H}(P,S)$.
\end{proof}

In particular, we want to construct a fusion mapping triple with the obvious projection $\theta \colon S \Right2{} S/O^p_{\FF}(S)$. The following is an induction tool to construct such fusion mapping triple, and it corresponds to \cite[Lemma 4.1]{BCGLO2}.

\begin{lmm}\label{fmt2}

Let $\FF$ be a saturated fusion system over a discrete $p$-toral group $S$. Let also $\hh_0$ be a set of $\FF$-quasicentric subgroups of $S$ which is closed by $\FF$-conjugacy and overgroups. Let also $\pp$ be an $\FF$-conjugacy class of $\FF$-quasicentric subgroups, maximal among those not contained in $\hh_0$, and set $\hh = \hh_0 \bigcup \pp$. Fix a group $\Gamma$ and a homomorphism $\theta \colon S \to \Gamma$, and let
$$
\Theta \colon \Mor(\FF_{\hh_0}) \Right4{} \mathfrak{Sub}(\Gamma)
$$
be such that $(\Gamma, \theta, \Theta)$ is a fusion mapping triple for $\FF_{\hh_0}$.

Let $P \in \pp$ be fully $\FF$-normalized, and fix a homomorphism
$$
\Theta_P \colon \Aut_{\FF}(P) \Right4{} N_{\Gamma}(\theta(C_S(P)))/\theta(C_S(P))
$$
such that the following conditions hold:
\begin{enumerate}[(a)]

\item $x \cdot \theta(f) \cdot x^{-1} = \theta(f(g))$ for all $g \in P$, $f \in \Aut_{\FF}(P)$ and $x \in \Theta_P(f)$; and

\item $\Theta_P(f) \supseteq \Theta(f')$ for all $P \lneqq Q \leq S$ such that $P \lhd Q$ and $Q$ is fully $\FF$-normalized, and for all $f \in \Aut_{\FF}(P)$ and $f' \in \Aut_{\FF}(Q)$ such that $f = f'|_P$.

\end{enumerate}
Then, there exists a unique extension of $\Theta$ to a fusion mapping triple $(\Gamma, \theta, \widetilde{\Theta})$ on $\FF_{\hh}$ such that $\widetilde{\Theta}(f) = \Theta_P(f)$ for all $f \in \Aut_{\FF}(P)$.

\end{lmm}

\begin{rmk}

Let $\FF^{\bullet} \subseteq \FF$ be the retract of $\FF$ described in \cite[Section \S 3]{BLO3}, whose object set contains finitely many conjugacy classes. The above statement is still valid if one considers $\pp$ to be an $\FF$-conjugacy class of $\FF$-quasicentric subgroups, maximal among those not contained in $\hh_0$, \textit{and such that $\pp \subseteq \Ob(\FF^{\bullet})$}. Notice that this way the set $\hh$ may not be closed by overgroups. Nevertheless, once the fusion mapping triple has been extended to $\hh$, it can be formally extended to close it by overgroups, using the the retraction $\FF \to \FF^{\bullet}$.

\end{rmk}

\begin{prop}\label{fmt3}

Let $\FF$ be a saturated fusion system over a discrete $p$-toral group $S$, and let $\theta \colon S \Right2{} S/O^p_{\FF}(S)$ be the natural projection. Then there is a fusion mapping triple $(S/O^p_{\FF}(S), \theta, \Theta)$ for $\FF^q$.

\end{prop}

\begin{proof}

For simplicity, set $\Gamma = S/O^p_{\FF}(S)$ and $S_0 = \Ker(\theta)$. Let $\hh_0 \subseteq \Ob(\FF^q)$ be a (possibly empty) subset which is closed by $\FF$-conjugation and overgroups, and let $\pp$ be an $\FF$-conjugacy class in $\Ob(\FF^q) \cap \Ob(\FF^{\bullet})$ which is maximal among those not contained in $\hh_0$. Set $\hh = \hh_0 \bigcup P$. Assume that a fusion mapping triple $(\Gamma, \theta, \Theta_0)$ has already been constructed for $\FF_{\hh_0}$.

For an Artinian locally finite group $G$ with Sylow $p$-subgroup $S_G$, recall the notation
$$
O^p_G(S_G) = \gen{T_G, \{[g,x] \,\, | \,\, g \in P \leq S_G, x \in N_G(P) \mbox{ of order prime to } p\}},
$$
where $T_G$ is the maximal torus of $S_G$. By Lemma \ref{hyper0}, $O^p_G(S_G) = S_G \cap O^p(G)$, and thus there is an isomorphism $G/O^p(G) \cong S_G/O^p_G(S_G)$.

Fix $P \in \pp$ such that it is fully $\FF$-normalized, and let $N_0$ be the subgroup generated by the commutators $[g,x]$, with $g \in N_S(P)$ and $x \in N_{\Aut_{\FF}(P)}(N_S(P))$ of order prime to $p$, together with the maximal torus of $N_S(P)$. In this situation, $\Aut_S(P) \in \Syl_p(\Aut_{\FF}(P))$ because $P$ is fully $\FF$-normalized, and $\Aut_{N_0}(P) = O^p_{\Aut_{\FF}(P)}(\Aut_S(P))$. By Lemma \ref{hyper0},
$$
\Aut_{\FF}(P)/O^p(\Aut_{\FF}(P)) \cong \Aut_S(P)/\Aut_{N_0}(P) \cong N_S(P)/\gen{N_0, C_S(P)}.
$$
Also, $N_0 \leq O^p_{\FF}(S)$, and so the inclusion of $N_S(P)$ in $S$ induces a homomorphism
$$
\Theta_P \colon \Aut_{\FF}(P) \Right3{} N_S(P)/\gen{N_0, C_S(P)} \Right3{} N_S(C_S(P) \cdot S_0)/C_S(P) \cdot S_0.
$$
Condition (a) in Lemma \ref{fmt2} follows by construction of $\Theta_P$.

To prove that condition (b) also holds, let $P \lneqq Q \leq N_S(P)$ be a subgroup which is fully normalized in $N_{\FF}(P)$, and let $\alpha \in \Aut_{\FF}(P)$ and $\beta \in \Aut_{\FF}(Q)$ be such that $\alpha = \beta|_P$. We have to check that $\Theta_P(\alpha) \supseteq \Theta_0(\beta)$.

Taking the $k$-th power of both $\alpha$ and $\beta$ for some appropriate $k$ congruent to $1$ modulo $p$, we may assume that both morphisms have $p$-power order. Now, since $Q$ is fully $N_{\FF}(P)$-normalized (and this is a saturated fusion system), it follows that $\Aut_{N_S(P)}(Q) \in \Syl_p(\Aut_{N_{\FF}(P)}(Q))$, and thus there are automorphisms $f \in \Aut_{\FF}(Q)$ and $f' = f|_P \in \Aut_{\FF}(P)$ such that
$$
f \circ \beta \circ f^{-1} = (c_g)|_Q
$$
for some $g \in N_S(Q) \cap N_S(P)$. Thus, $(f') \alpha(f')^{-1} = (c_g)|_P$. It follows then that
$$
\Theta_0(\beta) = \Theta((c_g)|_Q) = g \cdot \theta(C_S(Q)) \subseteq g \cdot \theta(C_S(P)) = \Theta_P((c_g)|_P) = \Theta_P(\alpha).
$$

By Lemma \ref{fmt2} the map $\Theta_0$ extends uniquely to a fusion mapping triple $(\Gamma, \theta, \Theta)$ for $\FF_{\hh}$, and thus also to a fusion mapping triple for the closure of $\FF_{\hh}$ by overgroups. Since $\FF^{\bullet}$ contains finitely many $\FF$-conjugacy classes, after a finite number of steps we obtain a fusion mapping triple for $\FF^q$.
\end{proof}

\begin{thm}\label{fmt4}

Let $\FF$ be a saturated fusion system over a discrete $p$-toral group $S$, with hyperfocal subgroup $O^p_{\FF}(S)$, and let $\Gamma = S/O^p_{\FF}(S)$. Let also $(\Gamma, \theta, \Theta)$ be the fusion mapping triple constructed in Proposition \ref{fmt3}. Then, for each $R \leq S$ containing $O^p_{\FF}(S)$ there is a unique saturated fusion system $\FF_R \subseteq \FF$ over $R$ of $p$-power index. Furthermore, $\FF_R$ satisfies the following properties:
\begin{enumerate}[(i)]

\item a subgroup $P \leq R$ is $\FF_R$-quasicentric if and only if it is $\FF$-quasicentric; and

\item for each $P, Q \in \Ob(\FF_R^q)$, $\Hom_{\FF_R}(P,Q) = \{f \in \Hom_{\FF}(P,Q) \,\, | \,\ \Theta(f) \cap (R/O^p_{\FF}(S)) \neq \emptyset\}$.

\end{enumerate}

\end{thm}

\begin{proof}

Let $\FF_R \subseteq \FF$ be the fusion subsystem defined in Proposition \ref{fmt1}. Saturation, together with properties (i) and (ii) follow from Proposition \ref{fmt1}, and it remains to prove uniqueness of $\FF_R$. Let $\FF_R' \subseteq \FF$ be another saturated subsystem over $R$, of $p$-power index. By hypothesis, for each $P \leq R$ which is fully $\FF$-normalized, we have
$$
\Aut_{\FF_R} = \gen{O^p(\Aut_{\FF}(P)), \Aut_R(P)} = \Aut_{\FF_R'}(P).
$$
Uniqueness follows easily from this observation.
\end{proof}

Let $\FF$ be a saturated fusion system over a discrete $p$-toral group $S$, and let $S_0 = O^p_{\FF}(S)$. Let also $\FF_0 \subseteq \FF$ be the unique saturated fusion subsystem of $p$-power index over $S_0$.

\begin{cor}\label{fmt5}

The fusion subsystem $\FF_0$ is normal in $\FF$.

\end{cor}

\begin{proof}

Clearly $S_0$ is strongly $\FF$-closed, and $\FF_0$ is saturated by Theorem \ref{fmt4}. By property (ii) in Theorem \ref{fmt4}, if $P,Q \in \Ob(\FF^q_0)$ then
$$
\Hom_{\FF_0}(P,Q) = \{f \in \Hom_{\FF}(P,Q) \,\, | \,\, \Theta(f) \cap \{1\} \neq \emptyset\},
$$
and hence condition (N2) of normal subsystems holds.

Finally we check that property (N4) holds too. Let then $\widetilde{S}_0 = S \cdot C_S(S_0)$, and notice that both $S_0$ and $\widetilde{S}_0$ are fully $\FF$-centralized, since $S_0$ is strongly $\FF$-closed. In particular, $\Aut_{\FF_0}(S_0)$ is generated by $\Inn(S_0)$ and $O^p(\Aut_{\FF_0}(S_0))$.

Consider also the morphism induced by property (v) of fusion mapping triples,
$$
\Theta_0 \colon \Aut_{\FF}(\widetilde{S}_0) \Right4{} N_{\Gamma}(\theta(Z(\widetilde{S}_0)))/\theta(Z(\widetilde{S}_0)).
$$
Now, let $f \in \Aut_{\FF_0}(S_0)$. If $f \in \Inn(S_0)$ then $f = c_x$ for some $x \in S_0$, and $\gamma = c_x \in \Aut_{\FF}(\widetilde{S}_0)$ clearly satisfies $\gamma(g) \cdot g^{-1} = 1$ for all $g \in C_S(S_0)$. If $f \in O^p(\Aut_{\FF}(S_0))$ then we may find some $\gamma \in O^p(\Aut_{\FF}(\widetilde{S}_0))$ such that $\gamma|_{S_0} = f$. Since $O^p(\Aut_{\FF}(\widetilde{S}_0)) \leq \Ker(\Theta_0)$, it follows by properties (ii) and (v) of fusion mapping triples that $\Theta_0(\gamma) = \Theta_0(\Id_{\widetilde{S}_0}) = \theta(Z(\widetilde{S}_0))$. In particular $1 \in \Theta_0(\gamma)$, and thus by property (iv) $\gamma(g) g^{-1} \in Z(S_0) = Z(\widetilde{S}_0) \cap S_0$.
\end{proof}

\bibliographystyle{alpha}
\bibliography{/Users/agondem/Dropbox/MATES/Tex-extras/Main}

\end{document}